%% file: DYg_quantization.tex
\documentclass{amsart}
\usepackage[foot]{amsaddr}

\title{The restricted quantum double of the Yangian}
\author[C. Wendlandt]{Curtis Wendlandt}
\address{Department of Mathematics and Statistics, University of Saskatchewan.}
\email{wendlandt@math.usask.ca}


\usepackage{lmodern}

\input{Preamble}

\usepackage{upgreek}

\allowdisplaybreaks

\numberwithin{equation}{section}

\newcommand{\N}{\mathbb{N}}

\newcommand{\Yhg}{Y_\hbar(\mfg)}
\newcommand{\hYhg}{Y_\hbar\mfg}
\newcommand{\cYhg}{\widehat{Y_\hbar\mfg}}
\newcommand{\Yhgz}{\cYhg_z}
\newcommand{\LzYhg}{\mathds{L}\cYhg_z}
\newcommand{\LzhYhg}{\mathrm{L}\cYhg_z}

\newcommand{\dYh}[1]{\dot{\mathsf{Y}}_\hbar^{#1}(\mfg)}
\newcommand{\dhYh}[1]{\dot{\mathsf{Y}}_\hbar^{#1}\mfg}
\newcommand{\QFYhg}{\dot{\mathsf{Y}}_\hbar(\mfg)}
\newcommand{\QFhYhg}{\dot{\mathsf{Y}}_\hbar\mfg}
\newcommand{\QFJYhg}{\hYhg^\prime}
\newcommand{\Yhgstar}{\QFhYhg{\vphantom{)}}^\star}
\newcommand{\YTDstar}[1]{\dot{\msY}_\hbar^{#1}\mfg{\vphantom{)}}^\star}

\newcommand{\DYhg}{\mathrm{D}Y_\hbar\mfg}
\newcommand{\DYg}{\mathds{D}Y_{\hbar}\mfg}
\newcommand{\cDYhg}{\widehat{\mathrm{D}Y_\hbar\mfg}}

\newcommand{\tplus}{\mft_{\scriptscriptstyle{+}}}
\newcommand{\tminus}{\mft_{\scriptscriptstyle{-}}}
\newcommand{\Symt}{\mathsf{S}(\hbar\mft_{\scriptscriptstyle{+}})}
\newcommand{\hSymt}[1]{\mathsf{S}_{#1}(\hbar\mft_{\scriptscriptstyle{+}})}
\newcommand{\sSymt}[1]{\mathsf{S}^{#1}(\hbar\mft_{\scriptscriptstyle{+}})}
\newcommand{\Symtstar}{\Symt{\vphantom{)}}^\star}
\newcommand{\Root}{\Delta}
\newcommand{\id}{\mathbf{1}} 

\newcommand{\rcoad}{\blacktriangle} 
\newcommand{\lad}{\blacktriangledown} 

\usepackage[T1]{fontenc}

\usepackage{graphicx}
\def\chk#1{#1^{\smash{\scalebox{.7}[1.4]{\rotatebox{90}{\textnormal{\guilsinglleft}}}}}}

\newcommand{\lAngle}{\langle\hspace{-.27em}\langle}
\newcommand{\rAngle}{\rangle\hspace{-.27em}\rangle}
\newcommand{\op}[1]{{#1}{\vphantom{)}}^{\!\scriptscriptstyle{\mathrm{op}}}}
\DeclareMathSymbol{\shortminus}{\mathbin}{AMSa}{"39}

\DeclareMathAlphabet\EuScript{U}{eus}{m}{n}
\SetMathAlphabet\EuScript{bold}{U}{eus}{b}{n}
\newcommand{\boldR}{{\boldsymbol{\EuScript{R}}}}
\newcommand{\scriptE}{\EuScript{E}}

\begin{document}

\begin{abstract}
Let $\mfg$ be a complex semisimple Lie algebra with associated Yangian $\hYhg$. In the mid-1990s, Khoroshkin and Tolstoy formulated a conjecture which asserts that the algebra $\DYhg$ obtained by doubling the generators of $\hYhg$, called the Yangian double, provides a realization of the quantum double of the Yangian. We provide a uniform proof of this conjecture over $\C[\![\hbar]\!]$ which is compatible with the theory of quantized enveloping algebras. As a byproduct, we identify the universal $R$-matrix of the Yangian with the canonical element defined by the pairing between the Yangian and its restricted dual. 
\end{abstract}

\maketitle

{\setlength{\parskip}{0pt}
\setcounter{tocdepth}{1} 
\tableofcontents
}

\section{Introduction}\label{sec:Intro}

%
\subsection{}\label{ssec:I-background} 
This article is a continuation of \cite{WDYhg}, which studied the Yangian double $\DYhg$ associated to an arbitrary symmetrizable Kac--Moody algebra $\mfg$ through the lens of a $\Z$-graded algebra homomorphism 
\begin{equation*}
\Phi_z:\DYhg\to \LzhYhg \subset \hYhg[\![z^{\pm 1}]\!].
\end{equation*}
Here $\LzhYhg$ is a naturally defined $\Z$-graded $\C[\![\hbar]\!]$-algebra, described explicitly in Lemma \ref{L:DYhg-LzYhg}, and  $\hYhg$ is the Yangian of $\mfg$, defined over $\C[\![\hbar]\!]$. This homomorphism, called the \textit{formal shift operator}, naturally extends the so-called shift homomorphism $\tau_z$ on the Yangian, and has a number of remarkable properties. For instance, it induces a family of isomorphisms between completions of $\DYhg$ and $\hYhg$, realizes $\hYhg$ as a degeneration of $\DYhg$, and is injective provided $\mfg$ is of finite type or of simply laced affine type. In addition, it was applied in \cite{GWPoles} to characterize the category of finite-dimensional representations of $\DYhg$, for $\hbar\in \C^\times$ and $\mfg$ of finite type, as the tensor-closed Serre subcategory of that of the Yangian consisting of those representations which have no \textit{poles} at zero. 

In this article, we narrow our focus to the case where $\mfg$ is a finite-dimensional simple Lie algebra, and apply these results in conjunction with those of the recent paper \cite{GTLW19} to prove one of the main conjectures from the work \cite{KT96} of Khoroshkin and Tolstoy. Namely, we establish that $\DYhg$, which is defined by doubling the generators of $\hYhg$ (see Definition \ref{D:DYhg}), is isomorphic to the \textit{restricted} quantum double of the Yangian $\hYhg$, where the prefix ``restricted'' indicates that all duality operations are taken so as to respect the underlying gradings. 
As a consequence of this result and its proof, we find that $\Phi_z$ identifies the universal $R$-matrix of $\DYhg$, which arises from the quantum double construction, with Drinfeld's universal $R$-matrix $\mcR(z)\in (\hYhg\otimes \hYhg)[\![z^{-1}]\!]$. Our argument makes essential use of the constructive  proof of the existence of $\mcR(z)$ given in \cite{GTLW19}, which is independent from Drinfeld's cohomological construction of $\mcR(z)$ from the foundational paper \cite{Dr}.

%
\subsection{Main results}\label{ssec:I-results} Let us now sketch our main results in detail. The two results alluded to above form Parts \eqref{Intro:1} and \eqref{Intro:2} of the following theorem.
\begin{theoremintro}\label{T:Intro}
There is a unique $\Z$-graded Hopf algebra structure on $\DYhg$ over $\C[\![\hbar]\!]$ such that the formal shift operator 
\begin{equation*}
\Phi_z:\DYhg\into\LzhYhg\subset \hYhg[\![z^{\pm 1}]\!] 
\end{equation*}
intertwines the Hopf structures on $\DYhg$ and $\hYhg$. Moreover: 
\begin{enumerate}[font=\upshape]
\item\label{Intro:1} $\DYhg$ is isomorphic, as a $\Z$-graded Hopf algebra, to the restricted quantum double of the Yangian $\hYhg$. 
\item\label{Intro:2} Under the above identifications, the universal $R$-matrix $\boldR$ of $\DYhg$ satisfies
\begin{equation*}
(\Phi_w\otimes \Phi_z)(\boldR)=\mcR(w-z)\in \hYhg^{\otimes 2}[w][\![z^{-1}]\!]
\end{equation*}
\end{enumerate}
\end{theoremintro}
This is a combination of the three main results of this article: Theorems \ref{T:dual}, \ref{T:DYhg-DD} and \ref{T:R}. Part \eqref{Intro:1} is the statement of our second main result ---  Theorem \ref{T:DYhg-DD} --- and is precisely the variant of the conjecture from \cite{KT96}*{\S2} which we establish in the present paper. Our approach to proving it is, in a certain sense, dual to the strategy outlined in \cite{KT96} which was brought to fruition for $\mfg=\mfsl_2$. In more detail, our argument hinges on the fact, proven in Proposition \ref{P:dual-z}, that the universal $R$-matrix $\mcR(z)$ of the Yangian gives rise to a $\C[\![\hbar]\!]$-algebra homomorphism 
\begin{equation*}
\chk{\Phi_z}: \Yhgstar\to \hYhg[\![z^{-1}]\!]
\end{equation*}
which is compatible with the Hopf algebra structure on $\hYhg$ and the co-opposite Hopf structure on the dual $\Yhgstar$ of the Yangian $\hYhg$ taken in the category of $\Z$-graded quantized enveloping algebras. That is, $\Yhgstar$ is the restricted (or graded) dual of the Drinfeld--Gavarini \cites{DrQG,Gav02}  subalgebra $\QFhYhg\subset \hYhg$, defined in Section \ref{sec:QFYhg}, and provides a homogeneous quantization of the restricted dual $t^{-1}\mfg[t^{-1}]$ to the $\N$-graded Lie bialgebra $\mfg[t]$, as we prove in detail in Section \ref{sec:Yhg*}; see Theorem \ref{T:Yhg*-quant}. 

Using the construction of $\mcR(z)$ given in \cite{GTLW19} and properties of $\Phi_z$ established in \cite{WDYhg}, we deduce that the image of $\Yhgstar$ under $\chk{\Phi_z}$ is contained in the image of $\Phi_z$. We may thus compose $\chk{\Phi_z}$ with $\Phi_z^{-1}$ to obtain a $\C[\![\hbar]\!]$-algebra homomorphism 
\begin{equation*}
\check{\imath}:=\Phi_z^{-1}\circ \chk{\Phi_z}:\chk{\hYhg}\to \DYhg,
\end{equation*}
where $\chk{\hYhg}:=(\Yhgstar){\vphantom{)}}^{\scriptscriptstyle\mathrm{cop}}$.
In our first main result --- Theorem \ref{T:dual} ---  we show that there is a unique $\Z$-graded Hopf algebra structure on $\DYhg$ for which both $\check{\imath}$ and the natural inclusion $\imath:\hYhg\to \DYhg$
are injective homomorphisms of graded Hopf algebras. This is exactly the Hopf structure alluded to in the statement of Theorem \ref{T:Intro}, and is such that $\DYhg$ provides a homogeneous quantization of the restricted Drinfeld double $\mfg[t^{\pm 1}]$ of $\mfg[t]$. Using Theorem \ref{T:dual}, it is then not difficult to establish Part \eqref{Intro:1} above (\textit{i.e.,} Theorem \ref{T:DYhg-DD}) using the double cross product realization of the restricted quantum double (see Section \ref{ssec:D(Yhg)}).

Our third and final main result, Theorem \ref{T:R}, is a strengthening of Part \eqref{Intro:2} above. Indeed, it outputs the Gauss decomposition for $\boldR$ while identifying each  factor appearing in this decomposition with the factors $\mcR^\pm(z)$ and $\mcR^0(z)$ of $\mcR(z)$, which were studied in detail in \cites{GTLW19, GTL3}.

%
\subsection{Motivation}\label{ssec:I-motivation} Part \eqref{Intro:2} of Theorem \ref{T:Intro} implies that $\mcR(z)$ can be recovered from the canonical element defined by the pairing between $\hYhg$ and its restricted dual $\hYhg{\vphantom{)}}^\star\subset \DYhg$ by applying the injection $\Phi_{\shortminus{z}}$ to its second tensor factor:
\begin{equation*}
(\id\otimes \Phi_{\shortminus{z}})(\boldR)=\mcR(z)\in \hYhg^{\otimes 2}[\![z^{-1}]\!].
\end{equation*}
  Here we refer the reader to Theorem \ref{T:R} for further details, which takes into account the topological subtleties surrounding this statement. 
Obtaining this interpretation of $\mcR(z)$ is in fact our original motivation for addressing the conjecture of \cite{KT96}*{\S2}, and brings the theory surrounding the universal $R$-matrix of the Yangian to a more equal footing with that of the (extended, untwisted) quantum affine algebra $U_q(\wh{\mfg})$ and the quantum loop algebra $U_q(L\mfg)$. The differences are, however, still quite pronounced. Indeed, $U_q(\wh{\mfg})$ is itself nearly the quantum double of its (quantum Kac--Moody) Borel subalgebra $U_q(\mfb^+)$, and its universal $R$-matrix $\mathscr{R}$ lies in a completion of $U_q(\mfb^+)\otimes U_q(\mfb^-)$. One then recovers the universal $R$-matrix $\mathscr{R}_{L\mfg}$ of $U_q(L\mfg)$ as a truncation of $\mathscr{R}$, and its $z$-dependent analogue is
\begin{equation*}
\mathscr{R}_{L\mfg}(z):=(\id\otimes D_{z^{-1}})(\mathscr{R}_{L\mfg})\in U_q(L\mfg)^{\otimes 2}[\![z]\!],
\end{equation*}
where $D_z$ is given by the $\Z$-grading on $U_q(L\mfg)$; see \cite{EFK-Book}*{\S9.4} or \cite{FR92}*{\S4}, for instance. Crucially, $\mathscr{R}$ can be constructed by computing dual bases with respect to the pairing between $U_q(\mfb^+)$ and $U_q(\mfb^-)$, and was done explicitly by Damiani in \cite{Damiani98}. In contrast, the Yangian $\hYhg$ is not of Kac--Moody type and does not arise as a Hopf algebra from the quantum double construction applied to any analogue of $U_q(\mfb^+)$\footnote{We refer the reader to \cite{YaGu3}*{\S4} for a related construction of $\hYhg$ with respect to its deformed Drinfeld coproduct, which does not endow $\hYhg$ with the structure of a Hopf algebra.}. In addition, $\mcR(z)$ and $\mathscr{R}_{L\mfg}(z)$ exhibit significantly different analytic behaviour when evaluated on finite-dimensional representations \cite{GTLW19}. 
Nonetheless, the results of this article further cement that there are very strong parallels to be drawn between the two pictures. Indeed, one obtains the $\hYhg$-analogue of the above story by replacing $U_q(\mfb^+)$ by $\hYhg$, $U_q(L\mfg)$ by the Yangian double $\DYhg$, and $D_z$ by the formal shift operator $\Phi_z$. 

It should be noted that it appears that this realization of $\mcR(z)$ has been anticipated for some time in the mathematical physics community; see for instance \cite{Stukopin07}*{\S5}, which considers its super-analogue. The direction taken therein is, however, based on both the conjecture from \cite{KT96}*{\S2} at the heart of the present article, and on the  infinite product formulas for the factors $\boldR^\pm$ of $\boldR$ given \cite{KT96}*{\S5}, which remain conjectural. Some more discussion on this point is given in Section \ref{ssec:R-compute}. 

%
\subsection{Remarks}\label{ssec:I-remarks}

Let us now give a few brief remarks.  Firstly, it is essential that the Yangian double $\DYhg$ is defined as a topological algebra over $\C[\![\hbar]\!]$ for the above results to hold true. 
To expand on this, $\DYhg$ can be realized as the $\hbar$-adic completion of a $\Z$-graded $\C[\hbar]$-algebra $\DYg^\jmath$ defined by generators and relations; see Remark \ref{R:DYg-j}. One can further specialize $\hbar$ to any nonzero complex number $\zeta$ to obtain a $\C$-algebra $\mathds{D}Y_\zeta\mfg=\DYg^\jmath/(\hbar-\zeta)\DYg^\jmath$, whose category of finite-dimensional representations was characterized in terms of that of the corresponding Yangian $Y_\zeta(\mfg)$ in \cite{GWPoles}. Though this category has a tensor structure which corresponds to the Hopf structure on $\DYhg$, it is important to note that $\mathds{D}Y_\zeta\mfg$ is \textit{not} a Hopf algebra over $\C$, and in particular it does not coincide with the (restricted) quantum double of $Y_\zeta(\mfg)$ defined in any reasonable sense. 

That being said, $\mathds{D}Y_\zeta\mfg$ admits a natural $\Z$-filtration corresponding to the $\Z$-grading on $\DYhg$, and the expectation is that the formal completion of $\mathds{D}Y_\zeta\mfg$ with respect to this filtration coincides with the (restricted) quantum double of $Y_\zeta(\mfg)$ taken in the appropriate category of $\Z$-filtered, complete topological Hopf algebras. This is in fact the version of Part \eqref{Intro:2} of Theorem \ref{T:Intro} conjectured in \cite{KT96}, and is consistent with the situation that transpires in type A for the $R$-matrix realization of the Yangian, which has been developed in great detail in the recent paper \cite{Naz20}. For our purposes, it is more natural to work over $\C[\![\hbar]\!]$ within the framework of quantized enveloping algebras first developed by Drinfeld \cite{DrQG}, where we may study $\DYhg$ from the point of view of quantization of Lie bialgebras. At the same time, many of our results are ``global'' (in the sense of \cite{Gav07}) and admit an interpretation over both $\C[\hbar]$ and $\C$, including the realization of $\mcR(z)$ provided by Theorem \ref{T:R}; see Appendix \ref{A:R-matrix}.

%
\subsection{Outline}\label{ssec:I-outline} 
The paper is written so as to provide a complete picture, accessible to non-experts, where possible. For this reason, we take great care to lay the foundation needed to state and prove the results outlined in Section \ref{ssec:I-results}. The first three sections --- Sections  \ref{sec:Pr}, \ref{sec:Yhg} and \ref{sec:DYhg} -- are intended to serve a preliminary role: Section \ref{sec:Pr} surveys the theory of $\Z$-graded topological $\C[\![\hbar]\!]$-modules, algebras and Hopf algebras, including homogeneous quantizations of graded Lie bialgebras. This theory plays a prominent role throughout the article. In
Section \ref{sec:Yhg}, we review the definition and main properties of the Yangian $\hYhg$, defined both over $\C[\hbar]$ and $\C[\![\hbar]\!]$. Notably, this includes a review of the construction of the universal $R$-matrix $\mcR(z)$ carried out in \cite{GTLW19}. Section \ref{sec:DYhg} is focused on the Yangian double $\DYhg$ and, in particular, on reviewing the main results of \cite{WDYhg}; see Theorems \ref{T:Phiz} and \ref{T:Phi}.

In Sections \ref{sec:QFYhg} and \ref{sec:Yhg*}, we study the Drinfeld--Gavarini subalgebra $\QFhYhg$ of the Yangian, its $\C[\hbar]$-form, and its restricted dual $\Yhgstar$ in detail. This includes a detailed proof that $\Yhgstar$ provides a homogeneous quantization of $t^{-1}\mfg[t^{-1}]$, equipped with the Yangian Lie bialgebra structure; see Definition \ref{D:Yhg*} and Theorem \ref{T:Yhg*-quant}.

The last three sections of the article contain its three main results: Theorems \ref{T:dual}, \ref{T:DYhg-DD} and \ref{T:R}. We refer the reader to Section \ref{ssec:I-results} above, where these are outlined in detail. Finally, in Appendix \ref{A:R-matrix} we explain how to translate the construction of the universal $R$-matrix given in \cite{GTLW19} for $\hbar\in \C^\times$ to the setting of the present paper, in which $\hbar$ is a formal variable; see Proposition \ref{P:Yhg-R}, which appears in Section \ref{ssec:Yhg-R} as Theorem \ref{T:Yhg-R}.


\subsection{Acknowledgments}
I would like to thank Alex Weekes, Andrea Appel, Sachin Gautam, and Valerio Toledano Laredo for the many insightful discussions and helpful comments they have been the source of over the last few years. These have played a significant role in shaping this article.

\section{Homogeneous quantizations}\label{sec:Pr}

%
\subsection{Topological modules}\label{ssec:Pr-Top}

Recall that a $\C[\![\hbar]\!]$-module $\msM$ is \textit{separated} if the intersection of the family of submodules $\hbar^n\msM$ is trivial, and it is \textit{complete} if the natural $\C[\![\hbar]\!]$-linear map 
\begin{equation*}
\msM\to \varprojlim_n( \msM/\hbar^n \msM)
\end{equation*}
is surjective, where the inverse limit is taken over the set $\N$ of non-negative integers. In particular, $\msM$ is both separated and complete if and only if the above map is an isomorphism. If $\msM$ is separated, complete and torsion free as a $\C[\![\hbar]\!]$-module, then it is said to be \textit{topologically free}. This is equivalent to the existence of a $\C[\![\hbar]\!]$-module isomorphism $\msM\cong \msV[\![\hbar]\!]$ for a complex vector space $\msV$. Such an isomorphism is specified by a choice of complement $\msV\subset \msM$ to $\hbar\msM$:
\begin{equation*}
\msM= \msV \oplus \hbar \msM. 
\end{equation*}
More generally, if $\msM$ is any $\C[\![\hbar]\!]$-module, then the space $\msV=\msM/\hbar\msM$ is called the \textit{semiclassical limit} of $\msM$. Similarly, the semiclassical limit of a $\C[\![\hbar]\!]$-linear map $\uptau: \msM\to \msN$ is the $\C$-linear map $\bar\uptau: \msM/\hbar\msM\to \msN/\hbar\msN$ uniquely determined by the commutativity of the diagram 
\begin{equation*}
\begin{tikzcd}[column sep=13ex]
 \msM \arrow[two heads]{d}  \arrow{r}{\uptau} &  \msN \arrow[two heads]{d}\\
 \msM/\hbar\msM \arrow{r}{\bar\uptau}       & \msN/\hbar\msN
\end{tikzcd}
\end{equation*}
As the following elementary result illustrates, the semiclassical limit of a $\C[\![\hbar]\!]$-module homomorphism encodes important information about the original map. 
\begin{lemma}\label{L:scl-map}
Let $\msM$, $\msN$, $\uptau$ and $\bar\uptau$ be as above. 
\begin{enumerate}[font=\upshape]
\item Suppose that $\msM$ is separated, $\msN$ is torsion free and $\bar\uptau$ is injective. Then $\uptau$ is injective. 
\item Suppose that $\msM$ is complete, $\msN$ is separated and $\bar{\uptau}$ is surjective. Then $\uptau$ is surjective. 
\end{enumerate}
\end{lemma}
The topological tensor product $\msM \,{\widehat{\otimes}}\, \msN$ of two $\C[\![\hbar]\!]$-modules $\msM$ and $\msN$ is the $\hbar$-adic completion of the algebraic tensor product $\msM\otimes_{\C[\![\hbar]\!]}\msN$: 
\begin{equation*}
\msM \,{\widehat{\otimes}}\, \msN= \varprojlim_{n}(\msM\otimes_{\C[\![\hbar]\!]}\msN) /\hbar^n (\msM\otimes_{\C[\![\hbar]\!]}\msN).
\end{equation*}
If $\msM$ and $\msN$ are topologically free with $\msM\cong \msV[\![\hbar]\!]$ and $\msN\cong \msW[\![\hbar]\!]$, then $\msM \,{\widehat{\otimes}}\, \msN$ is topologically free and isomorphic to $(\msV\otimes_\C \msW)[\![\hbar]\!]$. 

In this article, we shall say that $\msM$ is a \textit{topological module} over $\C[\![\hbar]\!]$ if $\msM$ is a $\C[\![\hbar]\!]$-module which is both separated and complete. For any such module, we have 
\begin{equation*}
\C[\![\hbar]\!]\,\wh{\otimes}\, \msM \cong \msM \cong \msM\,\wh{\otimes}\, \C[\![\hbar]\!].
\end{equation*}
Similarly, by a \textit{topological algebra} $\msA$ over $\C[\![\hbar]\!]$ we shall always mean that $\msA$ is a $\C[\![\hbar]\!]$-algebra which is both separated and complete as a module over $\C[\![\hbar]\!]$. In particular, the multiplication $m$ can be viewed as a $\C[\![\hbar]\!]$-linear map 
\begin{equation*}
m: \msA\,{\widehat{\otimes}}\,\msA\to \msA.
\end{equation*}
A \textit{topological Hopf algebra} $\msH$ over $\C[\![\hbar]\!]$ is a topological $\C[\![\hbar]\!]$-algebra equipped with a counit $\veps:\msH\to \C[\![\hbar]\!]$, a coproduct $\Delta:\msH\to \msH \, \widehat{\otimes}\, \msH$ and an antipode $S:\msH\to \msH$ which collectively satisfy the axioms of a Hopf algebra with all tensor products given by the topological tensor product $\wh{\otimes}$. By modifying these definitions in the expected way, one obtains the notion of a topological coalgebra and bialgebra over $\C[\![\hbar]\!]$.

If $\msM$ and $\msN$ are topological $\C[\![\hbar]\!]$-modules and $\msN\cong \msW[\![\hbar]\!]$ is topologically free, then the space of $\C[\![\hbar]\!]$-module homomorphisms $\Hom_{\C[\![\hbar]\!]}(\msM,\msN)$ is separated, complete and torsion free. If in addition $\msM$ is topologically free with $\msM\cong \msV[\![\hbar]\!]$, then one has 
\begin{equation*}
\Hom_{\C[\![\hbar]\!]}(\msM,\msN)\cong \Hom_\C(\msV,\msW)[\![\hbar]\!].
\end{equation*}
In particular, the $\C[\![\hbar]\!]$-linear dual $\msM^\ast:=\Hom_{\C[\![\hbar]\!]}(\msM,\C[\![\hbar]\!])$ of a topologically free $\C[\![\hbar]\!]$-module $\msM\cong \msV[\![\hbar]\!]$ satisfies $\msM^\ast \cong \msV^\ast[\![\hbar]\!]$.
%

%
\subsection{Graded topological modules}\label{ssec:Pr-grTop}
Let us now turn towards the $\Z$-graded analogues of the above definitions. Henceforth, we view $\C[\hbar]=\bigoplus_{k\in \N}\C\hbar^k$ as an $\N$-graded ring. For brevity, we shall denote its $\N$-graded quotient $\C[\hbar]/\hbar^n \C[\hbar]$ by $\msK_n$, for each $n\in \N$. 
\begin{definition}\label{D:M-graded}
We say that a topological $\C[\![\hbar]\!]$-module $\msM$ is $\Z$-\textit{graded} if, for each $n\in \N$, $\msM/\hbar^n \msM=\bigoplus_{k\in \Z}(\msM/\hbar^n\msM)_k$ is a $\Z$-graded $\msK_n$-module and the natural homomorphism
\begin{equation*}
 \msM/\hbar^{n+1} \msM\to \msM/\hbar^{n} \msM
\end{equation*} 
is $\Z$-graded. If $(\msM/\hbar^n\msM)_k$ is trivial for $k<0$, we say that $\msM$ is $\N$-graded.

\end{definition}
A $\C[\![\hbar]\!]$-module homomorphism $\msM\to \msN$ between $\Z$-graded topological $\C[\![\hbar]\!]$-modules $\msM$ and $\msN$  is said to be  $\Z$-graded if the induced morphisms
\begin{equation*}
\msM/\hbar^n\msM \to \msN/\hbar^n\msN 
\end{equation*}
are all $\Z$-graded. More generally, it is $\Z$-graded of degree $a\in \Z$ if each of these induced morphisms is homogeneous of degree $a$. 

The category of $\Z$-graded topological modules is closed under the tensor product $\wh{\otimes}$. Indeed, this follows from the elementary observation that, given two $\C[\![\hbar]\!]$-modules $\msM$ and $\msN$, one has
\begin{equation*}
(\msM\, \wh{\otimes}\, \msN)/ \hbar^n (\msM\,\wh{\otimes}\, \msN) \cong  (\msM \otimes_{\C[\![\hbar]\!]} \msN)/ \hbar^n (\msM \otimes_{\C[\![\hbar]\!]}  \msN) \cong  \msM/\hbar^n \msM \otimes_{\msK_n} \msN/\hbar^n \msN,
\end{equation*}
which can be equipped with the standard tensor product grading, provided $\msM/\hbar^n\msM$ and $\msN/\hbar^n\msN$ are both $\Z$-graded.
\begin{definition}\label{D:A-graded}
A topological algebra $\msA$ is said to be $\Z$-graded if it is graded as a topological $\C[\![\hbar]\!]$-module and the multiplication map 
\begin{equation*}
m:\msA\,\wh{\otimes}\,\msA\to \msA
\end{equation*}
 is a $\Z$-graded homomorphism. Similarly, a topological Hopf algebra $\msH$ is $\Z$-graded if it is $\Z$-graded as a topological algebra and the structure maps 
\begin{equation*}
 \Delta:\msH\to \msH\,\wh{\otimes}\,\msH,\quad \veps:\msH\to \C[\![\hbar]\!], \quad S:\msH\to \msH
 \end{equation*}
are all $\Z$-graded $\C[\![\hbar]\!]$-module homomorphisms. Equivalently, a topological algebra or Hopf algebra $\msH$ is $\Z$-graded if the conditions of Definition \ref{D:M-graded} hold and each $\msH/\hbar^n\msH$ is a $\Z$-graded  algebra or Hopf algebra over $\msK_n$, respectively. 
\end{definition}
Of course, one also has the notion of a $\Z$-graded topological coalgebra and bialgebra, which are defined by making the obvious modifications to the above definition. 

The prototypical example of a $\Z$-graded topological module over $\C[\![\hbar]\!]$ is $\msM=\msV[\![\hbar]\!]$ where $\msV=\bigoplus_{k\in \Z}\msV_k$ is a $\Z$-graded complex vector space. In this case one has
\begin{equation*}
\msM/\hbar^n\msM \cong  \msV[\hbar]/\hbar^n \msV[\hbar] 
\end{equation*}
which is naturally graded, as $\msV[\hbar]$ is graded with $\msV[\hbar]_k= \bigoplus_{n\geq 0} \hbar^n \msV_{k-n}$, and $\hbar^n\msV[\hbar]$ is a graded submodule. The assertion that $\msM$ is $\Z$-graded may be recaptured as follows. For each $k\in \Z$, set 
\begin{equation*}
\msM_k:=\varprojlim_n \msV[\hbar]_k/\hbar^n \msV[\hbar]_{k-n} \cong \prod_{n\in \N} \hbar^n \msV_{k-n} \subset \msV[\![\hbar]\!].
\end{equation*}
Then each $\msM_k$ is a closed subspace of $\msM$ satisfying $\msM_k/\hbar^n \msM_{k-n}\cong (\msM/\hbar^n\msM)_k$, and $\msM$ contains the $\Z$-graded $\C[\hbar]$-module $\bigoplus_{k\in \Z}\msM_k$ as a dense $\C[\hbar]$-submodule. Moreover, the $\hbar$-adic topology on this submodule coincides with the subspace topology, so $\msM$ is the $\hbar$-adic completion of $\bigoplus_{k\in \Z}\msM_k$. If in addition $\msV$ is $\N$-graded, then $\bigoplus_{k\in \Z}\msM_k$ coincides with the polynomial space $\msV[\hbar]\subset \msV[\![\hbar]\!]$.

The below lemma provides an equivalent characterization of the definition of a $\Z$-graded topological module and algebra which generalizes this picture. 
\begin{lemma}\label{L:grad-top}
A topological $\C[\![\hbar]\!]$-module $\msM$ is $\Z$-graded if and only if it admits a dense, $\Z$-graded $\C[\hbar]$-submodule 
\begin{equation*}
\msM_\Z=\bigoplus_{k\in \Z}\msM_k \subset \msM
\end{equation*}
with each $\msM_k$ a closed subspace of $\msM$ and $\hbar^n \msM\cap \msM_\Z=\hbar^n\msM_\Z$ for all $n\in \N$. 

\noindent If in addition $\msM$ has the structure of a topological algebra, then it is $\Z$-graded if and only if the above conditions hold and $\msM_\Z$ is a $\Z$-graded $\C[\hbar]$-subalgebra of $\msM$. 
\end{lemma}
If $\msM$ is a $\Z$-graded topological module then the $k$-th component $\msM_k$ of $\msM_\Z$ from Lemma \ref{L:grad-top} is uniquely determined and recovered as the inverse limit
\begin{equation}\label{M_k}
\msM_k:=\varprojlim_n (\msM/\hbar^n \msM)_k \subset \varprojlim_{n} \msM/\hbar^n \msM =\msM.
\end{equation}
Moreover, one has $(\msM/\hbar^n \msM)_k\cong \msM_k/\hbar^n \msM_{k-n}$ for all $n\in \N$ and $k\in \Z$.

 We further observe that, for each $k\in \Z$, the system of linear projections $\pi_{n,k}:\msM/\hbar^n\msM\to (\msM/\hbar^n\msM)_k$ gives rise to a projection
\begin{equation*}
\pi_k:=\varprojlim_n \pi_{n,k}: \msM\to \msM_k
\end{equation*}
which restricts to the projection of $\msM_\Z$ onto its $k$-th homogeneous component. In particular, a $\C[\![\hbar]\!]$-linear map $\uptau:\msM\to \msN$ between $\Z$-graded topological modules is graded if and only if $\pi_k^\msN\circ \uptau=\uptau\circ \pi_k^\msM$ for each $k\in \Z$. 

We conclude this preliminary subsection with two corollaries of the above discussion. The first shows that any topologically free $\Z$-graded $\C[\![\hbar]\!]$-module is  of the form described above Lemma \ref{L:grad-top}. 
\begin{corollary}\label{C:Top-Z=Free}
Suppose $\msM$ is a  $\Z$-graded topologically free module over $\C[\![\hbar]\!]$, and let $\msV$ denote the $\Z$-graded complex vector space $\msM/\hbar\msM=\bigoplus_{k\in \Z} \msM_k/\hbar \msM_{k-1}$. Then
\begin{equation*}
\msM\cong \msV[\![\hbar]\!]
\end{equation*}
as a $\Z$-graded topological $\C[\![\hbar]\!]$-module. In particular, one has 
\begin{equation*}
\msM_\Z=\bigoplus_{k\in \Z} \msM_k \cong \bigoplus_{k\in \Z} \msV[\![\hbar]\!]_k \subset \msV[\![\hbar]\!], \quad \text{ where }\;\msV[\![\hbar]\!]_k=\prod_{n\in \N} \hbar^n \msV_{k-n}
\end{equation*}
\end{corollary}
\begin{proof}
This is a refinement of the elementary result, alluded to at the beginning of the section, that $\msM\cong \msV[\![\hbar]\!]$ as a $\C[\![\hbar]\!]$-module; see \cite{KasBook95}*{Prop.~XVI.2.4}, for instance. In more detail, an isomorphism of $\Z$-graded topological modules $\msM\cong \msV[\![\hbar]\!]$ is specified by choosing, for each $k\in \Z$, a complement $\msV_k\subset \msM_k$ to $\hbar\msM_{k-1}$ in $\msM_k$:  
\begin{equation*}
\msM_k=\msV_k\oplus \hbar \msM_{k-1}.
\end{equation*}
Setting $\msV:=\bigoplus_{k\in \Z}\msV_k\subset \msM_\Z$, we then have 
\begin{gather*}
\msV\cong \msM_\Z/\hbar\msM_\Z=\msM/\hbar\msM\\
\msM_k/\hbar^n \msM_{k-n} \cong \msV_k\oplus \hbar \msV_{k-1} \oplus \cdots \oplus \hbar^{n-1}\msV_{k-n+1}\cong \msV[\hbar]_k/\hbar^n \msV[\hbar]_{k-n}\\ 
\msM/\hbar^n \msM \cong \bigoplus_{k\in \Z} \msM_k/\hbar^n \msM_{k-n} \cong \bigoplus_{k\in \Z}\msV[\hbar]_k/\hbar^n \msV[\hbar]_{k-n} = \msV[\![\hbar]\!]/\hbar^n \msV[\![\hbar]\!]
\end{gather*}
where the third line is an identification of $\Z$-graded modules. Here we note that the second line follows from the definition of $\msV_k$ and that $\msM$ is a torsion free $\C[\![\hbar]\!]$-module. 
Taking inverse limits, one finds that $\msM\cong \msV[\![\hbar]\!]$ as $\Z$-graded topological $\C[\![\hbar]\!]$-modules. 
\end{proof}

Let us now shift our attention to the case where $\dot\msM=\bigoplus_{k\in \N} \dot\msM_k$ is an $\N$-graded $\C[\hbar]$-module. Any such module is automatically separated, and so embeds into its $\hbar$-adic completion 
\begin{equation*}
\msM=\varprojlim_n (\dot\msM/\hbar^n \dot\msM),
\end{equation*}
which is an $\N$-graded topological $\C[\![\hbar]\!]$-module.
Moreover, if $\dot\msM$ is a torsion free $\C[\hbar]$-module, then $\msM$ is topologically free as a  $\C[\![\hbar]\!]$-module. Since $\dot\msM_{k-n}$ is trivial for $n>k$, the submodule $\msM_k$ of $\msM$ (see \eqref{M_k}) coincides with $\dot\msM_k$ and so, in the notation of Lemma \ref{L:grad-top}, one has $\dot\msM=\msM_\N$. These observations, coupled with Corollary \ref{C:Top-Z=Free} and that $\msV[\![\hbar]\!]_k=\msV[\hbar]_k$ when $\msV$ is $\N$-graded, yield the following. 
\begin{corollary}\label{C:Top-Ngraded}
Let $\dot\msM$ be an $\N$-graded, torsion free $\C[\hbar]$-module. Then $\msM$ is a 
 topologically free $\N$-graded $\C[\![\hbar]\!]$-module. Moreover, we have 
\begin{equation*}
\dot\msM_k=\varprojlim_{n}(\dot\msM_k/\hbar^n \dot\msM_{k-n})=\msM_k \quad \text{ for all }\; k\in \N.
\end{equation*}
Consequently, $\dot\msM$ coincides with $\msM_\N$ and there is an isomorphism of $\N$-graded $\C[\hbar]$-modules
\begin{equation*}
\dot\msM\cong \msV[\hbar]=\bigoplus_{k\in \N}\msV[\hbar]_k, \quad \text{ where }\; \msV:=\dot\msM/\hbar\dot\msM.
\end{equation*}
\end{corollary}
Note that if $\dot\msM$ is a $\N$-graded $\C[\hbar]$-algebra or Hopf algebra, then $\msM$ is automatically an $\N$-graded topological algebra or Hopf algebra, respectively.

%
\subsection{The restricted dual}\label{ssec:Pr-gr*}
For a given $\Z$-graded complex vector space $\msV=\bigoplus_n \msV_n$, we let $\msV^\star=\bigoplus_n (\msV^\star)_n\subset \msV^\ast$ denote the restricted, or graded, dual of $\msV$, where 
\begin{equation*}
(\msV^\star)_n=\{f\in \msV^\ast: f(\msV_{m})\subset \C_{m+n}\}\cong (\msV_{-n})^\ast
\end{equation*}
and $\C$ is given the trivial grading with $\C_0=\C$ and $\C_m=\{0\}$ for $m\neq 0$.  One can similarly define the restricted dual $\msM^\star\subset \msM^\ast$ in the category of $\Z$-graded topological $\C[\![\hbar]\!]$-modules. In this subsection we will 
recall some properties of this duality operation in the $\N$-graded setting which will be applied to construct the dual Yangian in Section \ref{sec:Yhg*}. 

Suppose that $\msM$ is an $\N$-graded, topological $\C[\![\hbar]\!]$-module with $\N$-graded $\C[\hbar]$-submodule $\msM_\N=\bigoplus_{k\in \N}\msM_k$ as in Lemma \ref{L:grad-top}. For each $n\in \N$, let $\mathsf{J}_k$ denote the $\hbar$-adic completion of the ideal $\bigoplus_{k\geq n}\msM_k$ of $\msM_\N$. The gradation topology on $\msM$ is the topology associated to the descending filtration 
\begin{equation*}
\msM=\msJ_0\supset \msJ_1 \supset \cdots \supset \msJ_n \supset \cdots
\end{equation*} 
Equipped with this terminology, we may make the following definition. 
\begin{definition}\label{D:M-star}
The restricted dual $\msM^\star$ is defined to be the $\C[\![\hbar]\!]$-submodule of $\msM^\ast$ consisting of those $f$ which are continuous with respect to the gradation topology:
\begin{equation*}
\msM^\star:=\{f\in \msM^\ast: f(\msJ_k)\subset \hbar^n\C[\![\hbar]\!] \quad \forall\; n\in \N \;\text{ and }\; k\gg0\}.
\end{equation*}
\end{definition}
The restricted dual of any $\N$-graded topological $\C[\![\hbar]\!]$-module $\msM$ is easily seen to be separated, complete and torsion free. Let us now see that is admits a $\Z$-graded structure. 

For each $a\in \Z$, let $\Hom_{\C[\![\hbar]\!]}^a(\msM,\C[\![\hbar]\!])\subset \msM^\star$ denote the (closed) subspace consisting of $\Z$-graded $\C[\![\hbar]\!]$-module homomorphisms $f:\msM\to \C[\![\hbar]\!]$ of degree $a$.  Equivalently: 
\begin{equation*}
\Hom_{\C[\![\hbar]\!]}^a(\msM,\C[\![\hbar]\!])=\{f\in \msM^\ast: f(\msM_k)\subset \C[\hbar]_{k+a} \quad \forall\; k\in \N\}. 
\end{equation*}
Then the sum $\sum_{a\in \Z}\Hom_{\C[\![\hbar]\!]}^a(\msM,\C[\![\hbar]\!])$ is direct and the space
\begin{equation*}
(\msM^\star)_\Z:=\bigoplus_{a\in \Z}\Hom_{\C[\![\hbar]\!]}^a(\msM,\C[\![\hbar]\!]) \subset \msM^\star
\end{equation*}
is a $\Z$-graded $\C[\hbar]$-submodule of $\msM^\star$. 
\begin{remark}\label{R:M_N*}
Under the natural identification of $\msM^\ast$ with $\Hom_{\C[\hbar]}(\msM_\N,\C[\![\hbar]\!])$ one has  $\Hom_{\C[\![\hbar]\!]}^a(\msM,\C[\![\hbar]\!])\cong \Hom_{\C[\hbar]}^a(\msM_\N,\C[\hbar])$ and $(\msM^\star)_\Z$ coincides with the graded dual $(\msM_\N)^\star\subset \Hom_{\C[\hbar]}(\msM_\N,\C[\hbar])$ of $\msM_\N$ taken in the category of $\Z$-graded $\C[\hbar]$-modules. 
\end{remark}
It is not difficult to prove that, for each $n\in \N$, one has 
\begin{equation*}
\hbar^n\msM^\star \cap (\msM^\star)_\Z=\hbar^n(\msM^\star)_\Z \quad \text{ and }\quad (\msM^\star)_\Z/\hbar^n (\msM^\star)_\Z \cong \msM^\star/\hbar^n \msM^\star. 
\end{equation*}
Consequently, $\msM^\star$ coincides with the $\hbar$-adic completion of $(\msM^\star)_\Z$ and, by Lemma \ref{L:grad-top}, is a $\Z$-graded topological $\C[\![\hbar]\!]$-module. We note that, although Definition \ref{D:M-star} is strictly for an $\N$-graded $\C[\![\hbar]\!]$-module $\msM$, one can define the restricted dual in the $\Z$-graded setting precisely as the $\hbar$-adic completion of the space $(\msM^\star)_\Z$. 

If $\msM$ is itself topologically free with $\msM\cong \msV[\![\hbar]\!]$ for a graded vector space $\msV=\bigoplus_{k\in \N}\msV_n$, then the natural homomorphism $\msM^\star/\hbar \msM^\star\to\msV^\star$ is an isomorphism of graded vector spaces. As $\msM^\star$ is topologically free, Corollary \ref{C:Top-Z=Free} yields the following. 
\begin{corollary}\label{C:Vstar[[h]]}
Suppose that $\msM$ is a topologically free $\N$-graded $\C[\![\hbar]\!]$-module with $\msM\cong \msV[\![\hbar]\!]$ for a graded vector space $\msV=\bigoplus_{k\in \N}\msV_k$. Then $\msM^\star$ is isomorphic to $\msV^\star[\![\hbar]\!]$ as a $\Z$-graded topological $\C[\![\hbar]\!]$-module.
\end{corollary}
We shall say that a topologically free $\N$-graded $\C[\![\hbar]\!]$-module is of \textit{finite type} if the graded components $\msV_k$ of $\msV\cong\msM/\hbar\msM$ from the above corollary are all finite-dimensional complex vector spaces. 

Now suppose that $\msH$ is an $\N$-graded topological Hopf algebra with coproduct $\Delta$, counit $\veps$, antipode $S$, product $m$ and unit $\iota$. Since these are all $\N$-graded $\C[\![\hbar]\!]$-module homomorphisms and $\C[\![\hbar]\!]^\star \cong \C[\![\hbar]\!]$, taking transposes yields $\Z$-graded maps
\begin{gather*}
\Delta^t: (\msH\, \wh{\otimes}\,\msH)^\star\to \msH^\star,\quad  \veps^t: \C[\![\hbar]\!]\to \msH^\star, \quad S^t:\msH^\star\to \msH^\star \\ 
m^t:\msH^\star\to(\msH\, \wh{\otimes}\,\msH)^\star, \quad \iota^t: \msH^\star \to \C[\![\hbar]\!]
\end{gather*} 
which formally satisfy the axioms of a Hopf algebra. In particular, $\msH^\star$ is a topological $\C[\![\hbar]\!]$-algebra with unit $\veps^t$ and product given by restricting $\Delta^t$. It is not in general a topological coalgebra (or Hopf algebra)
as $m^t$ does not necessarily have image in $\msH^\star\, \wh{\otimes}\,\msH^\star$. However, this is the case when $\msH$ is of finite type, as we now explain. 

In general, for any two $\N$-graded topological $\C[\![\hbar]\!]$-modules $\msM$ and $\msN$, there is a canonical injective homomorphism of $\Z$-graded topological $\C[\![\hbar]\!]$-modules 
\begin{equation*}
\upgamma:\msM^\star\, \wh{\otimes}\,\msN^\star \into (\msM\, \wh{\otimes}\,\msN)^\star.
\end{equation*}
If $\msM$ and $\msN$ are topologically free with $\msM\cong \msV[\![\hbar]\!]$ and $\msN\cong \msW[\![\hbar]\!]$, then the semiclassical limit of $\upgamma$ is the natural inclusion $\msV^\star\otimes_\C \msW^\star\into (\msV\otimes_\C \msW)^\star$ which is an isomorphism provided the graded components of $\msV$ or $\msW$ are all finite-dimensional. This observation, together with Lemma \ref{L:scl-map}, implies the following proposition. 
\begin{proposition}\label{P:graded-tensor}
Let $\msM$ and $\msN$ be topologically free, $\N$-graded $\C[\![\hbar]\!]$-modules and suppose that either $\msM$ or $\msN$ is of finite type. Then $\upgamma$ is an isomorphism of $\Z$-graded topological $\C[\![\hbar]\!]$-modules
\begin{equation*}
\upgamma:\msM^\star\, \wh{\otimes}\, \msN^\star \iso (\msM\,\wh{\otimes}\, \msN)^\star
\end{equation*}
Consequently, if $\msH$ is a topologically free $\N$-graded Hopf algebra of finite type, then $\msH^\star$ is a $\Z$-graded topological Hopf algebra over $\C[\![\hbar]\!]$.
\end{proposition}

\begin{remark}\label{R:otimes}
Henceforth, we shall simply write $\otimes$ for the topological tensor product $\wh{\otimes}$. More generally, the use of the symbol $\otimes$ will always be clear from context and will be clarified should any ambiguity arise.
\end{remark}
%
%
%
\subsection{Homogeneous quantizations}\label{ssec:Pr-QUE}

Let us now recall some basic constructions from the theory of quantum groups, adapted to the graded setting. 

A topological Hopf algebra $\msH$ over $\C[\![\hbar]\!]$ is called a \textit{quantized enveloping algebra} if it is a flat deformation of the universal enveloping algebra $U(\mfb)$ of a complex Lie algebra $\mfb$ as a Hopf algebra.  Equivalently: 
\begin{itemize}
 \item The semiclassical limit $\msH/\hbar\msH$ of $\msH$ is isomorphic to $U(\mfb)$ as a Hopf algebra.
 \item $\msH$ is topologically free, and thus isomorphic to $U(\mfb)[\![\hbar]\!]$ as a $\C[\![\hbar]\!]$-module. 
 \end{itemize}
If $\msH=U_\hbar\mfb$ is a quantized enveloping algebra with semiclassical limit $U(\mfb)$, then $\mfb$ inherits from $U_\hbar\mfb$ the structure of a Lie bialgebra with cocommutator $\delta_\mfb:\mfb\to \mfb\wedge \mfb\subset U(\mfb)^{\otimes 2}$ given by the formula 
\begin{equation}\label{quant-comm}
\delta_\mfb(x):= \frac{\Delta(\dot{x})-\op{\Delta}(\dot{x})}{\hbar} \mod \hbar U_\hbar\mfb\otimes U_\hbar\mfb \quad \forall \; x\in \mfb,
\end{equation}
where $\dot{x}\in U_\hbar\mfb$ is any lift of $x$. We refer the reader to Propositions 6.2.3 and 6.2.7 of \cite{CPBook} for a detailed discussion of this point.

Conversely, if $(\mfb,\delta_\mfb)$ is a Lie bialgebra, then a \textit{quantization} of $(\mfb,\delta_\mfb)$ is a quantized enveloping algebra $U_\hbar\mfb$ with semiclassical limit $U(\mfb)$, such that $\delta_\mfb$ coincides with the cocommutator \eqref{quant-comm}. 

Now let us shift our attention to the graded setting. In what follows, we will say that a Lie bialgebra $(\mfb,\delta_\mfb)$ is $\Z$-graded if $\mfb=\bigoplus_{k\in \Z}\mfb_k$ is $\Z$-graded as a Lie algebra, and the cocommutator $\delta_\mfb$ is a graded linear map of degree $d$, for some $d\in \Z$. That is, $\delta_\mfb\in \Hom_{\C}^d(\mfb,\mfb^{\otimes 2})$. 
\begin{definition}\label{D:Hom-quant}
Let $\mfb$ be a $\Z$-graded complex Lie bialgebra with cocommutator $\delta_\mfb$. Then a \textit{homogeneous quantization} of $(\mfb,\delta_\mfb)$ is a topological Hopf algebra $U_\hbar\mfb$ satisfying: 
\begin{enumerate}
\item $U_\hbar\mfb$ is a quantization of $(\mfb,\delta_\mfb)$.
\item $U_\hbar\mfb$ is $\Z$-graded as a topological Hopf algebra, and the natural inclusion
\begin{equation*}
\mfb\into U(\mfb)\cong U_\hbar\mfb/\hbar U_\hbar\mfb
\end{equation*}
is a $\Z$-graded linear map. 
\end{enumerate}
\end{definition}
Note that the last condition guarantees that the grading on $U(\mfb)$ inherited from $U_\hbar\mfb$ coincides with that induced by the Lie algebra grading on $\mfb$. In addition, since the coproduct $\Delta$ is homogeneous of degree zero, \eqref{quant-comm} implies that the cocommutator $\delta_\mfb$ must be of degree $d=-1$. 

We shall employ similar terminology to the above in the $\N$-graded setting over $\C[\hbar]$. Namely, if $(\mfb,\delta_\mfb)$ is an $\N$-graded Lie bialgebra, then a \textit{homogeneous quantization} of $(\mfb,\delta_\mfb)$ over $\C[\hbar]$ is a space $\mathds{U}_\hbar\mfb$ such that
\begin{enumerate}
\item $\mathds{U}_\hbar\mfb$ is an $\N$-graded  torsion free Hopf algebra over $\C[\hbar]$. 
\item The semiclassical limit $\mathds{U}_\hbar\mfb/\hbar \mathds{U}_\hbar\mfb$ is isomorphic to $U(\mfb)$ as a graded Hopf algebra, with the cocommutator $\delta_\mfb$ given by \eqref{quant-comm}. 
\end{enumerate}
Note that, by Corollary \ref{C:Top-Ngraded}, such a quantization $\mathds{U}_\hbar\mfb$ is isomorphic to $U(\mfb)[\hbar]$ as an $\N$-graded $\C[\hbar]$-module, and its $\hbar$-adic completion is a homogeneous quantization of $\mfb$ over $\C[\![\hbar]\!]$.

%
\subsection{The Yangian Manin triple}\label{ssec:Pr-Manin}

The most well-known, non-trivial, example of a homogeneous quantization is the Yangian $\hYhg$ associated to an arbitrary simple Lie algebra $\mfg$ over the complex numbers. In this article we shall encounter two other, closely related, examples: the dual Yangian $\Yhgstar$ and the Yangian double $\DYhg$. Collectively, these three quantum groups arise as a quantization of a restricted Manin triple structure on $(\mft,\tplus,\tminus)$, where 
\begin{equation*}
\mft:=\mfg[t^{\pm 1}],\quad \tplus:=\mfg[t]\quad \text{ and }\quad \tminus:=t^{-1}\mfg[t^{-1}].
\end{equation*}
In this section we briefly recall how this structure is defined. 

The Lie algebra $\mft=\mfg[t^{\pm 1}]$ comes equipped with a nondegenerate, invariant bilinear form  $\langle\,,\,\rangle:\mft\otimes \mft\to \C$ given by
\begin{equation}\label{Res-t}
\langle f(t), g(s)\rangle:= -\mathrm{Res}_t(f(t),g(t)),
\end{equation}
where  $(\,,\,)$ is a fixed symmetric, invariant and non-degenerate bilinear form on $\mfg$, which has been extended to a $\C[t^{\pm 1}]$-valued bilinear form on $\mft$ by $\C[t^{\pm 1}]$-linearity.
The above form is a degree $1$ element of the restricted dual $(\mft\otimes \mft)^\star$, as defined in the beginning of Section \ref{ssec:Pr-gr*}. Namely, it vanishes on
\begin{equation*}
(\mft\otimes \mft)_k=\bigoplus_{a+b=k}\mft_a\otimes \mft_b
\end{equation*}
for any $k\neq -1$, and restricts to a nondegenerate pairing $\mft_a\otimes \mft_{-a-1}\to \C$ for any $a\in \Z$. Moreover, one has the polarization
\begin{equation*}
\mft=\tplus\oplus \tminus  \quad \text{ for }\quad  \tplus=\mfg[t],\quad  \tminus=t^{-1}\mfg[t^{-1}]
\end{equation*}
with $\tplus$ and $\tminus$ isotropic, graded Lie subalgebras of $\mft$, with gradings concentrated in non-negative and non-positive degrees, respectively. Said in fewer words, the above data gives rise to a \textit{restricted Manin triple} $(\mft,\tplus,\tminus)$; see \S5.2--5.3 of \cite{Andrea-Valerio-19}.

Since each homogeneous component $\mft_k$ of $\mft$ is finite-dimensional, this data gives rise to dual Lie bialgebra structures on $\tplus$ and $\tminus$, obtained as follows. The residue form \eqref{Res-t} yields isomorphisms of graded vector spaces
\begin{equation*}
\mathsf{Res}_{\pm}:\mft_{\scriptscriptstyle\pm}\iso \mft_{\scriptscriptstyle\mp}^\star
\end{equation*}  
which are homogeneous of degree $1$: $\mathsf{Res}_{\pm}(\mft_{{\scriptscriptstyle\pm},n})=(\mft^\star_{\scriptscriptstyle\mp})_{n+1}$ for all $n\in \Z$. 
Dualizing Lie brackets then gives rise to honest, degree $-1$, Lie bialgebra cobrackets 
\begin{equation*}
\delta_{\scriptscriptstyle\pm}=[,]_{\mft_{\mp}}^t:\mft_{\scriptscriptstyle\pm}\to \mft_{\scriptscriptstyle\pm}\wedge \mft_{\scriptscriptstyle\pm}.
\end{equation*}
Since the Casimir tensor $\Omega_\mfg\in (\mfg\otimes \mfg)^\mfg$ satisfies
\begin{equation*}
([x\otimes 1,\Omega_\mfg],y\otimes z)_{\mfg\otimes \mfg}=-(x,[y,z]) \quad \forall\; x,y,z\in \mfg,
\end{equation*}
where $(,)_{\mfg\otimes \mfg}=(,)\otimes (,)\circ (2\,3):\mfg\otimes \mfg\otimes\mfg\otimes \mfg\to \C$. It follows readily from this observation and the definition of $\mathsf{Res}_\pm$ that 
 $\delta_{\scriptscriptstyle +}$ and $\delta_{\scriptscriptstyle -}$ are given explicitly on each graded component by
\begin{gather*}
\delta_{\scriptscriptstyle +}(xt^k)=\sum_{a+b=k-1}\![x\otimes 1,\Omega_\mfg]t^as^{b}\in  \mfg[t]\otimes \mfg[s]=\mft_+\otimes \mft_+,\\
\delta_{\scriptscriptstyle -}(xt^{-k-1})=\sum_{a+b=k}[x\otimes 1,\Omega_\mfg]t^{-a-1}s^{-b-1}\in t^{-1}\mfg[t^{-1}]\otimes s^{-1}\mfg[s^{-1}]=\mft_-\otimes \mft_-
\end{gather*}
where we have used the natural identification of $(\mfg\otimes \mfg)[t^{\pm 1},s^{\pm 1}]$ with $\mfg[t^{\pm 1}]\otimes \mfg[s^{\pm 1}]$, and $a,b$ take values in $\N$. Since 
\begin{equation*}
(z-w)\sum_{a=0}^kz^aw^{k-a}=z^{k+1}-w^{k+1},
\end{equation*}
the linear map $\delta:=\delta_{\scriptscriptstyle+}\oplus(-\delta_{\scriptscriptstyle-}):\mft\to \mft \wedge \mft$ is given by the formula 
\begin{equation*}
\delta(f)(t,s)=\left[f(t)\otimes 1+ 1\otimes f(s), \frac{\Omega_\mfg}{t-s}\right]\in \mfg[t^{\pm 1}]\otimes \mfg[s^{\pm 1}] \quad \forall\; f(t)\in \mfg[t^{\pm 1}]
\end{equation*}
and defines a Lie bialgebra structure on the Lie algebra $\mft$ such that $(\tplus,\delta_{\scriptscriptstyle +})$ and $(\tminus, -\delta_{\scriptscriptstyle -})$ are Lie sub-bialgebras. This construction identifies $\mft$ with the \textit{restricted Drinfeld double} $D(\tplus)$ of the $\N$-graded Lie bialgebra $\tplus$, as defined in \cite{Andrea-Valerio-19}*{\S5.4}, for instance. 

\section{The Yangian revisited}\label{sec:Yhg}

%
\subsection{The Lie algebra \texorpdfstring{$\mfg$}{g}}\label{ssec:Yhg-g}
We henceforth fix $\mfg$ to be a finite-dimensional simple Lie algebra over the complex numbers, with invariant form $(\,,\,)$ as in Section \ref{ssec:Pr-Manin}. Let $\mfh\subset \mfg$ be a Cartan subalgebra, $\{\alpha_i\}_{i\in \mbI}\subset \mfh^\ast$ a basis of simple roots, and $\{\alpha_i^\vee\}_{i\in \mbI}$ the set of simple coroots, so that $\alpha_j(\alpha_i^\vee)=a_{ij}=2(\alpha_i,\alpha_j)/(\alpha_i,\alpha_i)$ are the entries of the Cartan matrix $\mbA=(a_{ij})_{i,j\in \mbI}$ of $\mfg$. 
Let $\Root^+\subset \mfh^\ast$ be the associated set of positive roots, and let $Q=\bigoplus_{i \in \mbI}\Z\alpha_i$ and $Q_+=\bigoplus_{i\in \mbI} \N\alpha_i$ denote the root lattice and its positive cone, respectively, where we recall that $\N$ denotes the set of non-negative integers. Set 
\begin{equation*}
d_{ij}=\frac{(\alpha_i,\alpha_j)}{2} \; \text { and }\; d_i=d_{ii} \quad \forall\; i,j\in \mbI. 
\end{equation*}
We normalize $(\,,\,)$, if necessary, so that the square length of a short root is $2$. In particular, we then have $\{d_i\}_{i\in \mbI}\subset \{1,2,3\}$. Let $\{e_i,f_i\}_{i\in \mbI}$ denote the Chevalley generators of $\mfg$, as in \cite{KacBook90}*{\S1.3}, and set 
\begin{equation*}
h_i=d_i\alpha_i^\vee, \quad x_i^+=\sqrt{d_i}e_i, \quad x_i^-=\sqrt{d_i}f_i \quad \forall \; i\in \mbI. 
\end{equation*}
These normalized generators satisfy $(x_i^+,x_i^-)=1$ and $h_i=[x_i^+,x_i^-]$ for all $i\in \mbI$.
%
\subsection{The Yangian}\label{ssec:Yhg-def}

We now recall the definition of the Yangian $\Yhg$. Let $S_m$ denote the symmetric group on $\{1,\ldots,m\}$. 
\begin{definition}\label{D:Yhg} The Yangian
 $\Yhg$ is the unital associative $\C[\hbar]$-algebra generated by  $\{x_{ir}^\pm, h_{ir}\}_{i\in \mbI,r\in \N}$, subject to the following relations for $i,j\in \mbI$ and $r,s\in \N$:
\begin{gather}
 [h_{ir},h_{js}]=0\label{Y:hh}\\
 [h_{i0},x_{js}^\pm]=\pm 2d_{ij}x_{js}^\pm, \label{Y:h0x}\\
 [x_{ir}^+,x_{js}^-]=\delta_{ij} h_{i,r+s}, \label{Y:xxh}\\
 [h_{i,r+1},x_{js}^\pm]-[h_{ir},x_{j,s+1}^\pm]=\pm \hbar d_{ij}(h_{ir}x_{js}^\pm+x_{js}^\pm h_{ir}),\label{Y:xh}\\
 [x_{i,r+1}^\pm,x_{js}^\pm]-[x_{ir}^\pm,x_{j,s+1}^\pm]=\pm \hbar d_{ij}(x_{ir}^\pm x_{js}^\pm+x_{js}^\pm x_{ir}^\pm ), \label{Y:xx}\\
 \sum_{\pi \in S_{m}} \left[x_{i,r_{\pi(1)}}^{\pm}, \left[x_{i,r_{\pi(2)}}^{\pm}, \cdots, \left[x_{i,r_{\pi(m)}}^{\pm},x_{js}^{\pm}\right] \cdots\right]\right] = 0, \label{Y:Serre}
\end{gather}
where in the last relation $i\neq j$, $m=1-a_{ij}$ and $r_1,\ldots,r_m\in \N$. 
\end{definition}
The Yangian $\Yhg$ is an $\N$-graded $\C[\hbar]$-algebra, with grading 
\begin{equation*}
\Yhg=\bigoplus_{k\in \N}\Yhg_k
\end{equation*}
 determined by $\deg x_{ir}^\pm=\deg h_{ir}=r$ for all $i\in \mbI$ and $r\in \N$. Moreover, Definition \ref{D:Yhg} is such that $\Yhg$ provides an $\N$-graded $\C[\hbar]$-algebra deformation of the enveloping algebra $U(\tplus)$, where we recall that $\tplus=\mfg[t]$. Indeed, the identification $\Yhg/\hbar\Yhg \cong U(\tplus)$ is induced by the graded algebra epimorphism $q:\Yhg\onto U(\tplus)$  given on generators by 
 \begin{equation*}
q:\;x_{ir}^\pm \mapsto x_i^\pm t^r, \quad  h_{ir}\mapsto h_it^r \quad  \forall\; i\in \mbI\;\text{ and }\; r\in \N. 
 \end{equation*}
In addition, the relations \eqref{Y:hh}--\eqref{Y:Serre} imply that the assignment
\begin{equation*}
 x_i^\pm \mapsto x_{i0}^\pm, \quad h_i\mapsto h_{i0} \quad \forall \; i\in \mbI
\end{equation*}
determines a $\C$-algebra homomorphism $U(\mfg)\to \Yhg$, which is injective as its composition with $q$ is the identity map $\id_{U(\mfg)}$ on $U(\mfg)$. Henceforth, we shall identify $\mfg$ with its image in $\Yhg$ without further comment.

To specify the standard Hopf algebra structure on $\Yhg$, we first note that $\Yhg$ is generated as a $\C[\hbar]$-algebra by the set  $\mfg\cup\{t_{i1}\}_{i\in \mbI}\subset \Yhg$, where 
\begin{equation*}
t_{i1}:=h_{i1}-\frac{\hbar}{2}h_{i0}^2 \quad \forall\; i\in \mbI.
\end{equation*}
 More precisely,  for each $s>0$, $x_{is}^\pm$ and $h_{i,s+1}$ are determined by the recursive formulas
\begin{equation*}
 x_{is}^\pm= \pm\frac{1}{2d_i}\left[t_{i1},x_{i,s-1}^\pm\right]\quad \text{ and }\quad
 h_{i,s+1}=[x_{is}^+,x_{i1}^-]. 
\end{equation*}
Now let $\msr\in \mfn_-\otimes \mfn_+$ denote the canonical tensor associated to the pairing $(\,,\,)|_{\mfn_-\times \mfn_+}$, where $\mfn_\pm=\bigoplus_{\alpha\in \Root^+}\mfg_{\pm\alpha}$ is the Lie subalgebra of $\mfg$ generated by $\{x_i^\pm\}_{i\in \mbI}$. Equivalently, $\msr$ is the unique preimage of the identity  map under the natural isomorphism $\mfn_- \otimes \mfn_+ \iso \End_\C(\mfn_-)$, determined by $(\,,\,)|_{\mfn_-\times \mfn_+}$. In addition, we set 
\begin{equation*}
\msr_i:=[h_i\otimes 1,\msr_i]\quad \forall\quad i\in \mbI.
\end{equation*}
If $x_\alpha^\pm \in \mfg_{\pm\alpha}$ are root vectors satisfying $(x_\alpha^+,x_\alpha^-)=1$, then one has the formulae
\begin{equation*}
\msr=\sum_{\alpha\in \Root^+} x_\alpha^- \otimes x_\alpha^+ \quad \text{ and }\quad \msr_i= - \sum_{\alpha\in \Root^+}\alpha(h_i)x_\alpha^- \otimes x_\alpha^+.
\end{equation*}
The following proposition describes the Hopf algebra structure on $\Yhg$, where $m:\Yhg^{\otimes 2}\to \Yhg$ denote the multiplication map.
\begin{proposition}\label{P:Yhg-Hopf}
The Yangian $\Yhg$ is an $\N$-graded Hopf algebra with  counit $\veps$, coproduct $\Delta$ and antipode $S$ uniquely determined by the requirement that $\mfg$ is primitive and that, for each $i\in \mbI$, one has 
\begin{gather*}
\veps(t_{i1})=0,\quad  \Delta(t_{i1})=t_{i1}\otimes 1+1\otimes t_{i1}+\hbar \msr_i,\quad S(t_{i1})=-t_{i1}+m(\hbar \msr_i).
\end{gather*}
In particular, $\Yhg$ is an $\N$-graded Hopf algebra deformation of $U(\tplus)$ over $\C[\hbar]$.
\end{proposition}
The crux of the proof of this proposition lies in showing that $\Delta$ is an algebra homomorphism. Though this is a consequence of \cite{Dr}*{Thm.~2} and \cite{DrNew}*{Thm.~1} (see also \cite{GRWEquiv}*{Thm.~2.6}), a complete proof has only recently appeared in \cite{GNW}; see Theorem 4.9 therein. 

The Yangian $\Yhg$ also admits a $Q$-grading compatible with the above $\N$-grading; that is to say, it is $\N\times Q$-graded as a Hopf algebra. This $Q$-grading arises from the adjoint action of the Cartan subalgebra $\mfh\subset \mfg$ on $\Yhg$. Namely, one has 
$\Yhg=\bigoplus_{\beta\in Q}\Yhg_\beta$, where $\Yhg_\beta$ is just the $\beta$-weight space 
\begin{equation*}
\Yhg_\beta=\{x\in \Yhg: \; [h,x]=\beta(h)x \quad \forall\quad h\in \mfh\} \quad \forall\; \beta\in Q. 
\end{equation*}
%

%
\subsection{Automorphisms}\label{ssec:Yhg-aut}
There are two families of (anti)automorphisms of $\Yhg$ which will play an especially pronounced role in this article: the \textit{shift} automorphisms and the \textit{Chevalley involution}. 
The former are a family $\{\tau_c\}_{c\in \C}\subset \mathrm{Aut}(\Yhg)$ which give rise to an action of the additive group $\C$ on $\Yhg$ by Hopf algebra automorphisms. In more detail, $\tau_c$ is defined explicitly by 
\begin{equation}\label{shift-c}
\tau_c(x_i^\pm(u))=x_i^\pm(u-c)\quad \text{ and } \quad \tau_c(h_i(u))=h_i(u-c) \quad \forall \; i\in \mbI,
\end{equation}
where  we have introduced the generating series $x_i^\pm(u)$ and $h_i(u)$ in $\Yhg[\![u^{-1}]\!]$ by  
\begin{equation*}
x_i^\pm(u)=\sum_{r\in \N} x_{ir}^\pm u^{-r-1} \quad \text{ and }\quad h_i(u)=\sum_{r\in \N} h_{ir} u^{-r-1}.
\end{equation*}
%
Replacing $c$ by a formal variable $z$ in \eqref{shift-c}, one obtains an $\N$-graded embedding
\begin{equation}\label{shift-z}
 \tau_z:\Yhg\into \Yhg[z]
\end{equation}
called the \textit{formal shift homomorphism}, where $\deg z=1$. Let us now turn to defining the Chevalley involution, beginning with the following lemma. 

\begin{lemma}\label{L:Chev}
The assignments $\omega$ and $\varsigma$ defined by 
\begin{gather*}
 \omega(x_i^\pm(u))=x_{i}^\mp(u), \quad \omega(h_i(u))=h_i(u)\\
 \varsigma(x_i^\pm(u))=x_i^\pm(-u), \quad \varsigma(h_i(u))=h_i(-u)
\end{gather*}
 extend to commuting anti-involutions $\omega$ and $\varsigma$  of $\Yhg$. Moreover, $\omega$ and $\varsigma$ satisfy 
 \begin{gather*}
 \tau_c \circ \omega = \omega \circ \tau_c, \quad \tau_{-c} \circ \varsigma = \varsigma \circ \tau_c \quad \forall \; c\in \C,\\
 \veps \circ \omega= \veps, \quad (\omega \otimes \omega) \circ \Delta=\op{\Delta}\circ \omega,\quad \omega\circ S = S \circ \omega,\\
  \veps \circ \varsigma= \veps, \quad (\varsigma\otimes \varsigma) \circ \Delta=\Delta\circ \varsigma, \quad \varsigma\circ S^{-1}= S\circ \varsigma.
 \end{gather*}
\end{lemma}
This result, which has appeared in various forms in the literature (for instance, \cite{ChPr1}*{Prop.~2.9}), is readily established using Definition \ref{D:Yhg} and the relations of Proposition \ref{P:Yhg-Hopf}. 
By the lemma, $\omega$ is an involutive Hopf algebra anti-automorphism of $\Yhg$, which we call the \textit{Chevalley involution} of $\Yhg$.  On $\mfg\subset Y_\hbar(\mfg)$, this recovers the standard Chevalley involution, given by
\begin{equation*}
\omega(x_i^\pm)=x_i^\mp \quad \text{ and }\quad \omega(h_i)=h_i \quad \forall\; i \in \mbI. 
\end{equation*}
Similarly,  under the identification $\Yhg/\hbar\Yhg\cong U(\tplus)$, the semiclassical limit $\bar\omega: U(\tplus)\to U(\tplus)$ of $\omega$ coincides with the anti-involution of $U(\tplus)$ uniquely extending the Lie algebra anti-automorphism 
\begin{equation}\label{bar-omega}
\bar{\omega}(x t^r)=\omega(x)t^r \quad \forall\; x\in \mfg\; \text{ and }\; r\in \N.
\end{equation}

 In addition, we note that the composite $\upkappa:=\omega\circ \varsigma$ is an involutive algebra automorphism of $\Yhg$, given explicitly by
\begin{equation}\label{Chev-aut}
 \upkappa(x_i^\pm(u))=x_i^\mp(-u) \quad \text{ and }\quad \upkappa(h_i(u))=h_i(-u) \quad \forall\; i\in \mbI.
\end{equation}
This automorphism is itself often called the Chevalley or Cartan involution of $\Yhg$, though here we shall reserve the former terminology for $\omega$.

%
\subsection{Poincar\'{e}--Birkhoff--Witt Theorem}\label{ssec:Yhg-coord}

An important foundational result in the theory of Yangians is the Poincar\'{e}--Birkhoff--Witt Theorem, which asserts the flatness of $\Yhg$ as an $\N$-graded Hopf algebra deformation of $U(\tplus)$. It can be stated concisely as follows. 
\begin{theorem}\label{T:Yhg-PBW}
The Yangian $\Yhg$ is a torsion free $\C[\hbar]$-module, and thus provides a flat deformation of $\Yhg/\hbar\Yhg\cong U(\tplus)$ as a graded Hopf algebra over $\C[\hbar]$. In particular, $\Yhg$ is isomorphic to $U(\tplus)[\hbar]$ as an $\N$-graded $\C[\hbar]$-module. 
\end{theorem}
As $\Yhg$ is an $\N$-graded algebra deformation of $U(\tplus)$, an isomorphism $\Yhg\cong U(\tplus)[\hbar]$ can be obtained by specifying an ordered, homogeneous, lift $\mathds{G}\subset \Yhg$ of any fixed homogeneous basis of the Lie algebra $\tplus$. For our purposes, it will be useful to specify a class of isomorphisms of this type with a number of useful properties. 

For each $\beta\in \Delta^+$, we may choose $i(\beta)\in \mbI$ and $\mathbf{X}^\beta\in U(\mfn_{+})_{\beta-\alpha_{i(\beta)}}\subset U(\mfg)$ such that 
\begin{equation}\label{X-beta}
x_\beta^+:=\mathbf{X}^\beta\cdot x_{i(\beta)}^+\in \mfg_{\beta} \quad \text{ and }\quad x_\beta^-:=\omega(x_\beta^+)\in \mfg_{-\beta}
\end{equation}
satisfy the duality condition $(x_\beta^+,x_\beta^-)=1$, where $\mathbf{X}^\beta$ acts on $x_{i(\beta)}^+$ via the adjoint action of $\mfg$ on $U(\mfg)$. In particular, we can (and shall) take $\mathbf{X}^{\alpha_i}=1$ for all $i\in \mbI$, so that $x_{\alpha_i}^\pm=x_i^\pm$. For each $k\in \N$, we then set 
\begin{equation*}
x_{\beta,k}^+:=\mathbf{X}^\beta\cdot x_{i(\beta),k}^+\in \Yhg_\beta\quad \text{ and }\quad x_{\beta,k}^-:=\omega(x_{\beta,k}^+)\in \Yhg_{-\beta},
\end{equation*}
where $\mfg$ now operates on $\Yhg$ via the adjoint action. This definition is such that $q(x_{\beta,k}^\pm)=x_\beta^\pm t^k$ for all $\beta\in \Root^+$ and $k\in \N$, 
and hence the set of elements 
\begin{equation*}
\mathds{G}:=\bigcup_{k\in \N}\{h_{ik},x_{\beta,k}^\pm\}_{i\in \mbI,\beta\in\Delta^+}
\end{equation*}
 reduces modulo $\hbar$ to the basis of $\tplus$ consisting of all Cartan elements $h_i t^k$ and root vector $x_\beta^\pm t^k$. For each choice of total order $\preceq$ on $\mathds{G}$, the corresponding set of ordered monomials 
\begin{equation*}
B(\mathds{G})=\{x_{1}x_2\cdots x_n: n\in \N,\, x_i\in \mathds{G} \;\text{ and }\; x_i\preceq x_j \; \forall \; i<j\}
\end{equation*}
is therefore a homogeneous basis of the $\C[\hbar]$-module $\Yhg$, and so defines an isomorphism $\N$-graded modules
\begin{equation}\label{nu_G}
\upnu_{\mathds{G}}:\Yhg\iso U(\tplus)[\hbar]
\end{equation}
uniquely determined by the property that $\upnu_{\mathds{G}}|_{B(\mathds{G})}$ coincides with the restriction of the quotient map $q$ to $B(\mathds{G})$. We note that $\upnu_\mathds{G}$ is automatically an isomorphism of $\mfh$-modules, and is thus $Q$-graded. 

We shall single out a subclass of isomorphisms of this type which are compatible with Chevalley involutions and satisfies a triangularity condition. To make this precise, we must first recall that $\Yhg$ admits a triangular decomposition, compatible with the decomposition 
\begin{equation*}
\mfg=\mfn_+\oplus \mfh\oplus \mfn_-.
\end{equation*}
Let us define  $Y_\hbar^0(\mfg)$ and $Y_\hbar^\pm(\mfg)$ to be the unital associative subalgebras of $\Yhg$ generated by 
$\{h_{ir}\}_{i\in \mbI,r\in \N}$ and $\{x_{ir}^\pm\}_{i\in \mbI,r\in \N}$, respectively.  These are $\N\times Q$-graded subalgebras of $\Yhg$. The triangular decomposition of $\Yhg$ is then encoded by the following proposition, which is a well-known consequence of Theorem \ref{T:Yhg-PBW}. 
\begin{proposition}\label{P:TriDec} 
\leavevmode 
\begin{enumerate}[font=\upshape]
\item\label{TriDec:1} $Y_\hbar^\pm(\mfg)$ is isomorphic to the unital, associative $\C[\hbar]$-algebra generated by the set $\{x_{ir}^\pm\}_{i\in \mbI,r\in \N}$, subject to relations
\eqref{Y:xx} and \eqref{Y:Serre} of Definition \ref{D:Yhg}:
%
\begin{gather*}
 [x_{i,r+1}^\pm,x_{js}^\pm]-[x_{ir}^\pm,x_{j,s+1}^\pm]=\pm \hbar d_{ij}(x_{ir}^\pm x_{js}^\pm+x_{js}^\pm x_{ir}^\pm )\\
 \sum_{\pi \in S_{m}} \left[x_{i,r_{\pi(1)}}^{\pm}, \left[x_{i,r_{\pi(2)}}^{\pm}, \cdots, \left[x_{i,r_{\pi(m)}}^{\pm},x_{js}^{\pm}\right] \cdots\right]\right] = 0,
\end{gather*}
where all indices are constrained as in Definition \ref{D:Yhg}. In particular, $Y_\hbar^\pm(\mfg)$ is an $\N$-graded, torsion free $\C[\hbar]$-algebra deformation of $U(\mfn_{\pm}[t])$.
%
%
\item\label{TriDec:3}  The assignment $h_{ik}\mapsto h_it^k$, for all $i\in \mbI$ and $k\in \N$, extends to an isomorphism of $\N$-graded, commutative $\C[\hbar]$-algebras 
\begin{equation*}
\upxi:Y_\hbar^0(\mfg)\iso U(\mfh[t])[\hbar]=\msS(\mfh[t])[\hbar].
\end{equation*}
\item\label{TriDec:4}  The multiplication map 
\begin{equation*}
m:Y_\hbar^+(\mfg)\otimes Y_\hbar^0(\mfg)\otimes Y_\hbar^-(\mfg)\to \Yhg 
\end{equation*}
is an isomorphism of graded $\C[\hbar]$-modules. 
\end{enumerate}
\end{proposition}

As a consequence of Part \eqref{TriDec:1} of Proposition \ref{P:TriDec} and Corollary \ref{C:Top-Ngraded}, one has $Y_\hbar^\pm(\mfg)\cong U(\mfn_\pm[t])[\hbar]$ as $\N$-graded $\C[\hbar]$-modules. Following the procedure outlined at the beginning of the section, let us fix an arbitrary total order $\preceq_{\scriptscriptstyle +}$ on the union
\begin{equation*}
\mathds{G}_+=\bigcup_{k\in \N}\{x_{\beta,k}^+\}_{\beta\in\Delta^+} = \mathds{G}\cap Y_\hbar^+(\mfg). 
\end{equation*}
The set of ordered monomials $B(\mathds{G}_+)$ in $\mathds{G}_+$ is a basis of $Y_\hbar^+(\mfg)$, and thus gives rise to an isomorphism of $\N\times Q$-graded $\C[\hbar]$-modules  
\begin{equation*}
\upnu_+:Y_\hbar^+(\mfg)\iso U(\mfn_+[t])[\hbar],
\end{equation*}
sending each ordered monomial in $\mathds{G}_+$ to its image in $Y_\hbar^+(\mfg)/\hbar Y_\hbar^+(\mfg)\cong U(\mfn_+[t])$. 
Using the Chevalley involution $\omega$ and its semiclassical limit $\bar\omega$ (see \eqref{bar-omega}), we then obtain an isomorphism
\begin{equation*}
\upnu_-:=\bar\omega\circ \upnu_+\circ \omega:Y_\hbar^-(\mfg)\iso U(\mfn_-[t])[\hbar]
\end{equation*}
Combining $\upnu_\pm$ with $\upxi$ from Part \eqref{TriDec:3} of Proposition \ref{P:TriDec} outputs an isomorphism of $\N\times Q$-graded $\C[\hbar]$-modules 
\begin{equation}\label{nu-Yhg}
\upnu:=\bar{m}\circ (\upnu_+\otimes \upxi \otimes \upnu_-) \circ m^{-1}: \Yhg\iso U(\tplus)[\hbar]
\end{equation}
where $\bar{m}:U(\mfn_+[t])\otimes_\C  U(\mfh[t])\otimes_\C U(\mfn_-[t])\iso U(\tplus)$ is the multiplication map, which we extend trivially by $\C[\hbar]$-linearity. By construction, $\upnu$ is compatible with the underlying triangular decompositions on $\Yhg$ and $U(\tplus)$ and satisfies 
\begin{equation*}
\upnu\circ \omega =\bar\omega\circ \upnu.
\end{equation*}
The definition \eqref{nu-Yhg} is such that $\upnu=\upnu_{\mathds{G}}$ for any total order $\preceq$ on $\mathds{G}$ which restricts to $\preceq_{\scriptscriptstyle{+}}$, satisfies $x^+\preceq h \preceq x^-$ for all $x^{\pm}\in \mathds{G} \cap Y_\hbar^\pm(\mfg)$ and $h\in Y_\hbar^0(\mfg)$, and for which $\omega$ is a decreasing function on $\mathds{G}$. We will denote the inverse of $\upnu$ by $\mu$: 
\begin{equation*}
\mu:=\upnu^{-1}: U(\tplus)[\hbar]\to \Yhg. 
\end{equation*}
Note that, for any total order on $\mathds{G}$, one has $\mu(x)=\upnu_{\mathds{G}}^{-1}(x)$ for all $x\in \mathds{G}$.

%
\subsection{Quantization}\label{ssec:Yhg-quant} 

As a consequence of Proposition \ref{P:Yhg-Hopf} and Theorem \ref{T:Yhg-PBW}, the Yangian $\Yhg$ provides a homogeneous quantization of an $\N$-graded Lie bialgebra structure on the Lie algebra $\tplus$ over $\C[\hbar]$, with cocommutator $\delta$ determined by the formula \eqref{quant-comm}.  By Proposition \ref{P:Yhg-Hopf}, $\delta$ is uniquely determined by $\delta(\mfg)=0$ and 
\begin{equation*}
\delta(h_it)=\msr_i-\msr_i^{21}=[h_i\otimes 1, \Omega_\mfg]=\left[h_it\otimes 1 + 1\otimes h_i s,\frac{\Omega_\mfg}{t-s}\right]=\delta_{\scriptscriptstyle +}(h_i t) \quad \forall\; i\in \mbI,
\end{equation*}
and thus coincides with $\delta_{\scriptscriptstyle +}$ from Section \ref{ssec:Pr-Manin}. This recovers the following well-known result, originally due to Drinfeld \cite{Dr}*{Thm.~2}:
\begin{theorem}\label{T:Yhg-quant}
$\Yhg$ is a homogeneous quantization of $(\tplus,\delta_{\scriptscriptstyle +})$ over $\C[\hbar]$. 
\end{theorem}
As explained in Section \ref{ssec:Pr-QUE}, it follows immediately that the $\hbar$-adic completion 
\begin{equation}\label{hYhg-def}
\hYhg:=\varprojlim_{n} ( \Yhg/\hbar^n \Yhg )
\end{equation}
is a homogeneous quantization of $(\tplus,\delta_{\scriptscriptstyle +})$ over $\C[\![\hbar]\!]$. We refer the reader to Definition \ref{D:Hom-quant} and Corollary \ref{C:Top-Ngraded} for a detailed discussion of this point. 
\begin{remark}\label{R:hYhg-TD}
Let $Y_\hbar^{\pm\!}\mfg$ and $Y_\hbar^0\mfg$ denote the topological $\C[\![\hbar]\!]$-algebras
\begin{equation*}
Y_\hbar^{\pm\!}\mfg:=\varprojlim_{n} Y_\hbar^\pm(\mfg)/\hbar^n Y_\hbar^\pm(\mfg) \quad \text{ and }\quad
Y_\hbar^0\mfg:=\varprojlim_{n} Y_\hbar^0(\mfg)/\hbar^n Y_\hbar^0(\mfg)
\end{equation*}
It follows from Corollary \ref{C:Top-Ngraded} and Proposition \ref{P:TriDec} that these are subalgebras of $\hYhg$, with $Y_\hbar^0\mfg$ isomorphic to $U(\mfh[t])[\![\hbar]\!]\cong \msS(\mfh[t])[\![\hbar]\!]$ as an $\N$-graded topological $\C[\![\hbar]\!]$-algebra, and $Y_\hbar^{\pm\!}\mfg$ a topologically free $\N$-graded $\C[\![\hbar]\!]$-algebra with semiclassical limit equal to $U(\mfn_{\pm}[t])$. By Part \eqref{TriDec:4} of Proposition \ref{P:TriDec}, the product $m$ on $\hYhg$ gives rise to an isomorphism of $\N$-graded topological $\C[\![\hbar]\!]$-modules 
\begin{equation}\label{m:hYhg}
m:Y_\hbar^{+\!}\mfg\otimes Y_\hbar^0\mfg \otimes Y_\hbar^{-\!}\mfg\iso \hYhg,
\end{equation} 
where, following Remark \ref{R:otimes}, $\otimes$ should now be understood to be the topological tensor product $\wh{\otimes}$ of $\C[\![\hbar]\!]$-modules. For later purposes, we note that the product on $\hYhg$ also gives rise to an isomorphism 
\begin{equation*}
Y_\hbar^{-\!}\mfg\otimes Y_\hbar^0\mfg \otimes Y_\hbar^{+\!}\mfg\iso \hYhg
\end{equation*}
which can be realized as $\upkappa\circ m\circ (\upkappa_-\otimes \upkappa_0\otimes \upkappa_+)$, where $m$ is the isomorphism \eqref{m:hYhg}, $\upkappa$ is the involutive automorphism of $\hYhg$ defined in \eqref{Chev-aut} (extended by continuity), and $\upkappa_\chi:=\upkappa|_{Y_\hbar^{\chi}\mfg}: Y_\hbar^{\chi}\mfg \iso Y_\hbar^{\shortminus\chi}\mfg$. 
\end{remark}
%
%

%
\subsection{The Universal \texorpdfstring{$R$}{R}-matrix}\label{ssec:Yhg-R}

We complete our survey of $\Yhg$ by reviewing the construction of the universal $R$-matrix $\mcR(z)$ of the Yangian, whose existence and uniqueness was first established by Drinfeld in \cite{Dr}*{Thm.~3}. We shall, however, need a refined version of Drinfeld's theorem only recently proven in \cite{GTLW19}*{\S7.4}, which reconstructs $\mcR(z)$ from the factors in its Gauss decomposition
\begin{equation*}
\mcR(z)=\mcR^+(z)\mcR^0(z)\mcR^-(z).
\end{equation*}

Let us begin with a few preliminaries. For each positive integer $n$, let 
\begin{equation*}
\Yhg^{\otimes n}[z;z^{-1}]\!]=\bigcup_{k\in \N} z^k \Yhg^{\otimes n}[\![z^{-1}]\!]\subset \Yhg^{\otimes n}[\![z^{\pm 1}]\!]
\end{equation*}
 denote the algebra of formal Laurent series in $z^{-1}$ with coefficients in $\Yhg^{\otimes n}$. Following \cite{WDYhg}*{\S4.2}, we then introduce the subspace 
\begin{equation*}
\cYhg{\vphantom{\hYhg}}^{\scriptscriptstyle{(n)}}_z:= \prod_{k\in \N} (\Yhg^{\otimes n})_k z^{-k}\subset \Yhg^{\otimes n}[\![z^{-1}]\!],
\end{equation*}
where $(\Yhg^{\otimes n})_k$ is the $k$-th graded component of the $\N$-graded algebra $\Yhg^{\otimes n}$. This is a $\C$-algebra isomorphic to the completion of $\Yhg^{\otimes n}$ with respect to its grading.  The $\C[z,z^{-1}]$-submodule of $\Yhg^{\otimes n}[z;z^{-1}]\!]$ that it generates is a $\Z$-graded $\C[\hbar]$-algebra
\begin{equation}\label{LzYhg-n}
\mathds{L}(\cYhg{\vphantom{\hYhg}}^{\scriptscriptstyle{(n)}}_z):=\bigoplus_{k\in \Z} z^k \cYhg{\vphantom{\hYhg}}^{\scriptscriptstyle{(n)}}_z \subset \Yhg^{\otimes n}[z;z^{-1}]\!]
\end{equation}
Though for the moment we shall only be interested in the case where $n=2$, such formal series spaces will reappear in later sections.  In addition to the above, we shall make use of two functions $Q_+\to \N$. First, we have the standard additive height function $\mathrm{ht}$ given by 
\begin{equation*}
\mathrm{ht}(\beta)=\sum_{i\in \mbI}n_i \quad \text{ for each }\; \beta=\sum_{i\in \mbI}n_i\alpha_i \in Q_+.
\end{equation*}
Secondly, we have an auxiliary function $\nu:Q_+\to \N$ defined by
\begin{equation}\label{nu-beta}
\nu(\beta)=\min\{k\in \N: \exists\, \beta_1,\ldots,\beta_k\in \Root^+\; \text{ with }\; \beta=\beta_1+\cdots +\beta_k\},
\end{equation}
where it is understood that $\nu(0)=0$. 

Let us now recall the construction of the factor $\mcR^-(z)$. Fix a Cartan element 
\begin{equation*}
\zeta\in \mfh\setminus \bigcup_{\beta\neq 0}\Ker(\beta),
\end{equation*}
where the union runs over all nonzero $\beta\in Q_+$. 
We then introduce
\begin{equation*}
\mcR^-_\beta(z)\in (Y_\hbar^-(\mfg)_{-\beta}\otimes Y_\hbar^+(\mfg)_\beta)[\![z^{-1}]\!] \quad \forall \; \beta \in Q_+
\end{equation*}
by setting $\mcR^-_0(z)=1$ and defining $\mcR^-_\beta(z)$ inductively in $\mathrm{ht}(\beta)$ using the formula
\begin{equation}\label{R-recur}
\mcR^-_\beta(z)=
\hbar\sum_{p\geq 0} \frac{\mathbf{T}(\zeta)^p}{(z\beta(\zeta))^{p+1}}
\sum_{\alpha\in \Root^+}\alpha(\zeta)\mcR_{\beta-\alpha}(z) (x_{\alpha}^-
\otimes x_{\alpha}^+ ),
\end{equation}
where $\mcR_{\gamma}(z)=0$ whenever $\gamma\notin Q_+$ and 
$
\mathbf{T}(\zeta)=\mathrm{ad}(\mathrm{T}(\zeta)\otimes 1 + 1\otimes \mathrm{T}(\zeta)),
$
with $\mathrm{T}:\mfh\to Y_\hbar^0(\mfg)$ the embedding determined by 
\begin{equation*}
\mathrm{T}(h_i)=t_{i1} \quad \forall \; i\in \mbI.
\end{equation*}

Using the fact that, for each $p\in \N$, $\mbT(\zeta)^p z^{-p-1}$ is a homogeneous operator on $\mathds{L}(\cYhg{\vphantom{\hYhg}}^{\scriptscriptstyle{(2)}}_z)$ of degree $-1$, one deduces from the recursive formula \eqref{R-recur} that 
\begin{equation*}
\mcR^-_\beta(z)\in \hbar^{\nu(\beta)}\mathds{L}(\cYhg{\vphantom{\hYhg}}^{\scriptscriptstyle{(2)}}_z)_{-\nu(\beta)}
\subset z^{-\nu(\beta)}\Yhg^{\otimes 2}[\![z^{-1}]\!] \quad \forall \; \beta \in Q_+.
\end{equation*}
As the set $\{\beta\in Q_+: \nu(\beta)\leq k\}$ is finite for any $k\in \N$, we obtain a well defined formal series 
\begin{equation*}
\mcR^-(z)=\sum_{\beta\in Q_+}\mcR^-_\beta(z) \in(Y_\hbar^-(\mfg)\otimes Y_\hbar^+(\mfg))[\![z^{-1}]\!]
\end{equation*} 
which by construction satisfies
$
\mcR^-(z)\in 1+\hbar \mathds{L}(\cYhg{\vphantom{\hYhg}}^{\scriptscriptstyle{(2)}}_z)_{-1}\subset \cYhg{\vphantom{\hYhg}}^{\scriptscriptstyle{(2)}}_z.
$

By Theorem 4.1 of \cite{GTLW19}, $\mcR^-(z)$ is independent of the choice of $\zeta\in \mfh$ made above and satisfies a number of remarkable properties. Notably, it intertwines $\tau_z\otimes \id \circ \Delta$ and the formal, deformed Drinfeld coproduct $\Delta_z^{\scriptscriptstyle{\operatorname{D}}}$ on  $\Yhg$, as defined in \cite{GTLW19}*{\S3.4}. We will not make direct use of these properties here, and refer the reader to \cite{GTLW19} for a detailed treatment of $\mcR^-(z)$. 

Let us now recall the definition of the abelian $R$-matrix $\mcR^0(z)\in Y_\hbar^0(\mfg)^{\otimes 2}[\![z^{-1}]\!]$ from \cite{GTLW19}*{\S6}.
 Let $\mbB=(d_i a_{ij})$ denote the symmetrization of the Cartan matrix $\mbA$. Given an indeterminate $v$, we let $\mbB(v)=([d_i a_{ij}]_v)\subset \mathrm{GL}_\mbI(\Q(v))$ be the associated matrix of $v$-numbers, where 
\begin{equation*}
[m]_v=\frac{v^m-v^{-m}}{v-v^{-1}}.
\end{equation*}   
Then it is is known \cite{GTL3}*{Thm.~A.1} that the auxiliary matrix 
\begin{equation*}
\mbC(v)=(c_{ij}(v))=[2\kappa]_v \mbB(v)^{-1}
\end{equation*}
has entries $c_{ij}(v)$ in $\N[v,v^{-1}]$, where $4\kappa$ is the eigenvalue of the Casimir element of $\mfg$ in the adjoint representation.  

Next, for each index $i\in \mbI$, we introduce the series $t_i(u)\in  \hbar  Y_\hbar^0(\mfg)[\![u^{-1}]\!]$ and its inverse Borel transform $B_i(u)\in \hbar Y_\hbar^0(\mfg)[\![u]\!]$ by 
\begin{equation*}
t_i(u)=\hbar \sum_{r\geq 0}t_{ir} u^{-r-1}=\log(1+\hbar h_i(u)) \quad \;\text{ and }\; \quad 
B_i(u)=\hbar\sum_{r\geq 0}\frac{t_{ir}}{r!}u^r.
\end{equation*}
Note that $t_{i1}$ coincides with the element of the same name introduced in Section \ref{ssec:Yhg-def}.
From this data, we obtain an element
$\mcL(z)\in (\hbar/z)^2 Y_\hbar^0(\mfg)^{\otimes 2}[\![z^{-1}]\!]$ defined by
\begin{equation*}
\mcL(z)=T^{2\kappa}\sum_{i,j\in \mbI}c_{ij}(T)B_i(\partial_z)\otimes B_j(-\partial_z) (-z^{-2}),
\end{equation*} 
where $T$ is the shift operator $T(f(z))=f(z+\frac{\hbar}{2})$ on $\Yhg^{\otimes 2}[\![z^{-1}]\!]$. Equivalently: 
\begin{equation*}
T=\exp\!\left(\frac{\hbar}{2}\partial_z\right):\Yhg^{\otimes 2}[\![z^{-1}]\!]\to \Yhg^{\otimes 2}[\![z^{-1}]\!].
\end{equation*}

As $\mcL(z)\in \hbar z^{-2}Y_\hbar^0(\mfg)^{\otimes 2}[\![z^{-1}]\!]$ and $Y_\hbar^0(\mfg)^{\otimes 2}$ is torsion free, there is a unique solution 
$\mcS(z)$ to the formal difference equation 
\begin{equation*}
\mcL(z)=\mcS(z+2\kappa\hbar)-\mcS(z) \quad \text{ with }\; \mcS(z)\in z^{-1}Y_\hbar^0(\mfg)^{\otimes 2}[\![z^{-1}]\!].
\end{equation*}
If $g(z)\in z^{-1}\C[\![z^{-1}]\!]$ is the unique solution of $-z^{-2}=g(z+1)-g(z)$, then by Proposition 6.6 of \cite{GTLW19}, we have 
\begin{equation}\label{logR0}
\mcS(z)=\frac{T^{2\kappa}}{(2\kappa\hbar)^2}\sum_{i,j\in \mbI}c_{ij}(T)B_i(\partial_z)\otimes B_j(-\partial_z)  \left(g\!\left(\frac{z}{2\kappa\hbar}\right)\right),
\end{equation}
where $g(z/2\kappa\hbar)$ is viewed as an element of $\mathds{L}(\cYhg{\vphantom{\hYhg}}^{\scriptscriptstyle{(2)}}_z)$.
The abelian $R$-matrix $\mcR^0(z)$ is defined to be the formal series exponential of this solution:
\begin{equation*}
\mcR^0(z)=\exp(\mcS(z))\in 1+z^{-1}Y_\hbar^0(\mfg)^{\otimes 2}[\![z^{-1}]\!].
\end{equation*}
Equivalently, it is the unique formal solution in $1+z^{-1}Y_\hbar^0(\mfg)^{\otimes 2}[\![z^{-1}]\!]$ of the equation
\begin{equation*}
\mcR^0(z+2\kappa\hbar)=\mcA(z)\mcR^0(z),  
\end{equation*}
 where $\mcA(z)=\exp(\mcL(z))$. As $T$ and $\hbar^{-2}B_i(\partial_z)\otimes B_j(-\partial_z)$ are homogeneous operators of degree zero on $\mathds{L}(\cYhg{\vphantom{\hYhg}}^{\scriptscriptstyle{(2)}}_z)$, it follows from \eqref{logR0} that 
\begin{equation*}
\mcR^0(z)\in 1+\hbar \mathds{L}(\cYhg{\vphantom{\hYhg}}^{\scriptscriptstyle{(2)}}_z)_{-1}\subset \cYhg{\vphantom{\hYhg}}^{\scriptscriptstyle{(2)}}_z. 
\end{equation*} 

We are now in a position to introduce the universal $R$-matrix of the Yangian. 
Set $\mcR^+(z)=\mcR_{21}(-z)^{-1}$ and define
\begin{equation*}
\mcR(z):=\mcR^+(z)\mcR^0(z)\mcR^-(z)\in 1+z^{-1}\Yhg^{\otimes 2}[\![z^{-1}]\!]. 
\end{equation*}
The following result is the content of Theorem 7.4 of \cite{GTLW19}.
\begin{theorem}\label{T:Yhg-R}
$\mcR(z)$ is the unique formal series in
$
1+z^{-1}\Yhg^{\otimes 2}[\![z^{-1}]\!]
$
satisfying the intertwiner equation 
\begin{equation}\label{R-inter}
\tau_z\otimes \id \circ \op{\Delta}(x)= \mcR(z) \cdot \tau_z\otimes \id \circ \Delta(x) \cdot \mcR(z)^{-1} \quad \forall\; x\in \Yhg
\end{equation}
in $\Yhg^{\otimes 2}[z;z^{-1}]\!]$, in addition to the cabling identities 
\begin{align*}
\Delta\otimes \id (\mcR(z))&= \mcR_{13}(z)\mcR_{23}(z)\\
\id\otimes \Delta (\mcR(z))&= \mcR_{13}(z)\mcR_{12}(z)
\end{align*}
in $\Yhg^{\otimes 3}[\![z^{-1}]\!]$. Moreover, $\mcR(z)$ has the following properties:
\begin{enumerate}[font=\upshape]
\item\label{Yhg-R:1} It is unitary: $\mcR(z)^{-1}=\mcR_{21}(-z)$. 
\item\label{Yhg-R:2} For any $a,b\in \C$, one has 
\begin{equation*}
(\tau_a\otimes \tau_b)\mcR(z)=\mcR(z+a-b).
\end{equation*}
\item\label{Yhg-R:3} $\mcR(z)$ is a homogeneous, degree zero, element of $\mathds{L}(\cYhg{\vphantom{\hYhg}}^{\scriptscriptstyle{(2)}}_z)$, with 
\begin{equation*}
\mcR(z)-1\in \hbar \mathds{L}(\cYhg{\vphantom{\hYhg}}^{\scriptscriptstyle{(2)}}_z)_{-1}= \hbar z^{-1}\cYhg{\vphantom{\hYhg}}^{\scriptscriptstyle{(2)}}_z
\end{equation*}
and semiclassical limit given by 
\begin{equation*}
q^{\otimes 2} \hbar^{-1}({\mcR(z)-1})=\frac{\Omega_\mfg}{z+t-w}\in (U(\mfg[t])\otimes U(\mfg[w]))[\![z^{-1}]\!]
\end{equation*}
\end{enumerate}
\end{theorem}
The series $\mcR(z)$ is the \textit{universal $R$-matrix} of the Yangian, and is related to the element $\mcR^D(z)\in \Yhg^{\otimes 2}[\![z^{-1}]\!]$ introduced by Drinfeld in Theorem 3 of \cite{Dr} by $\mcR(z)=\mcR^D(-z)^{-1}$; see \S1.1 and Corollary 7.4 of \cite{GTLW19}. 
\begin{remark}\label{R:strict}
Strictly speaking, the results of \cite{GTLW19} are stated with $\hbar$ replaced by an arbitrary nonzero complex number. However, it is easy to translate between the numerical and formal $\hbar$ settings via a homogenization procedure, and for the sake of completeness we make this rigorous in Appendix \ref{A:R-matrix}; see Proposition \ref{P:Yhg-R}. 
\end{remark}
The final result of this section shows that, in particular, $\mcR(z)$ is invariant under the Chevalley involution $\omega$ of Section \ref{ssec:Yhg-aut}. 
\begin{corollary}\label{C:Yhg-R-Chev}
The universal $R$-matrix $\mcR(z)$ satisfies 
\begin{equation*}
(\omega\otimes \omega)\mcR(z)=\mcR(z)\quad \text{ and } \quad (\varsigma\otimes \varsigma)\mcR(z)=\mcR_{21}(z)=(\upkappa\otimes \upkappa)\mcR(z).
\end{equation*}
\end{corollary}
\begin{proof}
Set $\mcR^\omega(z):=(\omega\otimes \omega)\mcR(z)$ and $\mcR^\varsigma(z):=(\varsigma\otimes \varsigma)\mcR(z)$. Applying the anti-automorphisms  $\omega\otimes \omega$ and $\varsigma\otimes \varsigma$ to the intertwiner equation \eqref{inter}, while making use of the relations of Lemma \ref{L:Chev}, we find that 
\begin{gather*}
\tau_z\otimes \id \circ \Delta(x)= \mcR^{\omega}(z)^{-1} \cdot \tau_z\otimes \id \circ \op{\Delta}(x) \cdot \mcR^\omega(z)\\
\tau_{-z}\otimes \id \circ \op{\Delta}(x)= \mcR^\varsigma(z)^{-1} \cdot \tau_{-z}\otimes \id \circ \Delta(x) \cdot \mcR^\varsigma(z)
\end{gather*}
for all $x\in \Yhg$. Hence, $\mcR^\omega(z)$ and $\mcR^\varsigma(-z)^{-1}=\mcR_{21}^{\varsigma}(z)$ are both solutions of \eqref{inter}. One verifies similarly that these both satisfy the cabling identities, and hence coincide with $\mcR(z)$ by the uniqueness statement of Theorem \ref{T:Yhg-R}. Since $\upkappa=\varsigma\circ \omega$, this completes the proof of the proposition. \qedhere
\end{proof}

\section{The Yangian double}\label{sec:DYhg}

We now recall the definition and main properties of the Yangian double $\DYhg$, including a review of some of the results of \cite{WDYhg}. These results, summarized in Theorems \ref{T:Phiz} and \ref{T:Phi}, will play an integral role in establishing in Sections \ref{sec:DYhg-quant} and \ref{sec:DYhg-double} that $\DYhg$ is a homogeneous quantization of the Lie bialgebra $\mft=\mfg[t^{\pm 1}]$ isomorphic to the restricted quantum double of the Yangian.

%
\subsection{The Yangian double}\label{ssec:DYhg-def}
The definition of the Yangian double $\DYhg$ is obtained by allowing the second index of the generators in Definition \ref{D:Yhg} to take values in $\Z$, while working in the category of topological $\C[\![\hbar]\!]$-algebras:
\begin{definition}\label{D:DYhg} The Yangian double
 $\DYhg$ is the unital, associative $\C[\![\hbar]\!]$-algebra topologically generated by  $\{x_{ir}^\pm, h_{ir}\}_{i\in \mbI,r\in \Z}$, subject to the relations \eqref{Y:hh} -- \eqref{Y:Serre} of Definition \ref{D:Yhg}. In terms of generating series
 \begin{equation*}
 \mcX_i^\pm(u)=\sum_{r\in \Z}x_{ir}^\pm u^{-r-1} \quad \text{ and }\quad \mcH_i(u)=\sum_{r\in \Z}h_{ir}u^{-r-1},
\end{equation*}
these defining relations can be expressed as follows, for $i,j\in \mbI$:
 \begin{gather*}
[\mcH_{i}(u),\mcH_j(v)] =0,\label{DY:hh}\\
[h_{i0},\mcX_j^\pm(u)]=\pm 2d_{ij}\mcX_j^\pm(u), \label{DY:h0x}\\
\left(u-v\mp\hbar d_{ij}\right)\mcH_i(u)\mcX_j^{\pm}(v)=\left(u-v\pm\hbar d_{ij}\right)\mcX_j^{\pm}(v)\mcH_i(u), \label{DY:xh}\\
\left(u-v\mp\hbar d_{ij}\right)\mcX_{i}^{\pm}(u)\mcX_{j}^{\pm}(v) =\left(u-v\pm\hbar d_{ij}\right) \mcX_{j}^{\pm}(v)\mcX_{i}^{\pm}(u) ,  \label{DY:xx}\\
[\mcX_i^+(u),\mcX_j^-(v)] =\delta_{ij} u^{-1}\delta(v/u)\mcH_i(v), \label{DY:xxh}\\
\sum_{\pi \in S_{m}} \left[\mcX_i^{\pm}(u_{\pi(1)}), \left[\mcX_i^{\pm}(u_{\pi(2)}), \cdots, \left[\mcX_i^{\pm}(u_{\pi(m)}),\mcX_j^{\pm}(v)\right] \cdots\right]\right] = 0, \label{DY:Serre}
\end{gather*}
where $\delta(u)=\sum_{r\in \Z}u^r\in \C[\![u^{\pm 1}]\!]$ is the formal delta function and in the last relation $i\neq j$ and $m=1-a_{ij}$.  
 \end{definition}
 $\DYhg$ is $\Z$-graded as a topological $\C[\![\hbar]\!]$-algebra, with grading induced by the degree assignment $\deg x_{ir}^\pm=\deg h_{ir}=r$ for all $i\in \mbI$ and $r\in \Z$. That is, if $\DYhg_k$ denotes the closure of the subspace of $\DYhg$ spanned over the complex numbers by monomials in $x_{ir}^\pm$, $h_{ir}$ and $\hbar$ of total degree $k$, then 
 \begin{equation*}
 \DYg:=\bigoplus_{k\in \Z}\DYhg_k
 \end{equation*}
 is a dense, $\Z$-graded $\C[\hbar]$-subalgebra of $\DYhg$ satisfying the conditions of Lemma \ref{L:grad-top}. In particular, in the notation of Section \ref{ssec:Pr-grTop}, one has $\DYg=\DYhg_\Z$. 
 \begin{remark}\label{R:DYg-j}
  Let $\DYg^\jmath$ denote the $\C[\hbar]$-algebra generated by $\{x_{ir}^\pm,h_{ir}\}_{i\in \mbI,r\in\Z}$, subject to the defining relations \eqref{Y:hh}--\eqref{Y:Serre}. Then $\DYg^\jmath$ is a $\Z$-graded $\C[\hbar]$-algebra, and there is a natural algebra homomorphism 
  \begin{equation*}
  \jmath: \DYg^\jmath\to \DYg\subset\DYhg. 
  \end{equation*}
  The kernel of $\jmath$ is the graded ideal $\cap_{n\in \N}\hbar^n \DYg^\jmath$, and $\DYhg$ can be recovered as the $\hbar$-adic completion of $\DYg^\jmath$; see \cite{WDYhg}*{Prop.~2.7}. Thus $\jmath(\DYg^\jmath)$ is a dense,  $\Z$-graded $\C[\hbar]$-subalgebra of $\DYhg$. It is, however, a proper subalgebra of $\DYg$ as the graded components $\jmath(\DYg^\jmath_k)$ are not closed in $\DYhg$. Rather, one has 
  \begin{equation*}
  \DYhg_k=\varprojlim_n(\DYg_k^\jmath/\hbar^n \DYg_{k-n}^\jmath) \quad \forall\; k\in \Z. 
  \end{equation*}
 \end{remark}
The above definition implies that $\DYhg$ is a $\Z$-graded $\C[\![\hbar]\!]$-algebra deformation of the enveloping algebra $U(\mft)$, where we recall that $\mft=\mfg[t^{\pm 1}]$. Analogously to the Yangian case recalled in Section \ref{ssec:Yhg-def}, the identification $\DYhg/\hbar\DYhg \cong U(\mft)$ is induced by the graded $\C[\![\hbar]\!]$-algebra epimorphism $\DYhg\onto U(\mft)$  given by 
 \begin{equation*}
x_{ir}^\pm \mapsto x_i^\pm t^r, \quad  h_{ir}\mapsto h_it^r \quad  \forall\; i\in \mbI\;\text{ and }\; r\in \Z. 
 \end{equation*}
The Poincar\'{e}--Birkhoff--Witt Theorem for $\DYhg$, established in Theorem 6.2 of \cite{WDYhg}, asserts that $\DYhg$ is a topologically free $\C[\![\hbar]\!]$-module, and thus a flat deformation of $U(\mft)$: 
 \begin{theorem}\label{T:DYhg-PBW}
 $\DYhg$ is a flat deformation of the $\Z$-graded algebra $\DYhg/\hbar\DYhg\cong U(\mft)$ over $\C[\![\hbar]\!]$. 
 In particular, $\DYhg\cong U(\mft)[\![\hbar]\!]$ as a $\Z$-graded topological $\C[\![\hbar]\!]$-module. 
 \end{theorem}
The notation for the generators of $\DYhg$ may seem, on the surface, to conflict with the notation used for generators in the Yangian associated to $\mfg$. However, there is a natural $\Z$-graded $\C[\![\hbar]\!]$-algebra homomorphism 
\begin{equation}
\imath:\hYhg \to \DYhg
\end{equation}
sending each generator of $\Yhg\subset \hYhg$ to the corresponding element of $\DYhg$, denoted with the same symbol. By Corollary 4.4 of \cite{WDYhg} $\imath$ is injective, and we shall henceforth identify $\hYhg$ with $\imath(\hYhg)$.

%
\subsection{Automorphisms and root vectors}\label{ssec:DYhg-root}
To each $i\in \mbI$, we may associate series $\dot{x}_i^\pm(u)$ and $\dot{h}_i(u)$ in $\DYhg[\![u]\!]$ by setting  
\begin{equation*}
\dot{x}_i^\pm(u):= x_i^\pm(u)-\mcX_i^\pm(u) \quad \text{ and }\quad  \dot{h}_i^\pm(u):= h_i^\pm(u)-\mcH_i^\pm(u). 
\end{equation*}
The following lemma is then a straightforward consequence of the defining relations of $\DYhg$, where $\omega$ and $\varsigma$ are as in Lemma \ref{L:Chev}.  
\begin{lemma}\label{L:DYhg-aut}
There are unique extensions of the anti-automorphisms $\omega$ and $\varsigma$ of $\Yhg$ to anti-automorphisms of the $\C[\![\hbar]\!]$-algebra $\DYhg$ such that, for each $i\in \mbI$, 
\begin{gather*}
 \omega(\dot{x}_i^\pm(u))=\dot{x}_{i}^\mp(u), \quad \omega(\dot{h}_i(u))=\dot{h}_i(u),\\
 \varsigma(\dot{x}_i^\pm(u))=\dot{x}_i^\pm(-u), \quad \varsigma(\dot{h}_i(u))=\dot{h}_i(-u).
\end{gather*}
\end{lemma}
Following the terminology from Section \ref{ssec:Yhg-aut}, we shall refer to the involution $\omega$ as the \textit{Chevalley involution} of $\DYhg$.

The adjoint action of $\mfh\subset \DYhg$ on $\DYhg$  gives rise to a topological $Q$-grading on  the $\C[\![\hbar]\!]$-algebra $\DYhg$ (\textit{cf.} Corollary \ref{C:Yhg*-Qgrad} and \cite{WDYhg}*{\S3.1}) with graded components given by the weight spaces 
\begin{equation*}
\DYhg_\beta:=\{ x\in \DYhg: [h,x]=\beta(h)x \quad \forall\; h\in \mfh\} \quad \forall \; \beta \in Q.
\end{equation*} 
That is to say, each of these subspaces is a closed $\C[\![\hbar]\!]$-submodule of $\DYhg$, and the direct sum
\begin{equation*}
\DYhg_Q:=\bigoplus_{\beta\in Q} \DYhg_\beta
\end{equation*}
is a $Q$-graded dense $\C[\![\hbar]\!]$-subalgebra of $\DYhg$ whose subspace topology coincides with its $\hbar$-adic topology. Here we have borrowed, and modified appropriately, the terminology of Section \ref{ssec:Pr-grTop}. It should be emphasized that the word \textit{topological} is key in this statement, as the $Q$-graded algebra $\DYhg_Q$ is a proper subset of $\DYhg$. 

We now introduce root vectors in $\DYhg$ of arbitrary degree, following the procedure used in Section \ref{ssec:Yhg-coord}. 
Recall from \eqref{X-beta} that to each positive root $\beta\in \Root^+$ we attached an index $i(\beta)\in \mbI$ and an element $\mathbf{X}^\beta\in U(\mfn_{+})_{\beta-\alpha_{i(\beta)}}$. For each $k\in \Z$, we then set 
\begin{equation*}
x_{\beta,k}^+:=\mathbf{X}^\beta\cdot x_{i(\beta),k}^+\in \DYhg_\beta\quad \text{ and }\quad x_{\beta,k}^-:=\omega(x_{\beta,k}^+)\in \DYhg_{-\beta},
\end{equation*}
where $\mfg$ operates on $\DYhg$ via the adjoint action. For $k\in \N$, these elements are identical to those introduced below \eqref{X-beta}. Moreover, we have 
\begin{equation*}
x_\beta^\pm t^r =x_{\beta,r}^\pm\mod \hbar \quad  \forall\; \beta\in \Root^+,\, r\in \Z.
\end{equation*}
It shall be convenient for us to organize the above elements into generating series $x_\beta^\pm(u)\in  \DYhg_{\pm\beta}[\![u^{-1}]\!]$ and $\dot{x}_\beta^\pm(u)\in \DYhg_{\pm\beta}[\![u]\!]$  by setting
\begin{equation*}
x_\beta^\pm(u):=\sum_{r\in \N} x_{\beta,r}^\pm u^{-r-1}\quad \text{ and }\quad  \dot{x}_\beta^\pm(u):=-\sum_{r\in \N} x_{\beta,-r-1}^\pm u^r \quad \forall\quad \beta \in \Root^+. 
\end{equation*}
%
%
\subsection{The formal shift operator}\label{ssec:DYhg-shift}

We now shift our attention to recalling some of the main constructions of \cite{WDYhg}, subject to our standing assumption that $\mfg$ is a finite-dimensional simple Lie algebra. To begin, we introduce a number of relevant spaces built from the Yangian $\Yhg$, following \S4.1--4.2 of \cite{WDYhg} and Section \ref{ssec:Yhg-R} above.  Firstly, let 
\begin{equation*}
\cYhg=\prod_{k\in \N} \Yhg_k 
\end{equation*}
denote the formal completion of $\Yhg$ with respect to its $\N$-grading. This is a topologically free $\C[\![\hbar]\!]$-algebra containing $\hYhg$ as a subalgebra; see \cite{GTL1}*{Prop.~6.3} or \cite{WDYhg}*{Lem.~4.1}. 

Next, let $\LzYhg$ and $\Yhgz$ denote the subspaces $\mathds{L}(\cYhg{\vphantom{\hYhg}}^{\scriptscriptstyle{(1)}}_z)$ and $\cYhg{\vphantom{\hYhg}}^{\scriptscriptstyle{(1)}}_z$ of the space of Laurent series $\Yhg[z;z^{-1}]\!]$ introduced in Section \ref{ssec:Yhg-R}. That is, 
\begin{equation*}
\Yhgz=\prod_{k\in \N}\Yhg_k z^{-k} \subset \Yhg[\![z^{-1}]\!], 
\end{equation*}
and $\LzYhg$ is the $\Z$-graded subalgebra of $\Yhg[z;z^{-1}]\!]$ over $\C[\hbar]$ defined by
\begin{equation*}
\LzYhg=\bigoplus_{n\in \Z} z^n \Yhgz. 
\end{equation*}
The following lemma, established in \cite{WDYhg}*{Prop.~4.2}, provides a characterization of the $\hbar$-adic completion of $\LzYhg$. 
\begin{lemma}\label{L:DYhg-LzYhg}
The $\hbar$-adic completion $\LzhYhg$ of $\LzYhg$ is the subspace of $\hYhg[\![z^{\pm 1}]\!]$ consisting of formal series 
\begin{equation*}
\sum_{k\in \Z} z^k f_k(z), \quad f_k(z)\in \Yhgz
\end{equation*}
with the property that, for each $n\in \N$, $\exists$ $N_n\in \N$ such that 
\begin{equation*}
f_k(z)\in (\hbar/z)^n \Yhgz \quad \forall \quad |k|\geq N_n.
\end{equation*}
Moreover, $\LzhYhg$ is a topologically free $\Z$-graded $\C[\![\hbar]\!]$-algebra with 
\begin{equation*}
(\LzhYhg)_\Z=\LzYhg. 
\end{equation*}
\end{lemma}
The last statement of the lemma employs the notation from Lemma \ref{L:grad-top} and follows from the fact that $\LzYhg$ is a torsion free $\Z$-graded $\C[\hbar]$-algebra and that each subspace $z^k\Yhgz$ is closed in $\hYhg[\![z^{\pm 1}]\!]$, equipped with the $\hbar$-adic topology, and thus in $\LzhYhg$.

Next, recall that $\tau_z$ is the formal shift homomorphism of the Yangian introduced in \eqref{shift-z}, which we may view as a $\C[\![\hbar]\!]$-algebra homomorphism 
\begin{equation*}
\tau_z:\hYhg\into \hYhg[\![z]\!].
\end{equation*}
In addition, we set $\partial_z^{(n)}=\frac{1}{n!}\partial_z^n$ for each $n\in \N$, where $\partial_z$ is the partial derivative operator with respect to $z$. 
The following theorem, which is a combination of a special case of Theorems 4.3 and 6.2 of \cite{WDYhg},
introduces the so-called \textit{formal shift operator} $\Phi_z$ on $\DYhg$. 
\begin{theorem}\label{T:Phiz}
There is a unique homomorphism of $\C[\![\hbar]\!]$-algebras 
\begin{equation*}
\Phi_z:\DYhg\to \LzhYhg
\end{equation*}
with the property that $\Phi_z\circ \imath=\tau_z$. Moreover:
\begin{enumerate}[font=\upshape]
\item\label{Phiz:1} $\Phi_z$ is injective, and satisfies
\begin{align*}\label{Phi_z}
\Phi_z(\dot{x}_\beta^\pm(u))&=\sum_{n\in \N}(-1)^{n}u^n \partial_z^{(n)} x_\beta^\pm(-z) \quad \forall \;\beta\in \Root^+,\\
\Phi_z(\dot{h}_i(u))&=\sum_{n\in \N}(-1)^{n}u^n \partial_z^{(n)} h_i(-z) \quad \forall \; i\in \mbI.
 \end{align*}
\item\label{Phiz:2} The restriction of $\Phi_z$ to $\DYg$ is a $\Z$-graded $\C[\hbar]$-algebra homomorphism
 \begin{equation*}
 \Phi_z|_{\DYg}: \DYg\to \LzYhg=\bigoplus_{n\in \Z}z^n \Yhgz
 \subset \Yhg[z;z^{-1}]\!]. 
 \end{equation*}
\end{enumerate}
\end{theorem}
\begin{remark}\label{R:Phiz-graded}
In the terminology of Section \ref{ssec:Pr-grTop}, Part \eqref{Phiz:2} is equivalent to the assertion that $\Phi_z$ is a $\Z$-graded $\C[\![\hbar]\!]$-algebra homomorphism. This is implied by Part (2) of Theorem 4.3 in \cite{WDYhg}, which asserts that the composition $\Phi_z\circ \jmath$ is a $\Z$-graded algebra homomorphism, where $\jmath$ is as in Remark \ref{R:DYg-j}. 
\end{remark}
By \cite{WDYhg}*{Prop. 4.2 (4)} the evaluation map
\begin{equation}\label{DYhg-Ev}
\mathscr{Ev}:\LzhYhg\to \cYhg,\quad f(z)\mapsto f(1),
\end{equation}
is an epimorphism of $\C[\![\hbar]\!]$-algebras. We may thus compose $\Phi_z$ with $\mathscr{Ev}$ to obtain a $\C[\![\hbar]\!]$-algebra homomorphism 
\begin{equation*}
\Phi:=\mathscr{Ev}\circ \Phi_z:\DYhg\to \cYhg.
\end{equation*}
By Theorem 6.2 of \cite{WDYhg}, this homomorphism is injective. 
One of the main results of \cite{WDYhg} is that $\Phi$ induces an isomorphism between completions of $\DYhg$ of $\Yhg$. To make this precise, let 
 $\mcJ\subset \DYhg$ denote the kernel of the composition
\begin{equation*}
 \DYhg\xrightarrow{\hbar\mapsto 0} U(\mfg[t^{\pm 1}])\xrightarrow{t\mapsto 1} U(\mfg),
\end{equation*}
and define $\cDYhg$ to be the completion of $\DYhg$ with respect to the $\mcJ$-adic filtration
\begin{equation*}
\DYhg=\mcJ^0\supset \mcJ \supset \mcJ^2\supset \cdots \supset \mcJ^n \supset \cdots 
\end{equation*}
We then have the following analogue of \cite{GTL1}*{Thm.~6.2}, which is a special case of Theorem 5.5 in \cite{WDYhg}.
\begin{theorem}\label{T:Phi}
$\Phi$ is injective and induces an isomorphism of $\C[\![\hbar]\!]$-algebras 
\begin{equation*}
\wh{\Phi}:\cDYhg\iso \cYhg
\end{equation*}
with inverse $\Gamma$ uniquely extending the embedding $\imath\circ\tau_{-1}:\Yhg\to \DYhg$.
\end{theorem}
\begin{remark}\label{R:DYhg->cDYhg}
One subtle consequence of this result is that the natural homomorphism
\begin{equation*}
\DYhg\to \cDYhg
\end{equation*}
is injective. Indeed, its composition with the isomorphism $\wh{\Phi}$ recovers the injection $\Phi$. Henceforth, we shall freely make use of this fact and view $\DYhg$ as a subalgebra of $\cDYhg$.
We further note that the subspace topology on $\DYhg$, with respect to the $\hbar$-adic topology on $\cDYhg$, coincides with the $\hbar$-adic topology on $\DYhg$. Indeed, as $\cDYhg\cong \cYhg$ is torsion free, to see this it suffices to show that 
\begin{equation*}
\hbar\cDYhg\cap \DYhg=\hbar \DYhg.
\end{equation*}
This, however, follows immediately from the injectivity of the semiclassical limit of $\Phi$, established in \cite{WDYhg}*{Thm.~6.2}. In particular, this discussion implies that $\DYhg$ is a closed subspace of the topological $\C[\![\hbar]\!]$-module $\cDYhg$. Similarly, one deduces that $\Phi_z(\DYhg)$ is a closed subspace of $\LzhYhg$.

%
\end{remark}
To conclude this preliminary section on $\DYhg$, we introduce the auxiliary $\C[\![\hbar]\!]$-algebra homomorphism 
\begin{equation}\label{Gamma_z}
\Gamma_z:= \Gamma\circ \mathscr{Ev}: \LzhYhg\to \cDYhg
\end{equation}
which has the property that $\Gamma_z|_{\mathrm{Im}(\Phi_z)}=\Phi_z^{-1}$. This homomorphism shall play a prominent role in the main result of Section \ref{sec:DYhg-quant} and its proof; see Theorem \ref{T:dual}.

\section{The Drinfeld--Gavarini Yangian}\label{sec:QFYhg}

In this section and Section \ref{sec:Yhg*}, we give a self-contained exposition to the dual Yangian $\Yhgstar$, which provides a homogeneous quantization of the graded dual $\tminus=t^{-1}\mfg[t^{-1}]$ to the Lie bialgebra $\tplus=\mfg[t]$, as will be explained in detail in Section \ref{sec:Yhg*}. The definition of $\Yhgstar$ takes as input the so-called Drinfeld--Gavarini subalgebra of the Yangian. The goal of the present section is to introduce this subalgebra and survey some of its key properties.

%
\subsection{Quantum duality}\label{ssec:QUE-dual} 
 
To provide context, let us first briefly recall the general construction of the dual of a quantized enveloping algebra, following \cite{DrQG}*{\S7} and \cite{Gav02}; see also \cite{Etingof-Kazhdan-I}*{\S4.4} and  \cite{Andrea-Valerio-18}*{\S2.19}, for example. 

Suppose that $U_\hbar\mfb$ is a quantization of a finite-dimensional Lie bialgebra $(\mfb,\delta_\mfb)$, where we follow the terminology and notation from Section \ref{ssec:Pr-QUE}. One would then like to introduce a notion of duality which sends $U_\hbar\mfb$ to a quantization of the Lie bialgebra dual $(\mfb^\ast, \delta_{\mfb^\ast}=[\,,\,]_\mfb^\ast)$ to $(\mfb,\delta_\mfb)$. 
The first crucial observation is that  $\C[\![\hbar]\!]$-linear dual $U_\hbar\mfb^\ast=\Hom_{\C[\![\hbar]\!]}(U_\hbar\mfb,\C[\![\hbar]\!])$ of $U_\hbar\mfb$ is not itself a quantized enveloping algebra; see  Lemma 2.1 of \cite{Gav02}, in addition to \cite{DrQG}*{\S7} and \cite{Andrea-Valerio-18}*{\S2.19}. The correct notion of duality within the category of quantized enveloping algebras was introduced in \cite{DrQG}*{\S7}. One considers the $\C[\![\hbar]\!]$-submodule 
\begin{equation*}
U_\hbar\mfb^\prime:=\{x\in U_\hbar\mfb:(\id-\veps)^{\otimes n}\Delta^n(x)\in \hbar^n U_\hbar\mfb ^{\otimes n}\; \forall\; n\in \N\}\subset U_\hbar\mfb,
\end{equation*}
where $\veps$ and $\Delta$ are the counit and coproduct, respectively, on the topological Hopf algebra $U_\hbar\mfb$, and all notation is as in Section \ref{ssec:Yhg'} below. Then, by \cite{Gav02}*{Prop.~3.6}, $U_\hbar\mfb^\prime$ is a quantized formal series Hopf algebra, with semiclassical limit isomorphic as an algebra to the completion of the symmetric algebra $\mathsf{S}(\mfb)=\bigoplus_{n\in \N} \mathsf{S}^n(\mfb)$ with respect to its standard grading. In particular, this means that although $U_\hbar\mfb^\prime$ is not in general a topological Hopf algebra over $\C[\![\hbar]\!]$ in the sense of Section \ref{ssec:Pr-Top}, it is a topological Hopf algebra with respect to the $\mathbf{J}_\mfb$-adic topology,  where 
\begin{equation*}
\mathbf{J}_\mfb=\hbar U_\hbar\mfb\cap U_\hbar\mfb^\prime=\veps|_{U_\hbar\mfb^\prime}^{-1}(\hbar\C[\![\hbar]\!]).
\end{equation*}
The subspace $ U_\hbar\mfb^{\circ}\subset (U_\hbar\mfb^\prime)^\ast$ consisting of continuous linear forms with respect to this topology is then a quantization of $(\mfb^\ast, \delta_{\mfb^\ast})$. This is the \textit{quantized enveloping algebra dual} of $U_\hbar\mfb$. 
\begin{remark}\label{R:QDP}
Here we note that $U_\hbar\mfb^{\circ}$ can be equivalently defined as the $\hbar$-adic completion of the $\C[\![\hbar]\!]$-module
\begin{equation*}
(U_\hbar\mfb^\ast)^\times=\sum_{n\in \N} \hbar^{-n} \mfm^n_\mfb \subset \C(\!(\hbar)\!)\otimes_{\C[\![\hbar]\!]}U_\hbar\mfb^\ast,
\end{equation*}
where $\mfm_\mfb:=\{f\in U_\hbar\mfb^\ast: f(1)\in \hbar\C[\![\hbar]\!]\}$. That this produces a topological Hopf algebra which can be identified with $U_\hbar\mfb^{\circ}$ is a non-trivial result, which is part of the \textit{quantum duality principle}. This was first announced in \cite{DrQG}*{\S7}, and proven in detail in \cite{Gav02}; see Theorem 1.6 therein. We will not, however, need this equivalent formulation in the present paper. 
\end{remark}
 In our setting, $\mfb=\tplus=\mfg[t]$ is \textit{not} finite-dimensional, but rather an $\N$-graded Lie bialgebra $\mfb=\bigoplus_{n\in \N}\mfb_n$ with finite-dimensional graded components. As $\hYhg=U_\hbar\mfb$ is a homogeneous quantization of $\tplus$, the above construction remains valid, provided the notion of duality is adjusted so as to respect the underlying gradings.  In fact, one can replace $U_\hbar\mfb^\prime$ with an $\N$-graded topological Hopf algebra $\QFhYhg$ over $\C[\![\hbar]\!]$ of finite type, and $U_\hbar\mfb^\circ$ with the restricted dual of $\QFhYhg$, as defined in Section \ref{ssec:Pr-gr*}. The topological Hopf algebra $\QFhYhg$ and its $\C[\hbar]$-form $\QFYhg=(\QFhYhg)_\N$ (see Section \ref{ssec:Pr-grTop}) are the focus of the present section.
%
\subsection{The Drinfeld--Gavarini subalgebra}\label{ssec:Yhg'}

Let us define $\Delta^n$ for any $n\in \N$ by setting $\Delta^0=\veps$, $\Delta^1=\id=\id_{\Yhg}$ and 
\begin{equation*}
\Delta^n:=(\Delta\otimes \id^{\otimes (n-2)})\circ \Delta^{n-1}:\Yhg\to \Yhg^{\otimes n}
\end{equation*}
for all $n\geq 2$. We then define the $\C[\hbar]$-submodule $\QFYhg\subset \Yhg$ by 
\begin{equation*}
\QFYhg:=\{x\in \Yhg:(\id-\veps)^{\otimes n} \Delta^n(x)\in \hbar^n\Yhg^{\otimes n} \; \forall \; n\in \N\}.
\end{equation*}
By Lemma 3.2 and Proposition 3.5 of \cite{KasTu00} (see also \cite{Gav02}*{Prop.~2.6}, \cite{Gav07}*{Thm.~3.5} and \cite{FiTs19}*{Lem.~A.1}), $\QFYhg$ is a subalgebra of $\Yhg$ which is commutative modulo $\hbar$: 
\begin{equation}\label{almost-comm}
[x,y]\in \hbar \QFYhg \quad \forall\; x,y\in \QFYhg.
\end{equation}
As the structure maps $\id$, $\veps$ and $\Delta^n$ are graded, $\QFYhg$ inherits from $\Yhg$ the structure of an $\N$-graded algebra. We shall call $\QFYhg$ the \textit{Drinfeld--Gavarini subalgebra} of the Yangian $\Yhg$. Its algebraic structure has been described in detail by Tsymbaliuk and Weekes in Appendix A of \cite{FiTs19}, following the general results obtained in the works \cites{Gav02, Gav07} of Gavarini. In this subsection we review, and partially extend, this description.

Let $\mathsf{R}_\hbar(U(\tplus))$
denote the Rees algebra associated to the standard enveloping algebra filtration $\mbF_\bullet$ on $U(\tplus)$:
\begin{equation*}
\mathsf{R}_\hbar(U(\tplus))=\bigoplus_{n\in \N}\hbar^n \mathbf{F}_n(U(\tplus))\subset U(\tplus)[\hbar].
\end{equation*}
Consider now the symmetric algebra $\Symt\subset \msS(\tplus)[\hbar]$ on $\hbar\tplus$. Here $\hbar$ can be viewed as a gradation parameter associated to the standard $\N$-grading on the symmetric algebra $\msS(\tplus)$. Namely, $\Symt\cap  \hbar^n\msS(\tplus)$ is precisely the $n$-th symmetric power $\sSymt{n}=\hbar^n\msS^n(\tplus)$, and $\Symt\cong \msS(\tplus)$ as an $\N$-graded algebra. 
As $U(\tplus)$ is a filtered deformation of  $\Symt$ (that is, one has $\mathrm{gr}\,U(\tplus)\cong \Symt$), $\mathsf{R}_\hbar(U(\tplus))$ is a flat  
deformation of $\Symt$ over $\C[\hbar]$. Let 
\begin{equation*}
\msq:\mathsf{R}_\hbar(U(\tplus))\onto \Symt
\end{equation*}
be the natural quotient map, under the identification of $\mathsf{R}_\hbar(U(\tplus))/\hbar\mathsf{R}_\hbar(U(\tplus))$ with $ \mathrm{gr}\,U(\tplus)\cong \Symt$.
In what follows we shall be primarily interested in the \textit{loop gradings} on  $\mathsf{R}_\hbar(U(\tplus))$ and $\Symt$, inherited from the natural grading on $U(\tplus)[\hbar]$ compatible with the $\N$-grading on $\tplus$. Namely, one has 
\begin{equation*}
\deg (\hbar \mft_{{\scriptscriptstyle{+}},k}) = \deg (\hbar \mfg t^k) =k+1 \quad \forall\; k\in \N. 
\end{equation*}
We shall denote the $n$-th graded component of $\Symt$ by $\hSymt{n}$  so that 
\begin{equation*}
\Symt=\bigoplus_{n\in \N}\hSymt{n}. 
\end{equation*}
%
%

Recall from Section \ref{ssec:Yhg-coord} that $\upnu_\mathds{G}$ denotes the graded $\C[\hbar]$-module isomorphism 
\begin{equation*}
\upnu_\mathds{G}:\Yhg\iso U(\tplus)[\hbar]
\end{equation*}
defined in \eqref{nu_G}, which depends on a fixed total order on the set $\mathds{G}$. We equip  $\hbar\mathds{G}$ with the induced ordering, for which multiplication by $\hbar$ defines an isomorphism of ordered sets $\mathds{G}\iso \hbar\mathds{G}$, and we let $B(\hbar\mathds{G})\subset \Yhg$ denote the corresponding set of ordered monomials in $\hbar\mathds{G}$.
We further recall that $\mu:U(\tplus)[\hbar]\to \Yhg$ is the inverse of the specific choice $\upnu=\upnu_\mathds{G}$ defined in \eqref{nu-Yhg}. 

The following Proposition is a consequence of Proposition 3.3 of \cite{Gav07} in addition to Proposition A.2 and Theorem A.7 of \cite{FiTs19}; see also \cite{Gav02}*{\S3.5}. 
\begin{proposition}\label{P:QFSH}  Let $\dot{\upnu}_{\mathds{G}}$ denote the restriction of $\upnu_{\mathds{G}}$ to $\QFYhg$. Then:
\leavevmode 
\begin{enumerate}[font=\upshape]
\item\label{QFSH:1} $\QFYhg$ is an $\N$-graded Hopf subalgebra of $\Yhg$. 
%
 
%
\item\label{QFSH:2} $\dot{\upnu}_{\mathds{G}}$ is an isomorphism of $\N$-graded $\C[\hbar]$-modules 
\begin{equation*}
\dot{\upnu}_{\mathds{G}}:\QFYhg\iso \mathsf{R}_\hbar(U(\tplus))
\end{equation*}
\item\label{QFSH:3} $\QFYhg$ is generated as a $\C[\hbar]$-algebra by $\hbar \mu(\tplus)$ and has basis
$B(\hbar\mathds{G})$. 
\item\label{QFSH:4} The composition $\dot\msq:=\msq\circ \dot{\upnu}_{\mathds{G}}$ is an epimorphism of $\N$-graded algebras which descends to an isomorphism 
\begin{equation*}
\QFYhg/\hbar \QFYhg\iso \Symt.
\end{equation*}
\end{enumerate}
\end{proposition}
\begin{proof}[Proof of \eqref{QFSH:2} and \eqref{QFSH:3}]
 These statements are a minor modification of the statement of Theorem A.7 of \cite{FiTs19}. For the sake of completeness, let us recall the main ingredients, beginning with the proof that $\hbar\mathds{G}$ (and thus $\hbar \mu(\tplus)$) is contained in $\QFYhg$, given in Lemmas A.5 and A.6 of \cite{FiTs19}. 

For each $n\in \N$, $x\in \mfg$ and $i\in \mbI$, the formulas of Proposition \ref{P:Yhg-Hopf} imply that   
\begin{equation*}
\Delta^n(x)=\sum_{a=1}^n x^{(a)} \quad \text{ and }\quad \Delta^n(t_{i1})=\sum_{a=1}^n t_{i1}^{(a)}+\hbar\sum_{ a<b}\msr_i^{ab},
\end{equation*}
where $y^{(b)}=1^{\otimes (b-1)}\otimes y \otimes 1 ^{\otimes (n-b)}\in \Yhg^{\otimes n}$ for any $y\in \Yhg$ and, for each $a<b$, $B\mapsto B^{ab}$ is the algebra homomorphism $\Yhg^{\otimes 2}\to \Yhg^{\otimes n}$ given on simple tensors by $(x\otimes y)^{ab}=x^{(a)}y^{(b)}$. Since $(\id-\veps)$ projects $\Yhg$ onto $\mathrm{Ker}(\veps)$, it follows readily from these formulas that 
\begin{equation*}
\hbar \mfg \cup \{\hbar t_{i1}\}_{i\in \mbI}\subset \QFYhg.
\end{equation*}
We may now deduce that  $\hbar\mathds{G}\subset \QFYhg$ as follows. By \eqref{almost-comm} and the above, $\QFYhg$ is a $\mfg$-submodule of $\Yhg$ which is preserved by the operators $\{\mathrm{ad}(t_{i1})\}_{i\in \mbI}$. Hence, the elements
\begin{equation*}
\hbar x_{ik}^\pm=(\pm 1)^{k}\mathrm{ad}(\mathrm{T}_i)^k(\hbar x_i^\pm) \quad \text{ and }\quad \hbar h_{ik}=[x_i^+,\hbar x_{ik}^-]
\end{equation*}
necessarily belong to $\QFYhg$ for each $i\in \mbI$ and $k\in \N$, where  $\mathrm{T}_i=(2d_i)^{-1}t_{i1}$. As $\hbar x_{\alpha,k}^\pm$ belongs to the $\mfg$-submodule of $\Yhg$ generated by $\hbar x_{ik}^\pm$, we can conclude that $\hbar \mathds{G}\subset \QFYhg$. 

Since $\msR_\hbar(U(\tplus))$ has basis given by the set of ordered monomials in $\hbar q(\mathds{G})$, to complete the proof of both Parts \eqref{QFSH:2} and \eqref{QFSH:3}, it suffices to see that $B(\hbar \mathds{G})$ spans $\QFYhg$. This follows from the fact that $B(\mathds{G})$ is a basis of $\Yhg$ together with the crucial Lemma 3.3 of \cite{Gav02} (see also \cite{Etingof-Kazhdan-I}*{Lem.~4.12}). We refer the reader to the proof of \cite{FiTs19}*{Prop.~A.2} for complete details. 
\let\qed\relax
\end{proof}
\begin{proof}[Proof of \eqref{QFSH:1}]
We have already seen that $\QFYhg$ is an $\N$-graded subalgebra of $\Yhg$. That it is a Hopf subalgebra of $\Yhg$ is a special case of Proposition 3.3 of \cite{Gav07}, which passes to completions and makes use of a modification of a technical result for quantized formal series Hopf algebras established in \cite{Enriquez-Halbout-03}*{Prop.~2.1}.  It is worth pointing out that, in our specialized setting, it is possible to give a concise direct proof that $\QFYhg$ is a subcoalgebra of $\Yhg$ stable under the antipode $S$.

Indeed, by Proposition \ref{P:Yhg-Hopf}, one has $\Delta(\hbar x)\in \QFYhg^{\otimes 2}$ for $x\in \hbar \mfg\cup\{t_{i1}\}_{i\in \mbI}$. As $\QFYhg^{\otimes 2}$ is a $\mfg$-submodule of $\Yhg^{\otimes 2}$ stable under all operators $\mathrm{ad}(t_{i1}\otimes 1 + 1\otimes t_{i1})$, the inclusion  $\Delta(\hbar \mathds{G})\subset \QFYhg^{\otimes 2}$ will  hold provided  $\QFYhg^{\otimes 2}$ is preserved by the operators 
$\mathrm{ad}(\hbar \msr_i)$. This is itself a consequence of the fact that $\msr_i\in \mfg\otimes \mfg$, the inclusion $\hbar\mfg \subset \QFYhg$, that $\QFYhg^{\otimes 2}$ is a $\mfg$-submodule of $\Yhg^{\otimes 2}$, and the relation 
\begin{equation*}
[\hbar \msr_i,x\otimes y]=[\hbar \msr_i,x^{(1)}]y^{(2)}+x^{(1)}[\hbar \msr_i,y^{(2)}] \quad \forall\; x,y\in \QFYhg.
\end{equation*}
Hence, by Part \eqref{QFSH:3}, we can conclude that $\Delta(\QFYhg)\subset \QFYhg^{\otimes 2}$. A similar argument, using the formulas of Proposition \ref{P:Yhg-Hopf} and that $\QFYhg$ is a $\mfg$-submodule of $\QFYhg$ stable under the operators $\mathrm{ad}(S(t_{i1}))$, implies that $S(\QFYhg)\subset \QFYhg$. \let\qed\relax 
\end{proof}
\begin{proof}[Proof of \eqref{QFSH:4}] By Part \eqref{QFSH:2}, $\dot\msq$ is an $\N$-graded $\C[\hbar]$-linear epimorphism with kernel $\hbar \QFYhg$, and thus gives rise to an isomorphism of graded vector spaces $\ddot\msq:\QFYhg/\hbar\QFYhg\iso \Symt$. To conclude, it suffices to prove that $\ddot\msq$ is an algebra homomorphism. By \eqref{almost-comm}, $\QFYhg/\hbar \QFYhg$ is commutative, and so the linear map $\hbar\tplus\to \QFYhg/\hbar \QFYhg$ sending any $\hbar x\in \hbar \tplus$ to the image of $\hbar\mu(x)$ in $\QFYhg/\hbar \QFYhg$ uniquely extends to an algebra homomorphism $\msp:\Symt\to \QFYhg/\hbar \QFYhg$. Since $\ddot\msq\circ\msp=\id$, we can conclude that $\msq$ is the inverse of $\msp$, and thus an algebra homomorphism. 
\end{proof}

Recall from \eqref{hYhg-def} that $\hYhg$ denotes the $\hbar$-adically complete Yangian associated to $\mfg$, and let $\QFhYhg$ denote the $\hbar$-adic completion of the Hopf algebra $\QFYhg$:
\begin{equation*}
\QFhYhg:= \varprojlim_n(\QFYhg/\hbar^n\QFYhg).
\end{equation*}
As an immediate consequence of the above proposition and Corollary \ref{C:Top-Ngraded}, we obtain the following result.
\begin{corollary}\label{C:QFhYhg}
$\QFhYhg$ is a topologically free, $\N$-graded topological Hopf algebra over $\C[\![\hbar]\!]$ of finite type. Moreover: 
\begin{enumerate}[font=\upshape]
\item $\QFhYhg$ is a flat deformation of the $\N$-graded algebra $\Symt$ over $\C[\![\hbar]\!]$. In particular, $\QFhYhg\cong \Symt[\![\hbar]\!]$ as an $\N$-graded topological $\C[\![\hbar]\!]$-module. 
\item $\QFhYhg$ is a topological Hopf subalgebra of the completed Yangian $\hYhg$. 
\end{enumerate}
\end{corollary}
The statement that $\QFhYhg$ is of finite type reduces to the fact that the homogeneous components $\hSymt{n}$ of $\Symt$ are all finite-dimensional complex vector spaces; see below Corollary \ref{C:Vstar[[h]]}.

Henceforth, we will identify $\Symt$ with the semiclassical limit of $\QFYhg$ as an $\N$-graded Hopf algebra. 
We emphasize that this is a non-cocommutative Hopf algebra; in particular, it is not isomorphic to the standard symmetric Hopf algebra on $\hbar\tplus$, which we denote by $\mathbb{S}(\hbar\tplus)$. However, $\QFYhg$ and $\Symt$ are filtered deformations of $\mathsf{R}_\hbar( U(\tplus))$ and $\mathbb{S}(\hbar\tplus)$, respectively, as we now explain.
Consider the Hopf ideal  
\begin{equation}\label{J-def}
\mbJ:=\hbar \Yhg \cap \QFYhg \subset \QFYhg
\end{equation} 
Here we note that it follows from Part \eqref{QFSH:3} of the above proposition that, for each pair of positive integers $k,n\in \N_+$, one has  
\begin{equation}\label{J^k-tensor}
\hbar^k \Yhg^{\otimes n} \cap \QFYhg^{\otimes n} =\sum_{\scriptscriptstyle k_1+\ldots+k_n=k} \!\mathbf{J}^{k_1}\otimes\cdots \otimes \mathbf{J}^{k_n}.
\end{equation}
In particular, $\mbJ^k=\hbar^k \Yhg \cap \QFYhg$ for each $k\in \N$, and so the associated graded Hopf algebra $\mathrm{gr}_\mbJ(\QFYhg)$ with respect to the $\mbJ$-adic filtration embeds inside the associated graded Hopf algebra $\mathrm{gr}_\hbar(\Yhg)$ of $\Yhg$ with respect to the $\hbar$-adic filtration. Since $\Yhg$ is a torsion free Hopf algebra deformation of $U(\tplus)$ over $\C[\hbar]$, we have isomorphisms of graded Hopf algebras 
\begin{gather*}
\mathrm{gr}_\hbar(\Yhg)=\bigoplus_{n\in \N} \hbar^n \Yhg/\hbar^{n+1}\Yhg \cong \bigoplus_{n\in \N}\hbar^n U(\tplus)=U(\tplus)[\hbar] \\
\mathrm{gr}_\mbJ(\QFYhg)=\bigoplus_{n\in \N} \mbJ^n/\mbJ^{n+1}\cong \bigoplus_{n\in \N}\hbar^n \mathbf{F}_n(U(\tplus))=\mathsf{R}_\hbar( U(\tplus))
\end{gather*}
where the second isomorphism follows from the first and Proposition \ref{P:QFSH}. In fact, the module isomorphisms $\upnu_\mathds{G}$ and $\dot{\upnu}_\mathds{G}$ are filtered, and the above identifications can be realized as the associated graded maps $\mathrm{gr}(\upnu_\mathds{G})$ and $\mathrm{gr}(\dot{\upnu}_\mathds{G})$, respectively.

The image of $\mbJ$ under the quotient map $\dot\msq$ is the Hopf ideal 
\begin{equation}\label{Symt-J}
\mathbb{J}:=\bigoplus_{n>0} \sSymt{n} \subset \Symt
\end{equation}
where $\sSymt{n}$ denotes the $n$-th symmetric power of $\hbar\tplus$.  This is also a Hopf ideal in $\mathbb{S}(\hbar\tplus)$, and one has a canonical isomorphism $\mathrm{gr}_\mathbb{J}\mathbb{S}(\hbar\tplus)\cong \mathbb{S}(\hbar\tplus)$ of graded Hopf algebras. Since the elements of $\mu(\tplus)$ are primitive modulo $\hbar\Yhg^{\otimes 2}$, the subspace $\hbar\tplus=\mathbb{J}/\mathbb{J}^2$ of  $\mathrm{gr}_\mathbb{J}\Symt$ consists of primitive elements, and we can conclude that 
\begin{equation*}
\mathrm{gr}_\mathbb{J}\Symt \cong  \mathbb{S}(\hbar\tplus)\cong \mathrm{gr}_\mathbb{J}\mathbb{S}(\hbar\tplus)
\end{equation*} 
as graded Hopf algebras. 
\begin{remark}\label{R:QFJYhg}
Since $\mbJ^k=\hbar^k \Yhg\cap \QFYhg$ for each $k\in \N$, the closure of $\QFYhg$ in the quantized enveloping algebra $\hYhg$ coincides with the $\mbJ$-adic completion 
\begin{equation*}
\QFJYhg:=\varprojlim_n (\QFYhg/\mbJ^n)\subset \hYhg.
\end{equation*}
This is precisely the quantized formal series Hopf algebra $U_\hbar\mfb^\prime$ from Section \ref{ssec:QUE-dual} associated to $U_\hbar\mfb=\hYhg$. As $\mathbf{J}$ surjects onto the ideal $\mathbb{J}$ from \eqref{Symt-J}, the quotient map $\dot\msq$ induces an isomorphism of $\C[\![\hbar]\!]$-algebras
\begin{equation*}
\QFJYhg/\hbar \QFJYhg \iso \msS\widehat{(\hbar\tplus)}:=\prod_{n\in \N} \sSymt{n}\cong \varprojlim_n(\Symt/\mathbb{J}^n).
\end{equation*}
\end{remark}

\subsection{Triangular decomposition}\label{ssec:QFYhg-TD}
Proposition \ref{P:QFSH} implies that the triangular decomposition of $\Yhg$, reviewed in Section \ref{ssec:Yhg-coord}, induces a triangular decomposition on $\QFYhg$, with the $\N$-graded $\C[\hbar]$-algebras 
\begin{equation*}
\dYh{\pm}:=Y_\hbar^\pm(\mfg)\cap \QFYhg \quad \text{ and }\quad \dYh{0}:=Y_\hbar^0(\mfg)\cap \QFYhg. 
\end{equation*}
playing the roles of $Y_\hbar^\pm(\mfg)$ and $Y_\hbar^0(\mfg)$, respectively. In this subsection we spell this out explicitly. Let us set
\begin{equation*}
\tplus^\pm:=\mfn_\mp[t]\subset \tplus \quad \text{ and }\quad \tplus^0:=\mfh[t]\subset \tplus
\end{equation*}
and recall from Section \ref{ssec:Yhg*-TD} that  $\upnu_\pm$ and $\upxi$ are the isomorphisms  
\begin{equation*}
\upnu_\pm: Y_\hbar^\pm(\mfg)\iso U(\tplus^\pm)[\hbar] \quad \text{ and }\quad \upxi:Y_\hbar^0(\mfg)\iso \msS(\tplus^0)[\hbar]
\end{equation*}
defined above \eqref{nu-Yhg} and in the statement of Proposition \ref{P:TriDec}, respectively. We further 
recall from \eqref{nu-beta} that, for each $\beta\in Q_+$, $\nu(\beta)\in \N$ is defined by  
\begin{equation*}
\nu(\beta)=\min\{k\in \N: \exists\, \beta_1,\ldots,\beta_k\in \Root^+\; \text{ with }\; \beta=\beta_1+\cdots +\beta_k\}.
\end{equation*}

The following corollary provides an analogue of Proposition \ref{P:QFSH} for the subalgebras $\dYh{\pm}$ and $\dYh{0}$. 
\begin{corollary}\label{C:QFYhg-chi} 
Let $\dot\upnu_\pm$ and $\dot\upxi$ denote the restrictions of $\upnu_\pm$ and $\upxi$ to $\dYh{\pm}$ and $\dYh{0}$, respectively. Then: 
\begin{enumerate}[font=\upshape]
\item\label{QFYhg-chi:1} $\dot\upnu_\pm$ is an isomorphism of $\N$-graded $\C[\hbar]$-modules 
\begin{equation*}
\dot\upnu_\pm:\dYh{\pm}\iso \msR_\hbar(U(\tplus^{\pm}))\subset U(\tplus^\pm)[\hbar].
\end{equation*}
%
\item\label{QFYhg-chi:2} $\dot\upxi$ is an isomorphism of $\N$-graded $\C[\hbar]$-algebras 
\begin{equation*}
\dot\upxi:\dYh{0}\iso \msS(\hbar\tplus^0)[\hbar]\subset \msS(\tplus^0)[\hbar].
\end{equation*}
\item\label{QFYhg-chi:3} $\dYh{\pm}$ is a torsion free, $\N$-graded $\C[\hbar]$-algebra deformation of $\msS(\hbar\tplus^\pm)$. In particular, there is an isomorphism of graded $\C[\hbar]$-modules
\begin{equation*}
\dYh{\pm}\cong \msS(\hbar\tplus^\pm)[\hbar].
\end{equation*}
\item\label{QFYhg-chi:4}  $\dYh{\pm}$ is a $Q$-graded subalgebra of $\dYh{\pm}$ with  
\begin{equation*}
\dYh{\pm}_{\pm \beta} \subset \mbJ^{\nu(\beta)}\subset \hbar^{\nu(\beta)}Y_\hbar^\pm(\mfg)_{\pm \beta} \quad \forall \; \beta \in Q_+. 
\end{equation*}
\end{enumerate}
\end{corollary}
We note that each of these results follows readily from Propositions \ref{P:TriDec} and \ref{P:QFSH}, in addition to Corollary \ref{C:Top-Ngraded} in the case of Part \eqref{QFYhg-chi:3}. We leave the details as an exercise to the interested reader. As a consequence of this corollary and Proposition \ref{P:QFSH}, we obtain the following analogue of Part \eqref{TriDec:4} from Proposition \ref{P:TriDec}.

\begin{corollary}\label{C:dYhg-TD}
The multiplication map 
\begin{equation*}
\dot{m}:\dYh{+}\otimes \dYh{0}\otimes \dYh{-}\to \QFYhg
\end{equation*}
is an isomorphism of $\N$-graded $\C[\hbar]$-modules.
\end{corollary}
\begin{remark}\label{R:QFhYhg-TD}
As in Remark \ref{R:hYhg-TD}, the above results can easily be lifted to the $\hbar$-adic setting using Corollary \ref{C:Top-Ngraded}; see also Corollary \ref{C:QFhYhg}. We especially note that, for each choice of the symbol $\chi$, the $\hbar$-adic completion 
\begin{equation*}
\dot{\msY}_\hbar^\chi\mfg:=\varprojlim_n \dYh{\chi}/\hbar^n \dYh{\chi} 
\end{equation*}
is a topologically free $\N$-graded $\C[\![\hbar]\!]$-algebra of finite type, which provides a flat deformation of the symmetric algebra $\msS(\hbar\tplus^\chi)$ over $\C[\![\hbar]\!]$. In addition, $\dot{\msY}_\hbar^\chi\mfg$ embeds inside the completed Yangian $\hYhg$, and the multplication map induces an isomorphism of $\N$-graded $\C[\![\hbar]\!]$-modules 
\begin{equation*}
\dot{m}:\dot{\msY}_\hbar^{+}\mfg\otimes \dot{\msY}_\hbar^{0}\mfg\otimes \dot{\msY}_\hbar^{-}\mfg\iso \QFhYhg,
\end{equation*}
where $\otimes$ is now the topological tensor product $\wh{\otimes}$ over $\C[\![\hbar]\!]$; see Remark \ref{R:otimes}. 
\end{remark}
%
\subsection{The adjoint and coadjoint actions}\label{ssec:ad-coad} We now prove two lemmas concerning the Drinfeld--Gavarini subalgebra $\QFYhg$ which will play an important role in the main results and constructions of Sections \ref{sec:DYhg-quant} and \ref{sec:DYhg-double}. The first of these, Lemma \ref{L:adjoint}, will be used to construct the quantum double of the Yangian in Section \ref{ssec:D(Yhg)}.

Let $\lad:\Yhg\otimes \Yhg\to \Yhg$ and $\rcoad:\Yhg\to \Yhg\otimes \Yhg$ denote the left adjoint action of $\Yhg$ on itself, and the right adjoint coaction of $\Yhg$ on itself, respectively. That is, 
\begin{gather*}
\lad= m^3\circ (\id^{\otimes 2}\otimes S)\circ (2\,3)\circ (\Delta \otimes \id)  \\
\rcoad=(1\otimes m)\circ (1\,2)\circ  (S\otimes \id^{\otimes 2})\circ \Delta^3
\end{gather*}
The $\hbar$-adic analogue of the below result, for a quantized enveloping algebra $U_\hbar\mfb$ with $\QFYhg$ replaced by $U_\hbar\mfb^\prime$, was established in Propositions 4.3 and 4.4 of \cite{Andrea-Valerio-18}; see also Proposition  A.5 therein. 
\begin{lemma}\label{L:adjoint}
One has 
\begin{equation*}
\lad(\Yhg\otimes \QFYhg)\subset \QFYhg \quad \text{ and }\quad \rcoad(\Yhg)\subset \Yhg\otimes \QFYhg.
\end{equation*}
\end{lemma}
\begin{proof}
Since $\lad$ makes $\Yhg$ a left module, to prove the first inclusion it suffices to show that  $x\cdot y:=\lad(x\otimes y)\in \QFYhg$ for all $x\in \mfg\cup \{t_{i1}\}_{i\in \mbI}$ and $y\in \QFYhg$. Since 
\begin{gather*}
x\cdot y=[x,y] \quad \text{ and } \quad t_{i1}\cdot y=[t_{i1},y]+m([y^{(1)},\hbar \msr_i]) \quad \forall\; x\in \mfg,\, i\in \mbI,\, y\in \QFYhg,
\end{gather*}
this follows from the observation that $\hbar \msr_i\in \mfg\otimes \hbar\mfg\subset \mfg\otimes \QFYhg$ and the fact that $\QFYhg$ is a $\mfg$-submodule of $\Yhg$ stable under $\mathrm{ad}(t_{i1})$ for all $i\in \mbI$. 

As for the second inclusion, note that $\rcoad$ is a homomorphism of right modules: 
\begin{equation}\label{adco-hom}
\rcoad\circ m = \gamma \circ (\rcoad\otimes \id),
\end{equation}
where $\gamma:\Yhg^{\otimes 2}\otimes \Yhg\to \Yhg^{\otimes 2}$ is the right action defined by 
\begin{equation*}
\gamma=m\otimes m^2 \circ (\id^{\otimes 2} \otimes S\otimes \id^{\otimes 2})\circ (2\,4)\circ (\id^{\otimes 2} \otimes \Delta^3).
\end{equation*}
This right action preserves $\Yhg\otimes \QFYhg\subset \Yhg^{\otimes 2}$. Indeed, if $y\in \Yhg$, $z\in \QFYhg$ and $x\in \mfg$,  we have 
\begin{equation*}
\gamma(y\otimes z\otimes x)=y\otimes [z,x]+yx\otimes z \in \Yhg\otimes \QFYhg,
\end{equation*}
and replacing $x$ by $t_{i1}$, for any $i\in \mbI$, yields instead 
\begin{equation*}
\gamma(y\otimes z\otimes t_{i1})=y\otimes ([z,t_{i1}]+ m([z^{(2)},\hbar\msr_i]))+yt_{i1}\otimes z +y^{(1)}(z^{(2)}\hbar\msr_i-\hbar \msr_i^{21} z^{(2)})
\end{equation*}
which again belongs to $\Yhg\otimes \QFYhg$ as $\hbar \msr_i^{21}, \hbar \msr_i\in (\hbar\mfg\otimes \mfg)\cap (\mfg\otimes \hbar\mfg)$ and $\QFYhg$ is an $\mathrm{ad}(t_{i1})$-stable $\mfg$-submodule of $\Yhg$.

Since $\Yhg\otimes \QFYhg$ is a  right submodule of $\Yhg^{\otimes 2}$ under $\gamma$, the condition \eqref{adco-hom} guarantees that $\rcoad(\Yhg)\subset \Yhg\otimes \QFYhg$ will hold provided it holds on $\mfg\cup\{t_{i1}\}_{i\in \mbI}$. To conclude, it suffices to note that, for each $x\in \mfg$ and $i\in \mbI$, one has 
\begin{gather*}
\rcoad(x)=x\otimes 1\in\Yhg\otimes \QFYhg, \\
\rcoad(t_{i1})=t_{i1}\otimes 1+\hbar(\msr_i-\msr_i^{21})\in \Yhg\otimes \QFYhg. \qedhere
\end{gather*}
\end{proof}

%
\subsection{The \texorpdfstring{$R$}{R}-matrix}
The second lemma we will need concerns the universal $R$-matrix $\mcR(z)$ of the Yangian, and will play a crucial role in identifying the dual Yangian as a subalgebra of the Yangian double $\DYhg$ in Section \ref{sec:DYhg-quant}. In what follows, all notation is as in Section \ref{ssec:Yhg-R}. 
\begin{lemma}\label{L:R-DG}
The factors $\mcR^\pm(z)$ and $\mcR^0(z)$ of the universal $R$-matrix $\mcR(z)$ have coefficients in $(\QFYhg\otimes \Yhg)\cap (\Yhg\cap \QFYhg)$. Consequently, 
\begin{equation*}
\mcR(z)\in (\QFYhg\otimes \Yhg)[\![z^{-1}]\!] \cap (\Yhg\otimes \QFYhg)[\![z^{-1}]\!].
\end{equation*}
\end{lemma}
\begin{proof} Since $\QFYhg$ is preserved by $\mathrm{ad}(t_{i1})$ for any $i\in \mbI$ and, for each $\alpha\in \Root^+$, the simple tensor 
$\hbar x_\alpha^-\otimes x_\alpha^+=x_\alpha^-\otimes \hbar x_\alpha^+$ belongs to the intersection of 
$\QFYhg\otimes \Yhg$ and $\Yhg\otimes \QFYhg$, it follows from \eqref{R-recur} and induction on $\mathrm{ht}(\beta)$ that 
\begin{equation*}
\mcR^-_\beta(z)\in (\QFYhg\otimes \Yhg)[\![z^{-1}]\!] \cap (\Yhg\otimes \QFYhg)[\![z^{-1}]\!] \quad \forall\; \beta\in Q_+.
\end{equation*}
Consequently, both $\mcR^-(z)$ and $\mcR^+(z)=\mcR^-_{21}(-z)^{-1}$ belong to the intersection of $(\QFYhg\otimes \Yhg)[\![z^{-1}]\!]$ and $(\Yhg\otimes \QFYhg)[\![z^{-1}]\!]$. It is thus enough to prove that 
\begin{equation*}
\mcR^0(z)\in (\QFYhg\otimes \Yhg)[\![z^{-1}]\!] \cap (\Yhg\otimes \QFYhg)[\![z^{-1}]\!].
\end{equation*}
Recall from \eqref{J-def} that $\mbJ=\QFYhg\cap \hbar\Yhg$. 
Since $\hbar h_i(u)\in \mbJ[\![u^{-1}]\!]$, the logarithm $t_i(u)=\log(1+\hbar h_i(u))$ and its Borel transform $B_i(u)$ both have coefficients in $\mbJ$. Therefore, we have 
\begin{equation*}
\hbar^{-1} B_i(u)\otimes B_j(-u)\in (\QFYhg\otimes \Yhg)[\![u]\!] \cap (\Yhg\otimes \QFYhg)[\![u]\!]\quad \forall\; i,j\in\mbI.
\end{equation*}
Since $g(z/2\kappa\hbar)$ is divisible by $\hbar$, it follows from \eqref{logR0} that the logarithm $\mcS(z)$ of $\mcR^0(z)$, and thus $\mcR^0(z)$ itself, has coefficients belonging to the intersection of $\QFYhg\otimes \Yhg$ and  $\Yhg\otimes \QFYhg$. \qedhere 
\end{proof}
%
%
\section{The dual Yangian}\label{sec:Yhg*}

With Section \ref{sec:QFYhg} at our disposal, we are now in a position to introduce the dual Yangian $\Yhgstar$. After defining $\Yhgstar$ and spelling out some of its basic properties in Sections \ref{ssec:Yhg*-Hopf} and \ref{ssec:Yhg*-Chev}, we prove in Sections \ref{ssec:Yhg*-scl} and \ref{ssec:Yhg*-quant} that it is a homogeneous quantization of the Lie bialgebra $(\tminus,\delta_{\scriptscriptstyle{-}})$ defined in Section \ref{ssec:Pr-Manin}. We conclude in Section \ref{ssec:Yhg*-TD} by identifying a family of generators for $\Yhgstar$ and establishing a triangular decomposition.
%
\subsection{The dual Yangian \texorpdfstring{$\Yhgstar$}{Yhg*} }\label{ssec:Yhg*-Hopf}
By Corollary \ref{C:QFhYhg}, the $\hbar$-adic completion $\QFhYhg$ of $\QFYhg$ is an $\N$-graded topological Hopf algebra of finite type. We may thus apply the machinery from Section \ref{ssec:Pr-gr*} to obtain a $\Z$-graded topological Hopf structure on its restricted dual. 

In more detail, by Proposition \ref{P:graded-tensor} the restricted dual $\Yhgstar$, as defined in Definition \ref{D:M-star}, is a $\Z$-graded topological Hopf algebra over $\C[\![\hbar]\!]$ with product, unit, coproduct, counit and antipode given by the transposes 
\begin{gather*}
\Delta^t: \Yhgstar \otimes \Yhgstar \to \Yhgstar, \quad \veps^t: \C[\![\hbar]\!]\to \Yhgstar,\\
m^t:\Yhgstar \to \Yhgstar \otimes \Yhgstar, \quad \iota^t:\Yhgstar\to\C[\![\hbar]\!],\quad S^t:\Yhgstar\to \Yhgstar,
\end{gather*}
respectively, where $\Delta$, $\veps$, $m$, $\iota$ and $S$ are the coproduct, counit, product, unit and antipode of $\QFhYhg$, respectively, and $\otimes$ is the topological tensor product over $\C[\![\hbar]\!]$; see Remark \ref{R:otimes}.
\begin{definition}\label{D:Yhg*}
The topological Hopf algebra  $\Yhgstar$ introduced above is called \textit{the dual Yangian} of $\mfg$. 
\end{definition}
Explicitly,  $\Yhgstar$ is the subspace of the $\C[\![\hbar]\!]$-linear dual $\QFhYhg{\vphantom{)}}^\ast$ consisting of those $f:\QFhYhg\to \C[\![\hbar]\!]$ which are continuous with respect to the gradation topology on $\QFhYhg$. It can be recovered as the $\hbar$-adic completion of the $\Z$-graded $\C[\hbar]$-algebra 
\begin{equation*}
(\Yhgstar)_\Z:=\bigoplus_{a\in \Z}\Yhgstar_a\cong \QFYhg{\vphantom{)}}^\star,
\end{equation*}
where $\QFYhg{\vphantom{)}}^\star$ is the graded dual of $\QFYhg$ over $\C[\hbar]$ which, as explained in Remark \ref{R:M_N*}, coincides with $(\Yhgstar)_\Z$ under the natural identification of
$\QFhYhg{\vphantom{)}}^\ast$ with $\Hom_{\C[\hbar]}(\QFYhg,\C[\![\hbar]\!])$. Here we have set 
\begin{equation*}
\Yhgstar_a:=\Hom_{\C[\![\hbar]\!]}^a(\QFhYhg,\C[\![\hbar]\!])\cong \Hom_{\C[\hbar]}^a(\QFYhg,\C[\hbar]) \quad \forall\; a\in \Z. 
\end{equation*}
%

By Corollary \ref{C:Vstar[[h]]}, $\Yhgstar$ is a flat deformation of the graded Hopf algebra 
\begin{equation*} 
 \Yhgstar/\hbar\Yhgstar\cong\Symtstar.
\end{equation*}
In particular, the dual Yangian $\Yhgstar$ is isomorphic to  $\Symtstar[\![\hbar]\!]$ as a $\Z$-graded topological $\C[\![\hbar]\!]$-module. Here we recall that $\Symtstar$ is the graded dual of the $\N$-graded Hopf algebra $\Symt$ over $\C$. As a vector space, one has
\begin{equation*}
\Symtstar=\bigoplus_{n\in \N}\Symtstar_{-n}
\end{equation*}
where  $\Symtstar_{-n}=\hSymt{n}{\vphantom{)}}^\ast\subset \Hom_\C(\Symt,\C)$. We shall identify this Hopf algebra with the enveloping algebra of the Lie algebra $\tminus=t^{-1}\mfg[t^{-1}]$ in Section \ref{ssec:Yhg*-scl} below.

Now let us make a few comments which concern the restricted dual $\hYhg{\vphantom{)}}^\star$ of the full, completed, Yangian $\hYhg$. Since $\hYhg$ is an $\N$-graded topological Hopf algebra over $\C[\![\hbar]\!]$, the formalism of Section \ref{ssec:Pr-gr*} implies that $\Yhgstar$ is a $\Z$-graded topological $\C[\![\hbar]\!]$-algebra, which provides a flat deformation of the algebra $U(\tplus){\vphantom{)}}^\star$ over $\C[\![\hbar]\!]$. In particular, there is an isomorphism of $\Z$-graded topological $\C[\![\hbar]\!]$-modules 
\begin{equation*}
\hYhg{\vphantom{)}}^\star\cong U(\tplus){\vphantom{)}}^\star[\![\hbar]\!]. 
\end{equation*}
However, $\hYhg{\vphantom{)}}^\star$ is not itself a topological Hopf algebra over $\C[\![\hbar]\!]$ with respect to the $\hbar$-adic topology. It is, however, naturally a subalgebra of the topological Hopf algebra $\Yhgstar$. This is made explicit by the below result.
\pagebreak 
\begin{proposition}\label{P:hYhg*->QFhYhg*}
The $\C[\![\hbar]\!]$-linear map 
\begin{equation*}
\hYhg{\vphantom{)}}^\star \to \Yhgstar, \quad f\mapsto f|_{\QFhYhg} \quad \forall \quad f\in \hYhg{\vphantom{)}}^\star,
\end{equation*}
is an injective homomorphism of $\Z$-graded topological $\C[\![\hbar]\!]$-algebras. 
\end{proposition}
\begin{proof}
Since $\QFhYhg$ is a  $\Z$-graded Hopf subalgebra of $\hYhg$,  the map $f\mapsto f|_{\QFhYhg}$ respects the underlying $\Z$-graded algebra structures. Moreover, if $f\in  \hYhg{\vphantom{)}}^\star$ vanishes on the basis $B(\hbar \mathds{G})$ from Proposition \ref{P:QFSH}, then it vanishes on the basis $B(\mathds{G})$ of $\QFYhg$ as $\C[\![\hbar]\!]$ is torsion free. This yields the injectivity. \qedhere
\end{proof}
\begin{remark}
Let $\mbJ$ be as in \eqref{J-def}. Then, since  $\mbJ$ satisfies 
\begin{equation*}
\mbJ^n\subset \bigoplus_{k\geq n}\QFYhg_k \quad \forall \; n\in \N, 
\end{equation*}
  every $f\in \QFhYhg{\vphantom{)}}^\star$ is automatically continuous with respect to the $\mbJ$-adic topology on $\QFYhg$, and so uniquely extends to an element of the associated topological dual $\hYhg{\vphantom{)}}^{\circ}$ to $\QFJYhg$; see Section \ref{ssec:QUE-dual} and Remark \ref{R:QFJYhg}. In this sense, the notion of duality considered here is compatible with that for a general quantized enveloping algebra $U_\hbar\mfb$ outlined in Section \ref{ssec:QUE-dual}, despite the fact that we did not need to leave the category of topological Hopf algebras over $\C[\![\hbar]\!]$ to define $\Yhgstar$. 
\end{remark}

\begin{remark}
An alternative description of $\QFhYhg{\vphantom{)}}^\star$ using the formalism of Remark \ref{R:QDP} can be found in \cite{Etingof-Kazhdan-III}*{\S3.1}. 
\end{remark}
%
%

%
\subsection{Chevalley involution and \texorpdfstring{$\mfg$}{g}-action}\label{ssec:Yhg*-Chev}
We now make a handful of simple observations which will play an important role in the remainder of this article. In what follows, we shall freely make use of the fact that $\Yhgstar$ can be naturally viewed as a subspace of $\Hom_{\C[\hbar]}(\QFYhg,\C[\![\hbar]\!])$.

Since the Chevalley involution $\omega$ defined in Lemma \ref{L:Chev} is an anti-automorphism of the graded Hopf algebra $\Yhg$, it follows from the definition of $\QFYhg$ (or, alternatively, from Proposition \ref{P:QFSH}) that it restricts to an anti-automorphism of the graded Hopf algebra $\QFYhg$, which we again denote by $\omega$. Consequently, the transpose $\omega^t$ of $\omega$, uniquely determined by
\begin{equation*}
\omega^t(f)(x)=f(\omega(x)) \quad \forall\; f\in \Yhgstar \; \text{ and }\; x\in \QFYhg,
\end{equation*}
is an involutive Hopf algebra anti-automorphism of $\Yhgstar$. We call $\omega^t$ the \textit{Chevalley involution} of $\Yhgstar$. 

Next, recall that the adjoint action of $\mfg$ on $\Yhg$ preserves the Drinfeld--Gavarini subalgebra $\QFYhg$. Since each graded component $\QFYhg_k$ is also a submodule, the restricted dual $\Yhgstar$ is a $\mfg$-module equipped with the coadjoint action
\begin{equation*}
(x\cdot f)(y)=f(S(x)\cdot y)\quad \forall \; x\in U(\mfg),\; y\in \QFYhg\; \text{ and }\; f\in \QFhYhg^\star.
\end{equation*} 

We now introduce a topological $Q$-grading on $\Yhgstar$ compatible with the above action which is analogous to that obtained for $\DYhg$ in Section \ref{ssec:DYhg-root}. 
For each $\beta\in Q$, define the closed $\C[\![\hbar]\!]$-submodule $\Yhgstar_\beta\subset \Yhgstar$ by
\begin{equation*}
\Yhgstar_\beta=\{f\in \Yhgstar: f(\QFYhg_\alpha)\subset \C[\![\hbar]\!]_{\alpha+\beta} \;\forall \;\alpha\in Q\},
\end{equation*}
where $\C[\![\hbar]\!]_\alpha$ is $\C[\![\hbar]\!]$ if $\alpha=0$ and is $\{0\}$ otherwise. It is easy to see that $\Yhgstar_\beta$ is just the $\beta$-weight space of $\Yhgstar$ with respect to the $\mfg$-module structure introduced above. That is, one has 
\begin{equation*}
\Yhgstar_\beta=\{f\in \Yhgstar:\, h\cdot f= \beta(h)f \quad \forall\; h\in \mfh\}. 
\end{equation*}
As $\QFYhg$ is $Q$-graded as a Hopf algebra, the direct sum 
\begin{equation*}
\Yhgstar_Q=\bigoplus_{\beta\in Q}\Yhgstar_\beta \subset \Yhgstar
\end{equation*}
is a $Q$-graded $\C[\![\hbar]\!]$-subalgebra of $\Yhgstar$. Moreover, the counit, coproduct and antipode of $\Yhgstar$ are all $Q$-graded, degree zero, maps. It is not difficult to prove that $\Yhgstar_Q$ is a dense subalgebra of $\Yhgstar$ whose subspace topology coincides with its $\hbar$-adic topology. Hence, we obtain the following result:
\begin{corollary}\label{C:Yhg*-Qgrad}
$\Yhgstar$ is $Q$-graded as a topological Hopf algebra over $\C[\![\hbar]\!]$. 
\end{corollary}  
%
%
%

%
\subsection{Classical duality}\label{ssec:Yhg*-scl}

We now wish to identify the graded dual $\Symtstar$ of the $\N$-graded Hopf algebra $\Symt$ with the enveloping algebra $U(\tminus)$, where we recall that $\tminus=t^{-1}\mfg[t^{-1}]$.

 To formulate this result optimally, we must first give a few preliminary remarks. To begin, we note that the semiclassical limit of the Chevalley involution $\omega^t$ of $\Yhgstar$ coincides, by definition, with the tranpose $\dot{\omega}^t$ of the automorphism $\dot{\omega}$ of $\Symt$ given on $\hbar\tplus$ by 
\begin{equation*}
\dot{\omega}(\hbar x)=\hbar \bar{\omega}(x) \quad \forall \; x\in \tplus,
\end{equation*}
where $\bar{\omega}$ is as in \eqref{bar-omega}. Similarly, the coadjoint action of $\mfg$ on $\Yhgstar$ introduced in Section \ref{ssec:Yhg*-Chev} specializes to an action of $\mfg$ on $\Symtstar$. By definition, this action is dual to that of $\mfg$ on $\Symt$ inherited from the adjoint action of $\mfg$ on $\QFYhg$. 

On the other hand, the Chevalley involution $\bar\omega$ of $U(\mfg[t^{\pm 1}])$, defined as in \eqref{bar-omega} with $r$ taking values in $\Z$, and the adjoint action of $\mfg$ on $U(\mfg[t^{\pm 1}])$ both preserve $U(\tminus)$. The resulting involution and $\mfg$-module structure on $U(\tminus)$ will be compared to those of $\Symtstar$ described in the previous paragraph in \eqref{Pdual:3} of Proposition \ref{P:dual}. 

Consider now the standard symmetric algebra grading
\begin{equation*}
\Symt=\bigoplus_{n\in \N}\sSymt{n},
\end{equation*}
where $\sSymt{n}$ is the $n$-th symmetric power of $\hbar\tplus$, as in \eqref{Symt-J}. Since $\sSymt{1}=\hbar\tplus$, every linear functional $f$ in $(\hbar\tplus)^\ast$ trivially extends to an element of $\Symt{\vphantom{)}}^\ast$ satisfying $f(\sSymt{n})=0$ for all $n\neq 1$, which is contained in $\Symtstar$ provided $f\in (\hbar\tplus)^\star$. That is, we have $(\hbar\tplus)^\star\subset \Symtstar$. In addition, we have a homogeneous, degree zero, isomorphism of graded vector spaces 
\begin{equation*}
 \mathsf{Res}_{\scriptscriptstyle-}^\hbar:\tminus \iso (\hbar\tplus)^\star, \quad \mathsf{Res}_{\scriptscriptstyle-}^\hbar(x)(\hbar y)=\mathsf{Res}_{\scriptscriptstyle-}(x)(y)=\langle x, y\rangle \quad \forall\; x\in \tminus, \, y\in \tplus,
\end{equation*}
where $\mathsf{Res}_{\scriptscriptstyle-}$ and $\langle \,, \,\rangle$ are as defined in Section \ref{ssec:Pr-Manin}. 
With the above at our disposal, we are now prepared to identify $\Symtstar $ and $U(\tminus)$. 
\begin{proposition}\label{P:dual} 
The restricted dual $\Symtstar $ has the following properties:
\begin{enumerate}[font=\upshape]
\item\label{Pdual:1} $(\hbar\tplus)^\star$ is the Lie algebra of primitive elements in $\Symtstar $, with bracket
\begin{equation*}
[f,g]=(f\otimes g) \circ \hbar\delta_{\scriptscriptstyle +} \quad \forall \quad f,g\in (\hbar\tplus)^\star.
\end{equation*}
\item\label{Pdual:2} $\mathsf{Res}_{\scriptscriptstyle -}^\hbar$ uniquely extends to an isomorphism of graded Hopf algebras 
\begin{equation*}
\varphi:U(\tminus)\iso \Symtstar
\end{equation*}
\item\label{Pdual:3} $\varphi$ is a $\mfg$-module intertwiner commuting with Chevalley involutions. 
\end{enumerate}
\end{proposition}
Parts \eqref{Pdual:1} and \eqref{Pdual:2} of this proposition can be viewed as a variant of \cite{KWWY}*{Cor.~3.4} applied to a restricted version of the setting in \cite{KWWY}*{\S3E}. They can be seen as a consequence of a graded generalization of (a special case of) Theorem 4.8 in \cite{Gav07}. It is not difficult to prove this variant directly using a fairly general argument, as we illustrate below. 

\begin{proof}[Proof of \eqref{Pdual:1}] 
An element $f\in \Symtstar$ is primitive 
precisely when it satisfies
\begin{equation}\label{prim}
f(xy)=f(x)\veps(y)+\veps(x)f(y) \quad \forall \;x,y\in \Symt.
\end{equation}
Since the counit $\veps$ of $\Symt$ vanishes on $\mathbb{J}=\bigoplus_{n>0}\sSymt{n}$, it follows readily that $f$ must vanish on $\mathbb{J}^2=\bigoplus_{n>1}\sSymt{n}$. Moreover, the above condition gives $f(1)=2f(1)$, and hence $f$ vanishes on $\C$. It follows that $f\in \sSymt{1}{\vphantom{(}}^\star=(\hbar\tplus)^\star$. 

Conversely, if $f\in (\hbar\tplus)^\star$, then $\Delta(f)$ vanishes on $\sSymt{n}\otimes \sSymt{m}$ unless $n+m=1$, and on $\hbar\tplus\otimes \C$ and $\C\otimes \hbar\tplus$ the identity \eqref{prim} trivially holds. This completes the proof that the Lie algebra $\mathrm{Prim}\Symtstar$ coincides with $(\hbar\tplus)^\star$ as a vector space.  Let us now prove that its bracket is given by 
\begin{equation*}
[f,g]=(f\otimes g) \circ \hbar\delta_{\scriptscriptstyle +} \quad \forall \quad f,g\in (\hbar\tplus)^\star.
\end{equation*}
Since $[f,g]\in (\hbar\tplus)^\star$, it is enough to establish this equality on $\hbar\tplus$. By definition, we have
\begin{equation*}
[f,g](\hbar x)=(fg-gf)(\hbar x)=(f\otimes g)(\Delta-\op{\Delta})(\hbar x) \quad \forall\;\hbar x\in \hbar\tplus.
\end{equation*}
It thus suffices to prove that $(\Delta-\op{\Delta})|_{\hbar\tplus}-\hbar\delta_{\scriptscriptstyle +}$ has image in $\mathrm{Ker}(f\otimes g)$. 
As $\Yhg$ is a quantization of $(\tplus,\delta_{\scriptscriptstyle +})$, we have 
\begin{equation*}
(\Delta_{\Yhg}-\op{\Delta}_{\Yhg})(\hbar \mu(x))-\hbar^2(\mu\otimes \mu)\delta_{\scriptscriptstyle +}(x)\in \hbar^3\Yhg^{\otimes 2}\cap \QFYhg^{\otimes 2},
\end{equation*}
where we recall that $\mu=\upnu^{-1}$, with $\upnu$ as in \eqref{nu-Yhg}. 
Applying $\dot\msq\otimes \dot\msq$ and taking note of \eqref{J^k-tensor}, we obtain 
\begin{equation}\label{delta-ker}
(\Delta-\op{\Delta})(\hbar x)-\hbar\delta_{\scriptscriptstyle +}(\hbar x)\in \sum_{n+m=3} \mathbb{J}^n\otimes \mathbb{J}^m\subset \mathrm{Ker}(f\otimes g),
\end{equation}
which completes the proof of \eqref{Pdual:1}.
\end{proof}

\begin{proof}[Proof of \eqref{Pdual:2}] \let\qed\relax 
Consider now \eqref{Pdual:2}. By Part \eqref{Pdual:1}, $\mathsf{Res}_{\scriptscriptstyle-}^\hbar$ is an isomorphism of graded Lie algebras $\tminus\iso \mathrm{Prim}\Symtstar\subset \Symtstar$. By the universal property of $U(\tminus)$, it extends uniquely to a homomorphism of graded Hopf algebras 
\begin{equation*}
\varphi: U(\tminus)\to \Symtstar,
\end{equation*}
which is necessarily injective (by \cite{Mont}*{Lem.~5.3.3}, for instance). As the finite-dimensional graded components $U(\tminus)_{-n}$ and $\Symtstar_{-n}=\hSymt{n}{\vphantom{(}}^\ast$ have the same dimension for each $n\in \N$, it follows that $\varphi$ is an isomorphism. 
\end{proof}
\begin{proof}[Proof of \eqref{Pdual:3}] 
If $x\in \tminus$, then $\varphi(\bar\omega (x))$ is the element of $(\hbar\tplus)^\star$ determined by
\begin{equation*}
\varphi(\bar\omega (x))(\hbar y)=\langle\bar\omega(x), y\rangle=\langle x,\bar\omega(y)\rangle=\dot{\omega}^t(\varphi(x))(\hbar y) \quad \forall \; y\in \tplus,
\end{equation*}
where the second equality follows from the fact that the bilinear form $(\,,\,)$ on $\mfg$ is $\omega$-invariant. As $(\hbar\tplus)^\star$ is stable under $\dot\omega^t$ and $U(\tminus)$ is generated by $\tminus$, we may conclude that $\varphi\circ \bar\omega=\dot\omega^t \circ \varphi$.

Similarly, if $x\in\mfg$ and $y\in \tminus$, then $\varphi([x,y])\in (\hbar\tplus)^\star$ is 
determined by 
\begin{equation*}
\varphi([x,y])(\hbar z)=\langle[x,y],z\rangle=\langle y, [z,x] \rangle=\varphi(y)(\hbar[z,x]) \quad \forall\; z\in \tplus.
\end{equation*}
 On the other hand, since $\mathbb{J}^k$ is a $\mfg$-submodule of $\Symt$ for each $k\in \N$, $x \cdot \varphi(y)$ also belongs to $(\hbar\tplus)^\star$. Moreover, we have 
\begin{equation*}
(x \cdot \varphi(y)) (\hbar z)=\varphi(y)\left( S(x) \cdot \hbar z\right)=\varphi(y)(\hbar[z,x]) \quad \forall \; z\in \tplus,
\end{equation*}
where we have used that 
$S(x)\cdot \hbar z-\hbar[z,x]\in \mathbb{J}^2\subset \mathrm{Ker}(\varphi(y))$, which is proven analogously to \eqref{delta-ker}. 
Since $\mfg$ acts on both $U(\tminus)$ and $\Symtstar$ by derivations and $U(\tminus)$ is generated by $\tminus$, the above computation proves that $\varphi$ is a $\mfg$-module homomorphism. \qedhere 
\end{proof}
%
%
%
%

%
\subsection{\texorpdfstring{$\QFhYhg{\vphantom{)}}^\star$}{Yhg*} as a quantization}\label{ssec:Yhg*-quant}
Since the dual Yangian $\Yhgstar$ is a $\Z$-graded topologically free Hopf algebra over $\C[\![\hbar]\!]$ with semiclassical limit that, by Proposition \ref{P:dual}, can be identified with $U(\tminus)$ as a graded Hopf algebra, it is a quantized enveloping algebra which provides a homogeneous quantization of a $\Z$-graded Lie bialgebra structure $(\tminus,\delta)$ on the graded Lie algebra $\tminus=t^{-1}\mfg[t^{-1}]$. 

The following theorem asserts that this Lie bialgebra structure on $\tminus$ is precisely that associated to the Manin triple $(\mft,\tplus,\tminus)$ from Section \ref{ssec:Pr-Manin}.  
\begin{theorem}\label{T:Yhg*-quant}
 $\QFhYhg{\vphantom{)}}^\star$ is a homogeneous quantization of the Lie bialgebra $(\tminus,\delta_{\scriptscriptstyle -})$.
\end{theorem}
\begin{proof}
In light of the above discussion, the definition $\delta_{\scriptscriptstyle -}$ given in Section \ref{ssec:Pr-Manin}, and the identification of Proposition \ref{P:dual},  it is sufficient to prove that the Lie bialgebra structure on $\tminus\cong (\hbar\tplus)^\star$  quantized by $\Yhgstar$ has cobracket $\delta$ given by 
\begin{equation*}
\delta(f)(\hbar x \otimes \hbar y)=f(\hbar[x,y]) \quad \forall\; x,y\in \tplus,\, f\in (\hbar\tplus)^\star.
\end{equation*}
By definition, $\delta$ is given by the formula
\begin{equation*}
\delta(f):=\hbar^{-1}(m^t(\hat f)-m^t_{21}(\hat f)) \mod \hbar \QFhYhg{\vphantom{)}}^\star\otimes \QFhYhg{\vphantom{)}}^\star,
\end{equation*}
where $\hat f\in \Yhgstar$ is any lift of $f\in (\hbar\tplus)^\star$ and $m$ is the product on $\QFhYhg$. For any two elements $x,y\in \tplus$, we have 
\begin{equation}\label{ad-J}
[\mu(x),\hbar\mu(y)]-\hbar\mu[x,y]\in \mathbf{J}^2\subset \hat{f}^{-1}(\hbar\C[\![\hbar]\!])
\end{equation}
which implies the desired result:
\begin{align*}
(\delta(f)-&f\circ \hbar^{-1}[,])(\hbar x\otimes \hbar y)\\
      &=\hat{f}([\mu(x),\hbar\mu(y)]-\hbar\mu[x,y])\mod \hbar\C[\![\hbar]\!]=0. \qedhere
\end{align*}
\end{proof}

%
\subsection{The dual triangular decomposition} \label{ssec:Yhg*-TD}

Our main goal in this subsection is to establish a triangular decomposition for $\Yhgstar$ dual to that for $\QFhYhg$ established in Corollaries \ref{C:QFYhg-chi} and \ref{C:dYhg-TD}; see also Remark \ref{R:QFhYhg-TD}. Along the way, we shall identify a family of generators for $\Yhgstar$; see Lemma \ref{L:Yhg*-gen}.

  For each choice of the symbol $\chi$, consider the restricted dual $\YTDstar{\chi}$ of the $\Z$-graded topological $\C[\![\hbar]\!]$-algebra $\dhYh{\chi}$. By Corollaries \ref{C:Vstar[[h]]} and \ref{C:QFYhg-chi}, this is a $\Z$-graded topologically free $\C[\![\hbar]\!]$-module with semiclassical limit equal to the graded dual 
\begin{equation*}
\msS(\hbar\tplus^\chi){\vphantom{)}}^\star=\bigoplus_{n\in \N} \msS(\hbar\tplus^\chi){\vphantom{)}}^\star_{-n}
\end{equation*}
of the symmetric algebra $\msS(\hbar\tplus^\chi)$. Moreover, by Remark \ref{R:M_N*}, $\YTDstar{\chi}$ coincides with the $\hbar$-adic completion of the graded dual to $\dYh{\chi}$ taken in the category of $\Z$-graded $\C[\hbar]$-modules. Now let $\pi^\chi:\QFhYhg\to \dhYh{\chi}$ denote the $\C[\![\hbar]\!]$-linear projection associated to the identification of $\QFhYhg$ with $\dhYh{+}\otimes \dhYh{0}\otimes \dhYh{-}$ established in Remark \ref{R:QFhYhg-TD}. That is, we have 
\begin{equation*}
\pi^+=\id_{\dhYh{+}}\otimes \veps \otimes \veps,\quad \pi^0= \veps\otimes \id_{\dhYh{0}}\otimes \veps\quad\text{ and }\quad \pi^-=\veps \otimes \veps\otimes \id_{\dhYh{-}}.
\end{equation*}
Taking the transpose of $\pi^\chi$ yields a $\Z$-graded embedding of $\C[\![\hbar]\!]$-modules
\begin{equation*}
(\pi^\chi)^t:\YTDstar{\chi}\into \Yhgstar, \quad f \mapsto f \circ \pi^\chi \quad \forall \quad f\in \YTDstar{\chi},
\end{equation*}
with image consisting of precisely those $g\in \Yhgstar$ for which $g\circ \pi^\chi=g$. Note that if $\hbar^ng$ is contained in this image for some $n\in \N$ then $g$ itself is, and so the subspace topology on $(\pi^\chi)^t(\YTDstar{\chi})$ coincides with its $\hbar$-adic topology. Furthermore, the semiclassical limit of $(\pi^\chi)^t$ is the embedding $\msS(\hbar\tplus^\chi){\vphantom{)}}^\star\into \Symt{\vphantom{)}}^\star$ induced by the projection $\tplus=\tplus^+\oplus \tplus^0\oplus\tplus^-\to \tplus^\chi$. 
 
We shall henceforth adopt the viewpoint that $\YTDstar{\chi}$ is a $\Z$-graded topological submodule of $\Yhgstar$, with the above identification assumed.  We further note that the Chevalley involution $\omega^t$ of $\Yhgstar$ satisfies 
\begin{equation*}
\omega^t(\YTDstar{\pm})=\YTDstar{\mp} \quad \text{ and }\quad \omega^t|_{\YTDstar{0}}=\id_{\YTDstar{0}}. 
\end{equation*}

 Let us now identify a set of elements which generate $\Yhgstar$ as a topological $\C[\![\hbar]\!]$-algebra. These generators are constructed so as to naturally correspond to the coefficients of the $\DYhg$-valued series $\dot{h}_i(u)$ and $\dot{x}_\beta^\pm(u)$ defined in Section \ref{ssec:DYhg-root}, and such an identification will be made precise in the proof of Lemma \ref{L:Yhg*->DYhg}; see \eqref{gen-id}.

Given $\beta\in Q$, let 
$\pi_{\beta}:\dYh{-}\to \dYh{-}_\beta$ denote the natural projection onto the $\beta$-component of $\dYh{-}$.  Since the $\C[\hbar]$-module isomorphism 
\begin{equation*}
\dot\upnu=\upnu|_{\QFYhg}:\QFYhg\iso \msR_\hbar(U(\tplus))\subset U(\tplus)[\hbar]
\end{equation*}
 is $Q$-graded and respects the underlying triangular decompositions, we have 
\begin{equation*}
\dot\upnu(\pi_{-\alpha_i}(\dYh{-}))\subset \bigoplus_{k\in \N}\C[\hbar] \cdot \hbar x_i^- t^k \subset \hbar\tplus [\hbar] \quad \forall \; i\in \mbI.
\end{equation*}
We may thus compose $\dot\upnu\circ \pi_{-\alpha_i}$ with $\mathsf{Res}_{\scriptscriptstyle-}^\hbar(x_i^+t^{-k-1})\in (\hbar\tplus)^\star$, as defined above Proposition \ref{P:dual}, for any fixed $k\in \N$. This outputs a degree $-k-1$ element 
\begin{equation}\label{def:X-dual}
\msX_{i,-k-1}^+:=\mathsf{Res}_{\scriptscriptstyle-}^\hbar(x_i^+t^{-k-1})\circ \dot\upnu\circ \pi_{-\alpha_i}\in \YTDstar{-\!}\subset \Yhgstar,
\end{equation}
where we work through the identification of  $\Hom_{\C[\hbar]}^a (\dYh{-},\C[\hbar])$ with the $a$-th component $\Hom_{\C[\![\hbar]\!]}^a (\dhYh{-\!},\C[\![\hbar]\!])$ of $\YTDstar{-\!}$; see Remark \ref{R:M_N*}.
We now enlarge this family of elements to a generating set for the $\C[\![\hbar]\!]$-algebra $\Yhgstar$ using the coadjoint action of $\mfg$ on  $\Yhgstar$ and the Chevalley involution $\omega^t$ from Section \ref{ssec:Yhg*-Chev}. For each $i\in \mbI$, $\beta\in \Root^+$ and $k\in \N$, we introduce the degree $-k-1$ elements
\begin{gather*}
\mathsf{h}_{i,-k-1}=-x_i^- \cdot \msX_{i,-k-1}^+,\\
\mathsf{X}_{\beta,-k-1}^+:=\mathbf{X}^\beta\cdot \msX_{i(\beta),-k-1}^+\in \Yhgstar_\beta,\quad  \mathsf{X}_{\beta,-k-1}^-:=\omega^t(\mathsf{X}_{\beta,-k-1}^+)\in \Yhgstar_{-\beta},
\end{gather*}
where $i(\beta)\in \mbI$ and $\mathbf{X}^\beta\in U(\mfn_{+})_{\beta-\alpha_{i(\beta)}}$ are as in \eqref{X-beta}.
In the same spirit as in Section \ref{ssec:DYhg-root}, we organize these elements into generating series in $\Yhgstar[\![u]\!]$ by setting
\begin{equation*}
\msh_i(u):=-\sum_{r<0} \msh_{i,r}u^{-r-1}\quad\text{ and }\quad \msX_\beta^\pm(u):=-\sum_{r<0}\msX_{\beta,r}^{\pm} u^{-r-1} \quad 
\end{equation*}
for each $i\in \mbI$ and $\beta\in \Root^+$. We then have the following lemma. 
\begin{lemma}\label{L:Yhg*-gen}
The dual Yangian $\Yhgstar$ is topologically generated as a $\C[\![\hbar]\!]$-algebra by 
the coefficients of $\{\msX_\beta^\pm(u)\}_{\beta\in \Root^+}$ and $\{\msh_i(u)\}_{i\in \mbI}$. Moreover, 
their images under the quotient map $\Yhgstar\onto  \Yhgstar/\hbar\Yhgstar\cong U(\tminus)$ are given by
\begin{equation*}
\mathsf{h}_{i,r}\mapsto h_i t^{r} \quad \text{ and } \quad \msX_{\beta,r}^\pm\mapsto x_\beta^\pm t^{r}
\quad \forall\quad i\in \mbI,\, \beta\in \Root^+\;\text{ and }\; r<0. 
\end{equation*} 
\end{lemma}
\begin{proof}
As $\Yhgstar$ is a flat deformation of the algebra $ U(\tminus)$ over  $\C[\![\hbar]\!]$ and the elements $h_i t^{-k-1}$ and $x_\beta^\pm t^{-k-1}$ generate $\tminus=t^{-1}\mfg[t^{-1}]$ as a Lie algebra, it is sufficient to prove the second assertion of the proposition.

 That the element $\msX_{ir}^+$ coincides with $x_i^+t^{r}$ modulo $\hbar\Yhgstar$ is an immediate consequence of its definition and the identification of $U(\tminus)$ with $\Yhgstar/\hbar\Yhgstar\cong \Symtstar$ provided by Part \eqref{Pdual:2} of Proposition \ref{P:dual}.   The remaining equivalences now follow from the definitions of $h_i$ and $x_\beta^\pm$ (see \eqref{X-beta}) and Part \eqref{Pdual:3} of Proposition \ref{P:dual}, which implies that the quotient map $\Yhgstar \onto U(\tminus)$ is a $\mfg$-module homomorphism intertwining Chevalley involutions. \qedhere
\end{proof}
%
%

Let us now turn towards establishing a triangular decomposition for the dual Yangian $\Yhgstar$.
Following the notation from Section \ref{ssec:QFYhg-TD}, let us introduce the Lie subalgebras $\tminus^\chi$ of $\mft=t^{-1}\mfg[t^{-1}]$ by setting 
\begin{equation*}
\tminus^\pm:=t^{-1}\mfn_{\pm}[t^{-1}]\subset \tminus \quad \text{ and }\quad \tminus^0:=t^{-1}\mfh[t^{-1}]\subset \tminus.
\end{equation*}
The below proposition provides a strengthening of the $\hbar$-adic analogues of Proposition 3.2 and Theorem 4.2 (i) from \cite{KT96}.
\begin{proposition}\label{P:Yhg*-TD}
$\YTDstar{\pm}$ and $\YTDstar{0}$ are $\C[\![\hbar]\!]$-subalgebras of $\Yhgstar$. Moreover: 
\begin{enumerate}[font=\upshape]
\item\label{Yhg*-TD:1} $\YTDstar{\pm}$ is a flat deformation of the $\Z$-graded algebra $U(\tminus^\mp)$ over $\C[\![\hbar]\!]$. In particular, there is an isomorphism of $\Z$-graded topological $\C[\![\hbar]\!]$-modules 
\begin{equation*}
\YTDstar{\pm}\cong U(\tminus^\mp)[\![h]\!].
\end{equation*}
\item \label{Yhg*-TD:2} $\YTDstar{0}$ is commutative and isomorphic to $U(\tminus^0)[\![\hbar]\!]\cong \msS(\tminus^0)[\![\hbar]\!]$ as a $\Z$-graded topological $\C[\![\hbar]\!]$-algebra.

\item\label{Yhg*-TD:3} $\YTDstar{\pm}$ and $\YTDstar{0}$ are topologically generated as $\C[\![\hbar]\!]$-algebras by the coefficients of $\{\msX_\beta^\mp(u)\}_{\beta\in \Root^+}$ and $\{\msh_i(u)\}_{i\in \mbI}$, respectively. 
\end{enumerate}
\end{proposition}
\begin{proof} \let\qed\relax
One can deduce that $\YTDstar{\chi}$ is a subalgebra of $\Yhgstar$ using properties of the Yangian coproduct, as in \cite{KT96}*{Prop.~3.2}. We shall give an alternate simple proof of this fact in Section \ref{ssec:Yhg*-quant-TD} which illustrates that it follows naturally from properties of $\mcR(z)$; see Corollary \ref{C:Yhg*-TD-2} and Remark \ref{R:Yhg*-TD-done}. 

Let us complete the proof of the Proposition assuming Corollary \ref{C:Yhg*-TD-2} which, as explained in Remark \ref{R:Yhg*-TD-done}, also implies that the coefficients of $\{\msX_\beta^\mp(u)\}_{\beta\in \Root^+}$ and $\{\msh_i(u)\}_{i\in \mbI}$ belong to $\YTDstar{\mp}$ and $\YTDstar{0}$, respectively, and that $\YTDstar{0}$ is commutative.

\begin{proof}[Proof of \eqref{Yhg*-TD:1} and \eqref{Yhg*-TD:2}]\let\qed\relax
The graded Lie bialgebra isomorphism $\mathsf{Res}_{\scriptscriptstyle -}^\hbar:\tminus\iso (\hbar\tplus)^\star$ of Proposition \ref{P:dual} restricts to an isomorphism $\tminus^{\shortminus\chi}\iso (\hbar\tplus^\chi)^\star \subset\msS(\hbar\tplus^\chi){\vphantom{)}}^\star$.  Since $\YTDstar{\chi}$ is a subalgebra of $\Yhgstar$, its semiclassical limit $\msS(\hbar\tplus^\chi){\vphantom{)}}^\star$ is a $\Z$-graded $\C[\![\hbar]\!]$-subalgebra of $\Symt{\vphantom{()}}^\star$ and not just a submodule. It follows from these observations that the graded Hopf algebra isomorphism $\varphi$ of Proposition \ref{P:dual} restricts to an injective $\Z$-graded algebra homomorphsim
 \begin{equation}\label{vphi-chi}
 \varphi_\chi:U(\tminus^{\shortminus\chi})\to \msS(\hbar\tplus^\chi){\vphantom{)}}^\star\subset \Symt{\vphantom{)}}^\star
 \end{equation}
 which is surjective by the same argument as used in the proof of Part \eqref{Pdual:2} of Proposition \ref{P:dual}. Since $\YTDstar{\chi}$ is a $\Z$-graded topologically free $\C[\![\hbar]\!]$-algebra with semiclassical limit $\msS(\hbar\tplus^\chi){\vphantom{)}}^\star\cong U(\tminus^{\shortminus\chi})$, taking $\chi=\pm$ recovers  Part \eqref{Yhg*-TD:1} of the Proposition. 
 
 As for Part \eqref{Yhg*-TD:2}, since $\YTDstar{0}$ is a commutative topological  $\C[\![\hbar]\!]$-algebra containing $\{\msh_{ik}\}_{i\in \mbI,k<0}$ (by Corollary \ref{C:Yhg*-TD-2}), there is a $\C[\![\hbar]\!]$-algebra homomorphism 
 \begin{equation*}
 \upeta:\msS(\tminus^0)[\![\hbar]\!]\to \YTDstar{0}
 \end{equation*}
 uniquely determined by $\upeta(h_it^k)=\msh_{ik}$ for all $i\in \mbI$ and $k<0$. By Lemma \ref{L:Yhg*-gen}, the semiclassical limit $\bar\upeta: \msS(\tminus^0)\to U(\tminus^0)$ of $\upeta$ satisfies $\bar\upeta(h_it^k)=h_it^k$ for all $i\in \mbI$ and $k<0$, and thus coincides with the canonical isomorphism  $\msS(\tminus^0)\iso U(\tminus^0)$. Here we have assumed the identification of $\YTDstar{0}/\hbar\YTDstar{0}$ with $U(\tminus^0)$ provided by $\varphi_0$ from \eqref{vphi-chi} above. As $\msS(\tminus^0)$ and $\YTDstar{0}$ are both topologically free, we can conclude from Lemma \ref{L:scl-map} that $\upeta$ is an isomorphism of topological $\C[\![\hbar]\!]$-algebras. 
 \end{proof}
 
\begin{proof}[Proof of \eqref{Yhg*-TD:3}]
For $\YTDstar{0}$, this follows from the definition of the isomorphism $\upeta$ given in the proof of \eqref{Yhg*-TD:2} above. 

Similarly, by Lemma \ref{L:Yhg*-gen} and Corollary \ref{C:Yhg*-TD-2},  $\msX_{\beta,k}^\mp$ belongs to $\YTDstar{\pm}$ and specializes to $x_\beta^\mp t^k$ in 
$\YTDstar{\pm}/\hbar \YTDstar{\pm}\cong U(\tminus^\mp) \subset U(\tminus)$, for each $\beta\in \Root^+$ and $k<0$. Since the elements $x_\beta^\mp t^k$ generate the Lie algebra $\tminus^\mp$ and, by Part \eqref{Yhg*-TD:1}, $\YTDstar{\pm}$ is a flat deformation of the algebra $U(\tminus^\mp)$ over $\C[\![\hbar]\!]$, this completes the proof of the proposition.  \qedhere
\end{proof}
\end{proof}
\begin{remark}
We caution that $\YTDstar{\pm}$ is \textit{not} generated by the elements $\msX_{i,-k-1}^\mp$, for $i\in \mbI$ and $k\in \N$, unless $\mfg=\mfsl_2$, just as $t^{-1}\mfn_{\mp}[t^{-1}]$ is not generated as a Lie algebra by the elements $x_i^\mp t^{-k-1}$ outside of the rank one case. In particular, the statement of Part (i) in \cite{KT96}*{Thm.~4.2}, which is the analogue of Part \eqref{Yhg*-TD:2} above, should be adjusted. 
\end{remark}

As an application of Lemma \ref{L:scl-map}, Proposition \ref{P:Yhg*-TD}, the decomposition $\tminus=\tminus^-\oplus \tminus^0 \oplus \tminus^+$ 
and the Poincar\'{e}--Birkhoff--Witt Theorem for enveloping algebras, we obtain the following variant of Theorem 3.1 (ii) in \cite{KT96}. 
\begin{corollary}\label{C:Yhg*-TD}
The multiplication map 
\begin{equation*}
\msm:\YTDstar{+} \otimes  \YTDstar{0} \otimes \YTDstar{-} \to \Yhgstar
\end{equation*}
is an isomorphism of $\Z$-graded topological $\C[\![\hbar]\!]$-modules. 
\end{corollary}
\begin{remark}
We note that for the statement of \cite{KT96}*{Thm.~3.1} to hold, the tensor product $\otimes$ must be taken to be a completion of the algebraic tensor product $\otimes_\C$ compatible with the underlying $\Z$-filtrations. 
\end{remark}

\section{\texorpdfstring{$\mathrm{D}Y_\hbar\mfg$}{DYhg} as a quantization}\label{sec:DYhg-quant}

In this section, we construct a $\Z$-graded topological Hopf algebra structure on $\DYhg$ which quantizes the graded Lie bialgebra structure on the loop algebra $\mft=\mfg[t^{\pm 1}]$ defined in Section \ref{ssec:Pr-Manin}. This will be achieved in Theorem \ref{T:dual} using Proposition \ref{P:dual-z} and Corollary \ref{C:chk-Phi-Hopf}. As a consequence of these results, we obtain in Section \ref{ssec:Yhg*-quant-TD} a characterization of the restricted  duals $\YTDstar{\pm}$ and $\YTDstar{0}$ in terms of the universal $R$-matrix $\mcR(z)$ which completes the proof of Proposition \ref{P:Yhg*-TD}.

%
\subsection{The morphism \texorpdfstring{$\chk{\Phi_z}$}{}}\label{ssec:chk-Phiz}
Henceforth, the notation $\chk{\hYhg}$ will be used to denote the topological Hopf algebra $(\Yhgstar){\vphantom{)}}^{\scriptscriptstyle\mathrm{cop}}$ over $\C[\![\hbar]\!]$. That is, $\chk{\hYhg}$ coincides with the dual Yangian $\Yhgstar$ as an algebra, and has coproduct $\chk{\Delta}$, counit $\chk{\veps}$ and antipode $\chk{S}$ given by 
\begin{gather*}
\chk{\Delta}:=(1\,2)\circ m^t:\chk{\hYhg}\to \chk{\hYhg}\otimes \chk{\hYhg},\\ \chk{\veps}:=\iota^t: \chk{\hYhg}\to \C[\![\hbar]\!], \quad \chk{S}:=(S^{-1})^t:\chk{\hYhg}\to \chk{\hYhg},
\end{gather*}
where $m$, $\iota$ and $S$ are the product, unit and antipode of the Drinfeld--Gavarini Yangian $\QFhYhg\subset \hYhg$; see Section \ref{ssec:Yhg*-Hopf}. By Theorem \ref{T:Yhg*-quant}, $\chk{\hYhg}$ is a flat deformation of the enveloping algebra $U(\tminus)=U(t^{-1}\mfg[t^{-1}])$ over $\C[\![\hbar]\!]$ which provides a homogeneous quantization of the Lie bialgebra $(\tminus, -\delta_{\scriptscriptstyle -})$. 

Our present goal is to construct a homomorphism of $\C[\![\hbar]\!]$-algebras 
\begin{equation*}
\chk{\Phi}_z:\chk{\hYhg}\to \hYhg[\![z^{-1}]\!] 
\end{equation*}
which is compatible with both the formal shift operator $\Phi_z$ of Theorem \ref{T:Phiz} and the Hopf structures on $\chk{\hYhg}$ and $\hYhg$. This is achieved using the universal $R$-matrix $\mcR(z)$ of the Yangian as follows. By Lemma \ref{L:R-DG}, $\mcR(z)$ is an element of $(\QFYhg\otimes \Yhg)[\![z^{-1}]\!]$, and thus gives rise to a $\C[\![\hbar]\!]$-module homomorphism
\begin{equation*}
\chk{\Phi}_z: \chk{\hYhg}\to \hYhg[\![z^{-1}]\!],\quad  f\mapsto (f\otimes \id) \mcR(-z) \quad \forall\; f\in \chk{\hYhg}.
\end{equation*}
Now recall that $\dot{h}_i(u),\dot{x}_\beta^\pm(u)\in \DYhg[\![u]\!]$ and $\msh_i(u),\msX_\beta^\pm(u)\in \chk{\hYhg}[\![u]\!]$ are the generating series defined in Sections \ref{ssec:DYhg-root} and \ref{ssec:Yhg*-TD}, respectively. In addition, we let $\scriptE_z$ denote the canonical $\C[\![\hbar]\!]$-algebra homomorphism 
\begin{equation*}
\scriptE_{z}: \hYhg[\![w^{-1}]\!]\otimes \hYhg[\![z^{-1}]\!] \to  \hYhg^{\otimes 2}[\![z^{-1}]\!]
\end{equation*}
given by evaluating $w$ to $z$. 
The following proposition asserts that $\chk{\Phi}_z$ indeed has the desired properties. 
\begin{proposition}\label{P:dual-z} $\chk{\Phi}_z$ has the following properties:
\leavevmode
\begin{enumerate}[font=\upshape]
\item\label{dual-z:1} It is a homomorphism of $\C[\![\hbar]\!]$-algebras satisfying
\begin{equation*}
\Delta\circ \chk{\Phi}_z= \scriptE_z\circ (\chk{\Phi}_w \otimes \chk{\Phi}_z)\circ \chk{\Delta},\quad \veps \circ \chk{\Phi}_z=\chk{\veps} \quad \text{ and }\quad S\circ \chk{\Phi}_z=\chk{\Phi}_z\circ \chk{S}. 
\end{equation*}
\item\label{dual-z:2} It is a $U(\mfg)$-module homomorphism compatible with Chevalley involutions:
\begin{equation*}
\chk{\Phi}_z\circ\omega^t=\omega\circ \chk{\Phi}_z. 
\end{equation*}

\item \label{dual-z:3} Its restriction to $\QFYhg{\vphantom{)}}^\star$ is a $\Z$-graded $\C[\hbar]$-algebra homomorphism
\begin{equation*}
\chk{\Phi}_z|_{\QFYhg{\vphantom{)}}^\star}:\QFYhg{\vphantom{)}}^\star \to \LzYhg.
\end{equation*}
\item\label{dual-z:4} For each $\beta\in \Root^+$ and $i\in \mbI$, one has 
\begin{gather*}
\chk{\Phi}_z(\msX_{\beta}^\pm(u))=\sum_{n\in \N}(-1)^{n}u^n \partial_z^{(n)} x_\beta^\pm(-z)=\Phi_z(\dot{x}_\beta^\pm(u)),\\
\chk{\Phi}_z(\msh_i(u))=\sum_{n\in \N}(-1)^{n}u^n \partial_z^{(n)} h_i(-z)=\Phi_z(\dot{h}_i(u)).
\end{gather*}
\end{enumerate}
\end{proposition}

\begin{proof}[Proof of \eqref{dual-z:1}] \let\qed\relax
This is a modification of the standard result that if $H$ is a finite-dimensional quasitriangular Hopf algebra over a field $\mathbb{K}$ with $R$-matrix $R\in H\otimes_{\mathbb{K}}H$, then the map $f\mapsto (f\otimes \id_H)R$ defines a homomorphism of Hopf algebras $(H^\ast){\vphantom{)}}^{\scriptscriptstyle\mathrm{cop}}\to H$; see for instance \cite{CPBook}*{\S4.2.B} or \cite{Radford93}*{\S2}. The only novelty in the present setting is the appearance of the formal parameter $z$. Nonetheless, for the sake of completeness we shall give a full proof.

As the product on $\chk{\hYhg}$ is the transpose of the coproduct $\Delta$ on $\QFhYhg$ and $\Delta\otimes \id (\mcR(z))=\mcR_{13}(z)\mcR_{23}(z)$, we have
\begin{align*}
\chk{\Phi_z}(fg)
=f\otimes g\otimes \id (\mcR_{13}(-z)\mcR_{23}(-z))= \chk{\Phi_z}(f) \chk{\Phi_z}(g) \quad \forall\; f,g\in \chk{\hYhg}.
\end{align*}
Since in addition $(\veps\otimes \id) \mcR(z)=1$, we can conclude that $\chk{\Phi}_z$ is a homomorphism of unital, associative $\C[\![\hbar]\!]$-algebras.

Let us now verify the coproduct identity $\Delta\circ \chk{\Phi}_z= \scriptE_z\circ (\chk{\Phi}_w \otimes \chk{\Phi}_z)\circ \chk{\Delta}$ on $f\in \chk{\hYhg}$. Using the cabling identity $\id\otimes \Delta(\mcR(z))=\mcR_{13}(z)\mcR_{12}(z)$, we obtain 
\begin{align*}
\Delta(\chk{\Phi}_z(f))&=(f\otimes \Delta)(\mcR(-z))\\
&=(f\otimes \id \otimes \id)(\mcR_{13}(-z)\mcR_{12}(-z))\\
&
=(\chk{\Delta}(f)\otimes \id\otimes \id)(\mcR_{13}(-z)\mcR_{24}(-z))\\
&
=(\scriptE_z\circ (\chk{\Delta}(f)\otimes \id\otimes \id) )(\mcR_{13}(-w)\mcR_{24}(-z))\\
&
=(\scriptE_z \circ (\chk{\Phi}_w \otimes \chk{\Phi}_z) \circ \chk{\Delta})(f).
\end{align*}
Finally, the remaining two identities follow from the relations $(\id\otimes \veps)\mcR(z)=1$ and $(S^{-1}\otimes \id)\mcR(z)=(\id \otimes S)\mcR(z)$. Indeed, for each $f\in \chk{\hYhg}$,  we have
\begin{gather*}
\veps \circ \chk{\Phi}_z(f)=(f\otimes \veps)\mcR(-z)=f(1)=\chk{\veps}(f),\\
S\circ \chk{\Phi}_z(f)= (f\otimes S)\mcR(-z)=(\chk{S}(f)\otimes \id) \mcR(-z)=\chk{\Phi}_z\circ \chk{S}(f).
\end{gather*}
\end{proof}

\begin{proof}[Proof of \eqref{dual-z:2}] \let\qed\relax

Since $\tau_z$ restricts to the identity on $U(\mfg)$, the intertwiner equation \eqref{R-inter} implies that $\mcR(z)$ is a $\mfg$-invariant element of $\Yhg^{\otimes 2}[\![z^{-1}]\!]$: 
\begin{equation*}
[x\otimes 1 + 1\otimes x,\mcR(z)]=0 \quad \forall \; x\in \mfg. 
\end{equation*}
It follows readily from this fact, and the definition of the $\mfg$-module structure on $\chk{\hYhg}=\Yhgstar$ introduced in Section \ref{ssec:Yhg*-Chev}, that $\chk{\Phi_z}$ is a $U(\mfg)$-module homomorphism. 

Similarly, by Corollary \ref{C:Yhg-R-Chev} the Chevalley involution $\omega$ satisfies $(\omega\otimes \omega) \mcR(z)=\mcR(z)$, and we thus have 
\begin{align*}
\chk{\Phi_z}(\omega^t(f))=((f\circ\omega) \otimes \omega^2)\mcR(-z)
=(f\otimes \omega)\mcR(-z)=\omega(\chk{\Phi_z}(f)) \quad \forall \; f\in \chk{\hYhg}.
\end{align*}
%
\end{proof}

\begin{proof}[Proof of \eqref{dual-z:3}] \let\qed\relax
Suppose that $f$ is a degree $k$ element of $\QFYhg{\vphantom{)}}^\star$ for some $k\in \Z$. Then $f\otimes \id:\QFYhg\otimes \Yhg \to \Yhg$ is homogeneous of degree $k$, and hence
\begin{equation*}
f\otimes \id (\cYhg{\vphantom{\hYhg}}^{\scriptscriptstyle{(2)}}_z \cap (\QFYhg\otimes \Yhg)[\![z^{-1}]\!]) \subset z^k\prod_{n\in \N} \Yhg_{n+k}z^{-n-k}\subset z^k\Yhgz,
\end{equation*}
where $\cYhg{\vphantom{\hYhg}}^{\scriptscriptstyle{(2)}}_z=\prod_{n\in \N}(\Yhg^{\otimes 2})_n z^{-n}$, as in Section \ref{ssec:Yhg-R}. The assertion now follows from Part \eqref{Yhg-R:3} of Theorem \ref{T:Yhg-R} and Lemma \ref{L:R-DG}, which yield 
\begin{equation*}
\mcR(z)\in \cYhg{\vphantom{\hYhg}}^{\scriptscriptstyle{(2)}}_z \cap (\QFYhg\otimes \Yhg)[\![z^{-1}]\!]. \qedhere
\end{equation*}
\end{proof}

\begin{proof}[Proof of \eqref{dual-z:4}]
Since $\chk{\Phi_z}$ is a $U(\mfg)$-module homomorphism intertwining Chevalley involutions, it is sufficient to establish that 
\begin{equation}\label{chkPhi_X}
\chk{\Phi}_z(\msX_{i,-k-1}^+)=(-1)^{k+1} \partial_z^{(k)} x_i^+(-z)=\Phi_{z}(x_{i,-k-1}^+) \quad \forall\; i \in \mbI \; \text{ and }k\in \N, 
\end{equation}
where we recall that $\msX_{i,-k-1}^+$ was defined explicitly in \eqref{def:X-dual}. 
Since $\veps\otimes \id$ sends both $\mcR^0(z)$ and $\mcR^+(z)$ to $1$ and $\msX_{i,-k-1}^+$ vanishes on $\dot{\msY}_\hbar^-(\mfg)_\beta$ for $\beta\neq -\alpha_i$, we have 
\begin{align*}
\chk{\Phi}_z(\msX_{i,-k-1}^+)=(\msX_{i,-k-1}^+\otimes \id)\mcR^-_{\alpha_i}(z).
\end{align*}
Using \eqref{R-recur}, we deduce that the element $\mcR^-_{\alpha_i}(z)$ is given by 
\begin{equation*}
\mcR^-_{\alpha_i}(z)=\sum_{p\geq 0}\frac{\alpha_i(\zeta)\mbT(\zeta)^p}{(z\alpha_i(\zeta))^{p+1}}
(\hbar x_{i,0}^-
\otimes x_{i,0}^+)=\sum_{p\geq 0}\sum_{n=0}^p \binom{p}{n}(-1)^n \hbar x_{i,n}^-\otimes x_{i,p-n}^+z^{-p-1},
\end{equation*}
where we have used $\mathrm{ad}(\mathrm{T}(\zeta))^p(x_{i,0}^\pm)=(\pm 1)^p \alpha_i(\zeta)^p x_{i,p}^\pm$. 
This can be rewritten as 
\begin{equation*}
\mcR^-_{\alpha_i}(z)=\sum_{n\in \N} \hbar x_{i,n}^- \otimes \partial_z^{(n)}x_i^+(z)
=-\sum_{n\in \N}\hbar x_{i,n}^- \otimes \Phi_{-z}(x_{i,-n-1}^+),
\end{equation*}
where the second equality is due to Part \eqref{Phiz:1} of Theorem \ref{T:Phiz}. 
As $\msX_{i,-k-1}^+(\hbar x_{i,n}^-)=\mathsf{Res}_{\scriptscriptstyle-}^\hbar(x_i^+t^{-k-1})(\hbar x_{i}^-t^n)=-\delta_{kn}$ for all $k,n\in \N$, this implies the identity \eqref{chkPhi_X}.
\end{proof}
\begin{remark}\label{R:dual-z-graded}
As $\chk{\hYhg}$ and $\LzhYhg$ are $\Z$-graded topological $\C[\![\hbar]\!]$-algebras with
\begin{equation*}
(\chk{\hYhg})_\Z=(\Yhgstar)_\Z\cong\QFYhg{\vphantom{)}}^\star \quad \text{ and }\quad (\LzhYhg)_\Z=\LzYhg.
\end{equation*}
Part \eqref{dual-z:3} of the Proposition is equivalent to the assertion that $\chk{\Phi_z}$ is a $\Z$-graded $\C[\![\hbar]\!]$-algebra homomorphism $\chk{\Phi_z}:\chk{\hYhg}\to \LzhYhg$.
\end{remark}
Since $\chk{\Phi_z}$ has image in $\LzhYhg$, we may compose it with $\mathscr{Ev}$ from \eqref{DYhg-Ev} to obtain a  $\C[\![\hbar]\!]$-algebra homomorphism 
\begin{equation*}\label{chk-Phi}
\chk{\Phi}:=\mathscr{Ev}\circ \chk{\Phi_z}:\chk{\hYhg}\to \cYhg.
\end{equation*}
Our present goal is to apply Part \eqref{dual-z:1} of Proposition \ref{P:dual-z} to interpret $\chk{\Phi}$ as a homomorphism of topological Hopf algebras, where the topological structure on the completed Yangian is that induced by the gradation topology on $\hYhg$. To make this precise, let us define $\cYhg{\vphantom{)}}^{\scriptscriptstyle(n)}$, for any $n\in \N$, to be the formal completion of $\Yhg^{\otimes n}$ with respect to its $\N$-grading: 
\begin{equation*}
\cYhg{\vphantom{)}}^{\scriptscriptstyle(n)}:=\prod_{k\in \N}(\Yhg^{\otimes n})_k.
\end{equation*}
Equivalently, it is the completion of the $\N$-graded topological $\C[\![\hbar]\!]$-algebra $\hYhg^{\otimes n}$ with respect to the filtration defining its gradation topology, as defined in Section \ref{ssec:Pr-gr*}. By \cite{WDYhg}*{Prop.~A.1}, this is a topologically free $\C[\![\hbar]\!]$-algebra containing $\hYhg^{\otimes n}$ as a subalgebra; see also  \cite{WDYhg}*{Lem.~4.1} and Section \ref{ssec:DYhg-shift} above. Furthermore, we can (and shall) view 
\begin{equation*}
(\cYhg)^{\otimes n}\subset \cYhg{\vphantom{)}}^{\scriptscriptstyle(n)} \quad \forall \; n\in \N,
\end{equation*}
where $\otimes$ denotes the topological tensor product over $\C[\![\hbar]\!]$. 
%

Next, observe that the $\N$-graded Hopf algebra structure on $\Yhg$ induces a topological Hopf structure on $\cYhg$, equipped with the grading-completed tensor product. More precisely, as the multiplication $m$, coproduct $\Delta$, counit $\veps$, and antipode $S$ on $\Yhg$ are $\N$-graded, they uniquely extend to $\C[\![\hbar]\!]$-module homomorphisms 
\begin{gather*}
m: \cYhg{\vphantom{)}}^{\scriptscriptstyle(2)}\to \cYhg, \quad \Delta: \cYhg\to \cYhg{\vphantom{)}}^{\scriptscriptstyle(2)}, \quad \veps: \cYhg\to \C[\![\hbar]\!], \quad S: \cYhg\to \cYhg,
\end{gather*}
which collectively satisfy the axioms of a Hopf algebra. Proposition \ref{P:dual-z} then admits the following corollary. 
\begin{corollary}\label{C:chk-Phi-Hopf}
The $\C[\![\hbar]\!]$-algebra homomorphism $\chk{\Phi}$ is a morphism of topological Hopf algebras. That is, it satisfies
\begin{equation*}
\Delta\circ \chk{\Phi}=(\chk{\Phi}\otimes \chk{\Phi})\circ \chk{\Delta}, \quad \veps \circ \chk{\Phi}=\chk{\veps} \quad \text{ and }\quad S\circ \chk{\Phi}=\chk{\Phi}\circ \chk{S}.
\end{equation*}
In particular, one has $\mathrm{Im}(\Delta\circ \chk{\Phi})\subset \cYhg\otimes \cYhg$. 
\end{corollary}
\begin{proof}
The counit and antipode relations are obtained by applying $\mathscr{Ev}$ to the corresponding relations of Part \eqref{dual-z:1} of Proposition \ref{P:dual-z} and appealing to the identity $\chk{\Phi}=\mathscr{Ev}\circ \chk{\Phi_z}$. The idea now is that the relation $\Delta\circ \chk{\Phi}=(\chk{\Phi}\otimes \chk{\Phi})\circ \chk{\Delta}$ should follow by applying $\mathscr{Ev}\otimes \mathscr{Ev}$ to both sides of the identity $\Delta\circ \chk{\Phi}_z= \scriptE_z\circ (\chk{\Phi}_w \otimes \chk{\Phi}_z)\circ \chk{\Delta}$.

However, to make this precise we must first make a few technical observations. Recall from \eqref{LzYhg-n} that $\mathds{L}(\cYhg{\vphantom{\hYhg}}^{\scriptscriptstyle{(2)}}_z)$ is the $\Z$-graded subalgebra of $\Yhg^{\otimes 2}[z;z^{-1}]\!]$ with $k$-th homogeneous component $ z^k \cYhg{\vphantom{\hYhg}}^{\scriptscriptstyle{(2)}}_z$, where 
\begin{equation*}
\cYhg{\vphantom{\hYhg}}^{\scriptscriptstyle{(2)}}_z=\prod_{n\in \N} (\Yhg^{\otimes 2})_n z^{-n} \subset \Yhg^{\otimes 2}[\![z^{-1}]\!]. 
\end{equation*}
Following Section \ref{ssec:DYhg-shift}, we shall write $\mathrm{L}(\cYhg{\vphantom{\hYhg}}^{\scriptscriptstyle{(2)}}_z)$ for the $\hbar$-adic completion of $\mathds{L}(\cYhg{\vphantom{\hYhg}}^{\scriptscriptstyle{(2)}}_z)$. This is a $\Z$-graded topological $\C[\![\hbar]\!]$-algebra contained in the formal series space $(\hYhg\otimes \hYhg)[\![z^{\pm 1}]\!]$; see Lemma \ref{L:DYhg-LzYhg}. As in  \eqref{DYhg-Ev}, evaluation at $z=1$ yields a $\C[\![\hbar]\!]$-algebra epimorphism
\begin{equation*}
\mathscr{Ev}^{\scriptscriptstyle{(2)}}: \mathrm{L}(\cYhg{\vphantom{\hYhg}}^{\scriptscriptstyle{(2)}}_z)
\onto \cYhg{\vphantom{)}}^{\scriptscriptstyle(2)}, \quad f(z)\mapsto f(1).
\end{equation*}
Next, let $\scriptE_z^{\mathrm{L}}$ denote the natural $\C[\![\hbar]\!]$-algebra homomorphism 
\begin{equation*}
\scriptE_z^{\mathrm{L}}: \mathrm{L}\cYhg{\vphantom{\hYhg}}_w\otimes \LzhYhg\to \mathrm{L}(\cYhg{\vphantom{\hYhg}}^{\scriptscriptstyle{(2)}}_z)
\end{equation*}
given by evaluating $w\mapsto z$. Then, since $\chk{\Phi_z}$ is a $\Z$-graded $\C[\![\hbar]\!]$-algebra homomorphism $\chk{\hYhg}\to \LzhYhg$ and $\Delta$ is homogeneous of degree zero, the first relation of Part \eqref{dual-z:1} in Proposition \ref{P:dual-z} is equivalent to the identity
\begin{equation*}
\Delta\circ \chk{\Phi}_z= \scriptE_z^{\mathrm{L}}\circ (\chk{\Phi}_w \otimes \chk{\Phi}_z)\circ \chk{\Delta}
\end{equation*}
in $\Hom_{\C[\![\hbar]\!]}(\chk{\hYhg},\mathrm{L}(\cYhg{\vphantom{\hYhg}}^{\scriptscriptstyle{(2)}}_z))$. As $\scriptE_z^\mathrm{L}$ satisfies the relation $\mathscr{Ev}^{\scriptscriptstyle{(2)}}\circ \scriptE_z^\mathrm{L}= \mathscr{Ev}\otimes \mathscr{Ev}$, applying $\mathscr{Ev}^{\scriptscriptstyle{(2)}}$ to the above identity recovers $\Delta\circ \chk{\Phi}=(\chk{\Phi}\otimes \chk{\Phi})\circ \chk{\Delta}$. \qedhere
\end{proof}

%
\subsection{\texorpdfstring{$\mathrm{D}Y_\hbar\mfg$}{DYhg} as a quantization}\label{ssec:DYhg-quant}
With Proposition \ref{P:dual-z} and Corollary \ref{C:chk-Phi-Hopf} in hand, we now turn to proving that $\DYhg$ provides a homogeneous quantization of the $\Z$-graded
Lie bialgebra $\mft=\mfg[t^{\pm 1}]$, equipped with the Lie cobracket $\delta$ defined in Section \ref{ssec:Pr-Manin}. 

To begin, observe that the homomorphisms  $\chk{\Phi_z}$ and $\chk{\Phi}$ satisfy the relation
\begin{equation*}
\Gamma\circ \chk{\Phi}=\Gamma_z\circ \chk{\Phi_z},
\end{equation*}
where $\Gamma$ and $\Gamma_z$ are as in Theorem \ref{T:Phi} and \eqref{Gamma_z}, respectively. The $\C[\![\hbar]\!]$-algebra homomorphism defined by either side of this relation shall be denoted $\check{\imath}$:
\begin{equation*}
\check{\imath}:=\Gamma_z\circ \chk{\Phi_z}:\chk{\hYhg}\to \cDYhg.
\end{equation*}
The following lemma shows that $\check{\imath}$ is injective, $\Z$-graded, and has image contained in the Yangian double $\DYhg$.
\begin{lemma}\label{L:Yhg*->DYhg} 
The morphism $\check{\imath}$ is an embedding of $\Z$-graded topological $\C[\![\hbar]\!]$-algebras
\begin{equation*}
\check{\imath}:\chk{\hYhg}\into \DYhg
\end{equation*}
satisfying $\Phi\circ \check{\imath}=\chk{\Phi}$ and $\Phi_z \circ \check{\imath} =\chk{\Phi}_z$. In particular, $\chk{\Phi}$ and $\chk{\Phi_z}$ are both injective.
\end{lemma}
\begin{proof}

Since $\Gamma_z\circ \Phi_z=\id_{\DYhg}$, Part \eqref{dual-z:4} of Proposition \ref{P:dual-z} yields
\begin{equation}\label{gen-id}
\begin{gathered}
\check{\imath}\,(\msX_\beta^\pm(u))=\Gamma_z(\Phi_z(\dot{x}_\beta^\pm(u)))=\dot{x}_\beta^\pm(u)\quad \forall\; \beta\in \Root^+,\\
\check{\imath}\,(\msh_i(u))=\Gamma_z(\Phi_z(\dot{h}_i(u)))=\dot{h}_i(u)\quad \forall\; i \in \mbI. 
\end{gathered}
\end{equation}
By Remark \ref{R:DYhg->cDYhg}, $\DYhg$ is a closed subspace of its $\mcJ$-adic completion, viewed as a topological $\C[\![\hbar]\!]$-module. As $\chk{\hYhg}$ is topologically generated as a $\C[\![\hbar]\!]$-algebra by the coefficients of all the series $\msX_\beta^\pm(u)$ and $\msh_i(u)$ (by Lemma \ref{L:Yhg*-gen}), the equalities \eqref{gen-id} imply that $\check{\imath}$ has image in $\DYhg$ and thus can be viewed as a $\C[\![\hbar]\!]$-algebra homomorphism 
\begin{equation*}
\check{\imath}:\chk{\hYhg}\to \DYhg
\end{equation*}
which necessarily satisfies $\Phi\circ \check{\imath}=\chk{\Phi}$. Similarly, since $\Phi_z(\DYhg)$ is closed in $\LzhYhg$ (see Remark \ref{R:DYhg->cDYhg}), Part \eqref{dual-z:4} of Proposition \ref{P:dual-z} and \eqref{gen-id} give $\Phi_z \circ \check{\imath} =\chk{\Phi}_z$. 

As $\chk{\hYhg}$ and $\DYhg$ are both topologically free, it follows from Lemma \ref{L:scl-map} That $\check{\imath}$ will be injective provided its semiclassical limit $\check{\underline{\imath}}: U(\tminus)\to U(\mft)$ is. That this is indeed the case is a consequence of  Lemma \ref{L:Yhg*-gen}, the relations \eqref{gen-id}, and the definitions of $\dot{x}_\beta^\pm(u)$ and $\dot{h}_i(u)$ given in Section \ref{ssec:DYhg-root}, which imply that $\check{\underline{\imath}}$ coincides with the natural inclusion of $U(\tminus)$ into $U(\mft)$ induced by the polarization $\mft=\tplus\oplus \tminus=\mfg[t]\oplus t^{-1}\mfg[t^{-1}]$.

Finally, by Part \eqref{Phiz:2} of Theorem \ref{T:Phiz} and Part \eqref{dual-z:3} of Proposition \ref{P:dual-z}, $\Phi_z$ and $\chk{\Phi_z}$ are both $\Z$-graded; see Remarks \ref{R:Phiz-graded} and \ref{R:dual-z-graded}. As $\Phi_z$ is injective, it follows automatically that $\check{\imath}$ is $\Z$-graded.\qedhere
\end{proof}

The embedding $\check{\imath}$ is at the heart of the following theorem, which provides the first main result of this article.
\begin{theorem}\label{T:dual} There is a unique $\Z$-graded topological Hopf algebra structure on $\DYhg$ such that the inclusions 
\begin{equation*}
\hYhg \xrightarrow{\;\imath^{\,}\;} \DYhg \xleftarrow{\;\check{\imath}^{\,}\;} \chk{\hYhg}
\end{equation*}
are morphisms of $\Z$-graded topological Hopf algebras over $\C[\![\hbar]\!]$. Equipped with this Hopf structure, $\DYhg$ is a homogeneous quantization of the Lie bialgebra $(\mft ,\delta)$.
\end{theorem}
\begin{proof}

If $\DYhg$ can be equipped with a coproduct $\dot\Delta$, counit $\dot\veps$ and antipode $\dot S$ 
which give it the structure of a topological Hopf algebra over $\C[\![\hbar]\!]$ and, in addition, make $\imath$ and $\check{\imath}$ morphisms of $\Z$-graded topological Hopf algebras, then this structure is necessarily unique and $\Z$-graded. Indeed,  $\dot \Delta$, $\dot\veps$ and $\dot S$ are determined by their values on any set of generators of the topological algebra $\DYhg$, and thus by their values on $\imath(\hYhg)\cup \check{\imath}(\chk{\hYhg})$.
 
Let us now establish the existence of a topological Hopf structure on $\DYhg$ over $\C[\![\hbar]\!]$ with the claimed properties. 
We begin by observing that $\imath$ and $\check{\imath}$ satisfy the relations
\begin{equation}\label{DeltaPhi-z}
\begin{aligned}
\Delta\circ \Phi \circ \imath&= (\Phi\otimes \Phi) \circ (\imath\otimes \imath)\circ \Delta,\\
\Delta\circ \Phi \circ \check{\imath}&= (\Phi\otimes \Phi) \circ (\check{\imath}\otimes \check{\imath})\circ \chk{\Delta}.
\end{aligned}
\end{equation}
The first relation follows from the identity $\Phi \circ \imath=\tau_1$ and that, for each $c\in \C$, $\tau_c$ is a Hopf algebra automorphism; see Section \ref{ssec:Yhg-aut}. As for the second relation, since $\Phi\circ \check{\imath}=\chk{\Phi}$, Corollary \ref{C:chk-Phi-Hopf} yields 
\begin{align*}
\Delta\circ \Phi \circ \check{\imath}&=\Delta\circ \chk{\Phi}= (\chk{\Phi} \otimes \chk{\Phi})\circ \chk{\Delta}=(\Phi\otimes \Phi) \circ (\check{\imath}\otimes \check{\imath})\circ \chk{\Delta}.
\end{align*}
Since $\DYhg$ is a generated as a topological $\C[\![\hbar]\!]$-algebra by $\imath(\hYhg)\cup \check{\imath}(\chk{\hYhg})$, these relations imply that $\Delta\circ \Phi$ has image satisfying
\begin{equation*}
\mathrm{Im}(\Delta\circ \Phi)\subset \mathrm{Im}(\Phi\otimes \Phi)\subset \cYhg\otimes \cYhg,
\end{equation*}
where we view the right-hand side as subspace of $\cYhg{\vphantom{)}}^{\scriptscriptstyle(2)}$, as in Corollary \ref{C:chk-Phi-Hopf}. 
We may therefore introduce a $\C[\![\hbar]\!]$-algebra homomorphism $\dot\Delta$ by 
\begin{equation*}
\dot\Delta:=(\Gamma\otimes \Gamma) \circ \Delta \circ \Phi:\DYhg\to \DYhg\otimes \DYhg,
\end{equation*}
where $\Gamma$ is the inverse of $\wh{\Phi}$, as in Theorem \ref{T:Phi}. Since $\Gamma\circ \Phi=\id$, it follows from this definition and the relations  \eqref{DeltaPhi-z} that
\begin{equation}\label{i-delta}
\dot\Delta\circ \imath=(\imath\otimes \imath) \circ \Delta\quad \text{ and }\quad \dot\Delta\circ \check{\imath}=(\check{\imath}\otimes \check{\imath}) \circ \chk{\Delta}.
\end{equation}
Similarly, from Corollary \ref{C:chk-Phi-Hopf} and the relations 
$\tau_1\circ S=S\circ \tau_1$ and $\veps \circ \tau_1=\veps$, we find that $\imath$ and $\check{\imath}$ satisfy
\begin{gather*}
S\circ \Phi \circ \imath=\Phi \circ \imath\circ S, \quad S\circ \Phi \circ \check{\imath}=\Phi \circ \check{\imath}\circ \chk{S},\\
 \veps \circ \Phi\circ \imath=\veps, \quad \veps \circ \Phi\circ \check{\imath}=\chk{\veps}.
\end{gather*}
In particular, these relations imply that $\mathrm{Im}(S\circ \Phi)\subset \mathrm{Im}(\Phi)$. We may therefore define morphisms $\dot{S}$ and $\dot{\veps}$ by 
\begin{equation*}
\dot{S}:=\Gamma\circ S\circ \Phi: \DYhg\to\DYhg \quad \text{ and }\quad \dot{\veps}:=\veps\circ \Phi:\DYhg\to \C[\![\hbar]\!]
\end{equation*}
which by construction satisfy the compatibility relations 
\begin{equation}\label{i-S-veps}
\dot{S}\circ \imath=\imath \circ S, \quad \dot{S}\circ \check{\imath}=\check{\imath}\circ \chk{S},\quad \dot{\veps}\circ \imath=\veps \quad \text{ and }\quad \dot{\veps}\circ \check{\imath}=\chk{\veps}.
\end{equation}

Since $\Gamma|_{\mathrm{Im}(\Phi)}=\Phi^{-1}$ and $\hYhg$ is a topological Hopf algebra with coproduct $\Delta$, antipode $S$ and counit $\veps$, the above definitions imply that $\DYhg$ is a topological Hopf algebra over $\C[\![\hbar]\!]$ with coproduct $\dot\Delta$, antipode $\dot S$, and counit $\dot\veps$. Moreover, 
the relations \eqref{i-delta} and \eqref{i-S-veps} prove that, when $\DYhg$ is given this Hopf structure, 
the $\Z$-graded embeddings $\imath$ and $\check{\imath}$ are homomorphisms of topological Hopf algebras.

We are left to establish the second assertion of the Theorem. 
By what we have shown so far and Theorem \ref{T:DYhg-PBW}, $\DYhg$ is a flat, $\Z$-graded Hopf algebra deformation of the universal enveloping $U(\mft)$ of $\mft=\mfg[t^{\pm 1}]$ over $\C[\![\hbar]\!]$. It thus provides a homogeneous quantization of a $\Z$-graded Lie bialgebra structure on the Lie algebra $\mft$ with cobracket $\delta_\mft$ given by the formula \eqref{quant-comm}. From Theorem \ref{T:Yhg-quant}, Theorem \ref{T:Yhg*-quant} and the first part of the theorem, $\delta_\mft$ satisfies  $\delta_\mft|_{\tplus}=\delta_{\scriptscriptstyle +}$ and $\delta_\mft|_{\tminus}=-\delta_{\scriptscriptstyle -}$. As $\mft=\tplus\oplus \tminus$, we can conclude that $\delta_\mft$ coincides with the Lie cobracket $\delta$ defined in Section \ref{ssec:Pr-Manin}.  
\end{proof}
\begin{remark}
The proof of Theorem \ref{T:dual} shows that the coproduct $\dot\Delta$, counit $\dot\veps$ and antipode $\dot S$ on $\DYhg$ are uniquely determined by the requirement that $\Phi$ is a homomorphism of topological Hopf algebras 
\begin{equation*}
\Phi:\DYhg\to \cYhg,
\end{equation*}
where $\cYhg$ is given the topological Hopf structure defined above Corollary \ref{C:chk-Phi-Hopf}. Here we emphasize that although $\DYhg$ is a genuine topological Hopf algebra over $\C[\![\hbar]\!]$ in the sense of Section \ref{ssec:Pr-Top}, the completed Yangian $\cYhg$ is not. In the same breath, the Hopf algebra structure on $\DYhg$ is uniquely characterized by the requirement that $\Phi_z$ satisfies the relations
\begin{equation*}
\Delta\circ \Phi_z= \scriptE_z^{\mathrm{L}}\circ (\Phi_w \otimes \Phi_z)\circ \dot{\Delta},\quad \veps \circ \Phi_z=\dot{\veps} \quad \text{ and }\quad S\circ \Phi_z=\Phi_z\circ \dot{S},
\end{equation*}
where $\scriptE_z^{\mathrm{L}}$ is as in the proof of Corollary \ref{C:chk-Phi-Hopf}. In particular, this makes precise the uniqueness statement in the first assertion of Theorem \ref{T:Intro}. 
\end{remark}
%
%
%

%
\subsection{The dual triangular decomposition revisited}\label{ssec:Yhg*-quant-TD}

We conclude this section by giving an equivalent characterization of the restricted dual $\YTDstar{\chi}$ to $\dhYh{\chi}$ considered in Section \ref{ssec:Yhg*-TD}, where we recall that $\chi$ takes value $+$, $-$ or $0$. 
\begin{corollary}\label{C:Yhg*-TD-2}
 For each each choice of $\chi$, $\YTDstar{\chi}$ satisfies
\begin{equation*}
\YTDstar{\chi}=\{f\in \Yhgstar: \chk{\Phi_z}(f)=(f\otimes \id)\mcR^{\chi}(-z)\}=(\chk{\Phi_z})^{-1}(Y_\hbar^{\shortminus \chi}\mfg [\![z^{-1}]\!]).
\end{equation*}
\end{corollary}
\begin{proof}
From the definitions of $\dot{\msY}_\hbar^\chi\mfg$ and $\chk{\Phi_z}$, and the Gauss decomposition of the universal $R$-matrix, we obtain the sequence of inclusions 
\begin{equation*}
\YTDstar{\chi}\subset\{f\in \Yhgstar: \chk{\Phi_z}(f)=(f\otimes \id)\mcR^{\chi}(-z)\}\subset(\chk{\Phi_z})^{-1}(Y_\hbar^{\shortminus \chi}\mfg [\![z^{-1}]\!]).
\end{equation*}
It therefore suffices to show that  $(\chk{\Phi_z})^{-1}(Y_\hbar^{\shortminus \chi}\mfg [\![z^{-1}]\!])\subset \YTDstar{\chi}$. We shall establish this for $\chi=-$. The proof in the remaining cases is identical, and hence omitted.

Let $\pi_+:\hYhg\to Y_\hbar^{+\!}\mfg$ be the $\C[\![\hbar]\!]$-linear projection associated to the opposite triangular decomposition $\hYhg\cong Y_\hbar^{-\!}\mfg\otimes Y_\hbar^0\mfg\otimes Y_\hbar^{+\!}\mfg$ from  Remark \ref{R:hYhg-TD}, and suppose $f\in \Yhgstar$ satisfies $\Phi_z(f)\in Y_\hbar^{+}\!\mfg [\![z^{-1}]\!]$. We wish to show that $f\in \YTDstar{-\!}$. 

To begin, note that since $(\id\otimes \veps)\mcR^\chi(z)=1$, we have 
\begin{equation*}
\chk{\Phi_z}(f)=\pi_+(\chk{\Phi_z}(f))=(f\otimes 1)\mcR^-(z).
\end{equation*}
Consider now the element $f\circ \pi^-\in \YTDstar{-} \subset \Yhgstar$, where $\pi^-:\QFhYhg\to \dhYh{-}$ is as in Section \ref{ssec:Yhg*-TD}. Since $(\pi^-\otimes \id)\mcR(z)=\mcR^-(z)$, it satisfies
\begin{equation*}
\chk{\Phi_z}(f\circ \pi^-)=(f\otimes 1)\mcR^-(z)=\chk{\Phi_z}(f).
\end{equation*}
As  $\chk{\Phi_z}$ is injective (by Lemma \ref{L:Yhg*->DYhg}) we can conclude that $f=f\circ \pi^-\in \YTDstar{-}$.  \qedhere
\end{proof}
\begin{remark}\label{R:Yhg*-TD-done}
Since $Y_\hbar^{\pm\!}\mfg$ and $Y_\hbar^0\mfg$ are $\C[\![\hbar]\!]$-subalgebras of $\hYhg$, it follows from this corollary that $\YTDstar{\chi}$ is a $\C[\![\hbar]\!]$-subalgebra of $\Yhgstar$ for each choice of $\chi$. In addition, it is immediate from this characterization and Part \eqref{dual-z:4} of Proposition \ref{P:dual-z} that the coefficients of the series $\msX_\beta^\pm(u)$ and $\msh_i(u)$ belong to $\YTDstar{\mp}$ and $\YTDstar{0}$, respectively, for each $\beta\in\Root^+$ and $i\in \mbI$. Similarly, since $\chk{\Phi_z}$ is injective and $Y_\hbar^0\mfg$ is commutative, we deduce that $\YTDstar{0}$ is commutative.

In particular, these observations complete the proof of Proposition \ref{P:Yhg*-TD}. We emphasize that we have not applied that Proposition in establishing any of the results in Section \ref{sec:DYhg-quant}. 
\end{remark}
%
%
\section{\texorpdfstring{$\mathrm{D}Y_\hbar\mfg$}{DYhg} as a quantum double}\label{sec:DYhg-double}

We now turn to reframing Theorem \ref{T:dual} in the context of the quantum double. Our central objective is to prove the second main result of this article, Theorem \ref{T:DYhg-DD}, which establishes that the Yangian double $\DYhg$ is isomorphic, as a $\Z$-graded topological Hopf algebra over $\C[\![\hbar]\!]$, to the restricted quantum double of the Yangian, which is defined explicitly in Section \ref{ssec:D(Yhg)}.
%
\subsection{The quantum double of \texorpdfstring{$U_\hbar\mfb$}{U\_hb}}\label{ssec:D(H)} Let us begin by recalling, in broad strokes, the general construction of the quantum double of a quantized enveloping algebra, as was first outlined by Drinfeld in \cite{DrQG}*{\S13}.

Suppose that $U_\hbar\mfb$ is a quantization of a finite-dimensional Lie bialgebra $\mfb$, and let $\chk{U_\hbar\mfb}$ denote the quantized enveloping algebra $(U_\hbar\mfb^{\circ}){\vphantom{)}}^{\scriptscriptstyle\mathrm{cop}}$, where $U_\hbar\mfb^{\circ}$ is the topological dual to $U_\hbar\mfb^\prime$ introduced in Section \ref{ssec:QUE-dual}.  Then there exists a unique topological Hopf algebra $D(U_\hbar\mfb)$ over $\C[\![\hbar]\!]$, the \textit{quantum double} of $U_\hbar\mfb$, satisfying the following three properties: 
\begin{enumerate}\setlength{\itemsep}{3pt}
\item\label{DH:1} There are embeddings of topological Hopf algebras  
\begin{equation*}
U_\hbar\mfb \xrightarrow{\;\imath\;} D(U_\hbar\mfb) \xleftarrow{\;\check{\imath}\;} \chk{U_\hbar\mfb}.
\end{equation*}
\item\label{DH:2} The composite $m\circ (\check{\imath}\otimes \imath):  \chk{U_\hbar\mfb}\otimes U_\hbar\mfb\to D(U_\hbar\mfb)$ is an isomorphism of $\C[\![\hbar]\!]$-modules.
\item\label{DH:3} The canonical element $R\in U_\hbar\mfb^\prime\otimes \chk{U_\hbar\mfb}\subset D(U_\hbar\mfb)\otimes D(U_\hbar\mfb)$ associated to the pairing between $U_\hbar\mfb^\prime$ and $U_\hbar\mfb^\circ$, which coincides with the canonical tensor in $U_\hbar\mfb\otimes U_\hbar\mfb^\ast$, defines a quasitriangular structure on $D(U_\hbar\mfb)$. That is, one has: 
\begin{gather*}
\op{\Delta}(x)=R \Delta(x)   R^{-1} \quad \forall\; x\in D(U_\hbar\mfb),\\
\Delta\otimes \id ( R ) =R_{13}R_{23}, \quad \id\otimes \Delta (R)=R_{13}R_{12}.
\end{gather*}
\end{enumerate}
In addition, $D(U_\hbar\mfb)$ provides a quantization of the Drinfeld double of the finite-dimensional Lie bialgebra $\mfb$. 

The quantum double $D(U_\hbar\mfb)$ can be realized explicitly as the tensor product of topological coalgebras $\chk{U_\hbar\mfb}\otimes U_\hbar\mfb$, with multiplication determined from the cross relations 
\begin{equation}\label{fd-double}
\imath(x)\check{\imath}(f)=f_3(x_1)f_1\!\left(S^{-1}(x_3)\right)f_2\otimes x_2 \quad \forall\; x\in U_\hbar\mfb^\prime, \, f\in \chk{U_\hbar\mfb},
\end{equation}
where we have used the sumless Sweedler notation for iterated coproducts on $U_\hbar\mfb$ and $\chk{U_\hbar\mfb}$, and $\imath$ and $\check{\imath}$ are now given by $\imath(x)=1\otimes x$ and $\check{\imath}(f)=f\otimes 1$ for all $x\in U_\hbar\mfb$ and $f\in \chk{U_\hbar\mfb}$. As spelled out in detail in \cite{Andrea-Valerio-18}*{\S A.5}, this can be realized as a special instance of the \textit{double cross product} construction. Namely, one has 
\begin{equation*}
D(U_\hbar\mfb)=\chk{U_\hbar\mfb} \bowtie U_\hbar\mfb,
\end{equation*}
with respect to the left coadjoint action $\rhd$ of $U_\hbar\mfb$ on $\chk{U_\hbar\mfb}$ and the right coadjoint action $\lhd$ on $\chk{U_\hbar\mfb}$ on $U_\hbar\mfb$. Given the $U_\hbar\mfb$ analogue of Lemma \ref{L:adjoint}, established in \cite{Andrea-Valerio-18}*{Prop.~A.5}, this construction of $D(U_\hbar\mfb)$ proceeds identically to the analogous construction for the quantum double of a finite-dimensional Hopf algebra; we refer the reader to the texts \cite{Majid-book}*{\S7}, \cite{KasBook95}*{\S IX.4}, \cite{KS-book}*{\S8.2} and \cite{Mont}*{\S10.3}, for instance, as well as the articles \cites{Radford93,Majid-90a, Majid-94}.

\subsection{The restricted quantum double of the Yangian}\label{ssec:D(Yhg)}
In our case, $U_\hbar\mfb=\hYhg$ is not a quantization of a finite-dimensional Lie bialgebra, but rather a homogeneous quantization of an $\N$-graded Lie bialgebra $\mfb$ with finite-dimensional graded components $\mfb_k$. In this setting, the double cross product construction alluded to above remains valid provided all duals are taken in the category of $\Z$-graded topological $\C[\![\hbar]\!]$-modules; that is, we replace $\chk{U_\hbar\mfb}$ with $\chk{\hYhg}$. This produces the \textit{restricted quantum double} $D(U_\hbar\mfb)$ of $U_\hbar\mfb$. 

Let us now give the detailed construction of this topological Hopf algebra. Following Section \ref{ssec:ad-coad}, let $\lad$ and $\rcoad$ denote the left adjoint action of $\hYhg$ on itself, and the right adjoint coaction of $\hYhg$ on itself, respectively:
\begin{gather*}
\lad= m^3\circ (\id^{\otimes 2}\otimes S)\circ (2\,3)\circ (\Delta \otimes \id),  \\
\rcoad=(1\otimes m)\circ (1\,2)\circ  (S\otimes \id^{\otimes 2})\circ \Delta^3,
\end{gather*}
where all tensor products are now taken to be the topological tensor product over $\C[\![\hbar]\!]$. By Lemma \ref{L:adjoint}, we have 
\begin{equation*}
\lad(\hYhg\otimes \QFhYhg)\subset \QFhYhg \quad \text{ and }\quad \rcoad(\hYhg)\subset \hYhg\otimes \QFhYhg.
\end{equation*}
We may thus  dualize  $\lad$ and $\rcoad$ to obtain the so-called left and right coadjoint actions
\begin{gather*}
\vartriangleright:\hYhg\otimes \chk{\hYhg}\to \chk{\hYhg} \quad \text{ and }\quad \vartriangleleft: \hYhg\otimes \chk{\hYhg}\to \hYhg,
\end{gather*}
respectively, of $\hYhg$ on $\chk{\hYhg}$ and  of $\chk{\hYhg}$ on $\hYhg$. These are defined on simple tensors by the formulas 
\begin{gather*}
x\vartriangleleft f= \id \otimes f \circ (\id \otimes S^{-1}) \circ \rcoad(x)\quad \text{ and }\quad  (x \vartriangleright \!f)(y)=f(S^{-1}(x)\cdot y) 
\end{gather*}
for all $x\in \hYhg$, $f\in \chk{\hYhg}$ and $y\in \QFhYhg$, 
where the action of $S^{-1}(x)$ on $y$ is given by $\lad$ and we have written $x \vartriangleright \!f$ for $\vartriangleright\!(x\otimes f)$ and $x\vartriangleleft f$ for $\vartriangleleft \!(x\otimes f)$. 

The tuple $(\hYhg,\chk{\hYhg},\rhd,\lhd)$ forms a \textit{matched pair} of $\Z$-graded topological Hopf algebras over $\C[\![\hbar]\!]$. Explicitly, $\rhd$ and $\lhd$ are $\Z$-graded homomorphisms of topological $\C[\![\hbar]\!]$-modules which satisfy the following set of conditions: 
\begin{enumerate}[label=(M\arabic*)]\setlength{\itemsep}{5pt}
\item $(\chk{\hYhg},\rhd)$ is a left $\hYhg$-module coalgebra and $(\hYhg,\lhd)$ is a right $\chk{\hYhg}$-module coalgebra. Equivalently, $\rhd$ and $\lhd$ are morphisms of topological coalgebras:
\begin{gather*}
\chk{\Delta}\circ \rhd = \rhd\otimes \rhd \circ (2\,3)\circ \Delta\otimes \chk{\Delta},\quad \Delta\circ \lhd = \lhd \otimes \lhd \circ (2\,3)\circ \Delta\otimes \chk{\Delta}\\\veps \otimes \chk{\veps}=\chk{\veps} \circ \rhd,\quad \veps \otimes \chk{\veps}=\veps \circ \lhd,
\end{gather*}
which make $\chk{\hYhg}$ a left $\hYhg$-module and $\hYhg$ a right $\chk{\hYhg}$-module.
\item $\lhd$ and $\rhd$ are compatible with the products $m$ and $\chk{m}$ on $\hYhg$ and $\chk{\hYhg}$:
\begin{gather*}
\lhd \circ (m\otimes \id) = m\circ (\lhd\otimes \id)\circ (\id \otimes \rhd\otimes \lhd)\circ (3\,4)\circ \id \otimes \Delta\otimes \chk{\Delta}\\
\rhd \circ (\id \otimes \chk{m})=\chk{m}\circ (\id \otimes \rhd)\circ (\rhd\otimes\lhd\otimes \id)\circ (2\,3)\circ \Delta\otimes \chk{\Delta}\otimes \id 
\end{gather*}
\item The unit maps $\iota$ and $\chk{\iota}$ are module homomorphisms:
\begin{equation*}
\lhd \circ (\iota\otimes \id) =\iota\otimes \chk{\veps} \quad \text{ and }\quad \rhd \circ\, ( \id \otimes \chk{\iota})=\veps \otimes \chk{\iota}.
\end{equation*}
\item $\lhd$ and $\rhd$ satisfy the compatibility relation
\begin{equation*}
(\lhd \otimes \rhd)\circ (2\,3)\circ \Delta\otimes\chk{\Delta} 
=
(1\,2)\circ (\rhd \otimes \lhd)\circ (2\,3)\circ  \Delta\otimes\chk{\Delta}.
\end{equation*}
\end{enumerate}
Here we note that the above conditions coincide with those from \cite{KasBook95}*{Def.~IX.2.2}, \cite{Radford93}*{\S2} and \cite{Andrea-Valerio-18}*{\S A.1}, which agree with those from \cite{Majid-book}*{Def.~7.2.1} up to conventions on the order in which the tensor factors appear. 

Since $\chk{\hYhg}$ and $\hYhg$ are matched, we may form the double cross product Hopf algebra $\chk{\hYhg}\bowtie \hYhg$ in the category of $\Z$-graded topological Hopf algebras over $\C[\![\hbar]\!]$ by following the standard procedure; see \cite{KasBook95}*{Thm.~IX.2.3} or \cite{Majid-book}*{Thm.~7.2.2}, for instance. As a $\Z$-graded topological coalgebra, $\chk{\hYhg}\bowtie \hYhg$ coincides with the tensor product of $\chk{\hYhg}$ and $\hYhg$:
\begin{gather*}
\chk{\hYhg}\bowtie \hYhg=\chk{\hYhg}\otimes \hYhg,\\
\Delta_D=(2\,3)\circ \chk{\Delta} \otimes \Delta \quad \text{ and }\quad \veps_D=\chk{\veps}\otimes \veps,
\end{gather*}
where $\Delta_D$ denotes the coproduct and $\veps_D$ the counit. The algebra structure on $\chk{\hYhg}\bowtie \hYhg$ is uniquely determined by the requirement that the inclusions
\begin{equation*}
\check{\imath}_D:\chk{\hYhg}\to \chk{\hYhg}\otimes \hYhg \quad \text{ and }\quad \imath_D:\hYhg\to \chk{\hYhg}\otimes \hYhg
\end{equation*}
are $\C[\![\hbar]\!]$-algebra homomorphisms satisfying the relations 
\begin{equation}\label{Dgr:m}
\begin{gathered}
m_{D}\circ (\check{\imath}_D\otimes \imath_D) = \id_{\chk{\hYhg}}\otimes \id_{\hYhg},\\
m_{D}\circ (\imath_D \otimes \check{\imath}_D)=(\rhd\otimes \lhd)\circ (2\,3)\circ \Delta \otimes \chk{\Delta},
\end{gathered}
\end{equation}
where $m_D$ denotes the multiplication on $D(\hYhg)$. 
In particular, the unit map is $\chk{\iota}\otimes \iota$. Finally, the antipode $S_D$ on $\chk{\hYhg}\bowtie \hYhg$ is given by 
\begin{equation*}
S_D=m_D\circ (\imath_D \otimes \check{\imath}_D)\circ (S\otimes \chk{S})\circ (1\,2).
\end{equation*}

This double cross product Hopf algebra is the restricted quantum double of the Yangian, as we formally record in the below definition.
\begin{definition}\label{D:D(Yhg)}
The $\Z$-graded topological Hopf algebra $\chk{\hYhg}\bowtie \hYhg$ is called the \textit{restricted quantum double} of the Yangian $\hYhg$, and is denoted 
\begin{equation*}
D(\hYhg):=\chk{\hYhg}\bowtie \hYhg.
\end{equation*}
\end{definition}
This definition, together with the general theory, implies that $D(\hYhg)$ provides a homogeneous quantization of the restricted Drinfeld double $\mft\cong D(\tplus)$ of the graded Lie bialgebra $(\tplus,\delta_{\scriptscriptstyle +})$, realized on the space $\tplus^\star\oplus \tplus \cong \tminus\oplus \tplus$\footnote{This also follows immediately from Theorem \ref{T:DYhg-DD}.}. 

One stipulation to carrying out the quantum double construction in the restricted homogeneous setting is that the canonical element $R$ associated to the pairing between $\QFhYhg$ and $\Yhgstar$, although formally satisfying the relations of \eqref{DH:3} in Section \ref{ssec:D(H)}, does not converge in the $\hbar$-adically complete tensor product $\QFhYhg\otimes \Yhgstar\subset D(\hYhg)\otimes D(\hYhg)$, and so $\DYhg$ is only \textit{topologically} quasitriangular. Here the prefix ``topologically'' is a bit subtle: it does not refer to the $\hbar$-adic topology and must be handled with care. In Section \ref{sec:R-matrix}, we shall identify $R$ with the universal $R$-matrix $\mcR(w-z)\in \Yhg^{\otimes 2}[w][\![z^{-1}]\!]$. This viewpoint allows for a precise interpretation of the  relations of \eqref{DH:3} in terms of those from Theorem \ref{T:Yhg-R}.
\begin{remark}\label{R:conventions}
There are many competing, though equivalent, variants of the definition of the quantum double. For instance, it may be realized on the tensor product $\hYhg\otimes \chk{\hYhg}$, as in \cite{DrQG}*{\S13}. Here we follow the conventions from \cites{KasBook95, Radford93}.
\end{remark}
%
%
%
\subsection{\texorpdfstring{$\mathrm{D}Y_\hbar\mfg$}{DYhg} as a quantum double}\label{ssec:DYhg=DD}

We now turn to proving the $\hbar$-adic variant of the main conjecture from \cite{KT96}*{\S2}, which postulates that $\DYhg$ and $D(\hYhg)$ are one and the same; see Theorem \ref{T:DYhg-DD} below.

Recall from the proof of Corollary \ref{C:chk-Phi-Hopf} that $\mathrm{L}(\cYhg{\vphantom{\hYhg}}^{\scriptscriptstyle{(2)}}_z)\subset (\hYhg\otimes \hYhg)[\![z^{\pm 1}]\!]$ denotes the $\Z$-graded topological $\C[\![\hbar]\!]$-algebra obtained by completing the graded subalgebra
$\mathds{L}(\cYhg{\vphantom{\hYhg}}^{\scriptscriptstyle{(2)}}_z)$ of  $\Yhg^{\otimes 2}[z;z^{-1}]\!]$  defined in \eqref{LzYhg-n}. Given this notation, we have $\C[\![\hbar]\!]$-linear maps
\begin{gather*}
\pi_z:\LzhYhg\otimes \hYhg\to \LzhYhg, \quad \pi_z:=m_z\circ (\id \otimes \tau_z),\\
\pi_z^\ast: \LzhYhg \to \mathrm{L}(\cYhg{\vphantom{\hYhg}}^{\scriptscriptstyle{(2)}}_z), \quad \pi_z^\ast(y)=(1\otimes y)\mcR(-z),
\end{gather*}
where $m_z$ denotes the multiplication in $\LzhYhg$. These obey a Drinfeld--Yetter right action/left coaction compatibility condition, as the next lemma makes explicit. 
\begin{lemma}\label{L:Dr-Yet}
The pair $(\pi_z,\pi_z^\ast)$ satisfy the relation
\begin{equation*}
\label{Dr-Yet}
\pi_z^\ast \circ \pi_z\circ (1\,2) = m\otimes \pi_z \circ (4\,3\,2)\circ \dot\rcoad \otimes \dot{\lad}\otimes \id \circ  \op{\Delta}\otimes \pi_z^\ast
\end{equation*}
in $\Hom_{\C[\![\hbar]\!]}(\hYhg\otimes \LzhYhg, \mathrm{L}(\cYhg{\vphantom{\hYhg}}^{\scriptscriptstyle{(2)}}_z))$, where we have set
\begin{equation*}
\dot\rcoad:=(S^{-1}\otimes \id)\circ (1\,2)\circ \blacktriangle\quad \text{ and }\quad
\dot{\lad}:=\lad\circ (S^{-1}\otimes \id).
\end{equation*}
\end{lemma}
\begin{proof}
This follows from a straightforward modification of the proof of (9) in \cite{Radford93}*{Lem.~1}, using the antipode relations 
\begin{gather*}
m\circ (\id \otimes S^{-1}) \circ \op{\Delta}=\iota\circ \veps=m\circ (S^{-1}\otimes \id) \circ \op{\Delta},
\end{gather*}
the counit relation $m\circ (\id \otimes \veps) \circ \Delta=\id$, and 
the intertwining relation 
\begin{equation}\label{inter}
1\otimes \tau_z \circ \Delta(x)=\mcR(-z)^{-1} \cdot 1\otimes \tau_z \circ \op{\Delta}(x) \cdot \mcR(-z)
\quad \forall \; x\in \hYhg,
\end{equation}
which itself follows from the identities \eqref{Yhg-R:1} and \eqref{R-inter} of Theorem \ref{T:Yhg-R}. 

To illustrate this, we apply both sides of the claimed identity to a simple tensor $x\otimes y\in \hYhg\otimes \LzhYhg$, while exploiting the sumless Sweedler notation $\Delta^n(x)=x_1\otimes \cdots \otimes x_n$ for iterated coproducts. Expanding the right-hand side of the resulting expression, while using that ${\dot\lad}(x\otimes y)=S^{-1}(x_2)yx_1$ and ${\dot\rcoad}(x)=S^{-1}(x_3)x_1\otimes x_2$, we obtain
\begin{align*}
S^{-1}(x_5)x_3S^{-1}(x_2)\otimes y &\cdot \mcR(-z)\cdot x_1 \otimes \tau_z(x_4)\\
&= S^{-1}(x_4)\otimes y \cdot \mcR(-z) \cdot x_1\veps(x_2) \otimes \tau_z(x_3)\\
&=S^{-1}(x_3)\otimes y \cdot \mcR(-z) \cdot x_1\otimes \tau_z(x_2)\\
&=S^{-1}(x_3)x_2\otimes y\tau_z(x_1) \cdot \mcR(-z)\\
&=1 \otimes y\tau_z(x_1\veps(x_2)) \cdot \mcR(-z)\\
&=1\otimes y\tau_z(x) \cdot \mcR(-z),
\end{align*} 
where we have applied the intertwining relation \eqref{inter} in the third equality, and the appropriate Hopf relations listed above in each of the remaining equalities. Since 
\begin{equation*}
\pi_z^\ast \circ \pi_z(y\otimes x)=\pi_z^\ast(y\tau_z(x))=(1\otimes y\tau_z(x))\mcR(-z),  
\end{equation*}
the above computation implies the lemma. \qedhere
\end{proof}

We now come to the main theorem of this section. Recall from Definition \ref{D:D(Yhg)} that we may identify the restricted quantum double $D(\hYhg)$ with $\chk{\hYhg}\otimes\hYhg$ as a topological coalgebra over $\C[\![\hbar]\!]$. We further recall that 
\begin{equation*}
\imath: \hYhg\into \DYhg \quad \text{ and }\quad \check{\imath}:\chk{\hYhg}\into \DYhg
\end{equation*}
are the $\Z$-graded embeddings of topological Hopf algebras over $\C[\![\hbar]\!]$ which featured prominently in Theorem \ref{T:dual}. 
\begin{theorem}\label{T:DYhg-DD}
The $\C[\![\hbar]\!]$-module homomorphism 
\begin{equation*}
\Upsilon:=m\circ (\check{\imath}\otimes \imath): \chk{\hYhg}\otimes\hYhg \to \DYhg
\end{equation*}
is an isomorphism  of $\Z$-graded topological Hopf algebras $D(\hYhg)\iso \DYhg$.
\end{theorem}
\begin{proof}

We shall first prove that $\Upsilon$ is an isomorphism of $\Z$-graded topological coalgebras over $\C[\![\hbar]\!]$, which follows from Theorem \ref{T:dual}. Afterwards, we will use Lemma \ref{L:Dr-Yet} to complete the proof of the theorem by proving that $\Upsilon$ is an algebra homomorphism $D(\hYhg)\to \DYhg$. Here we note that it follows automatically that $\Upsilon$ intertwines the underlying antipodes, though this is straightforward to verify.  

\begin{proof}[Proof that $\Upsilon$ is an isomorphism of $\Z$-graded coalgebras] \let\qed\relax

Since $m$, $\imath$ and $\check{\imath}$ are all $\Z$-graded $\C[\![\hbar]\!]$-module homomorphisms, the same is true of $\Upsilon$. Moreover, as the coproduct $\dot{\Delta}$ of $\DYhg$ is an algebra homomorphism, it satisfies 
\begin{align*}
\dot{\Delta} \circ \Upsilon &= \dot{\Delta} \circ m \circ (\check{\imath} \otimes \imath)\\
&
=m\otimes m \circ  (2\,3)\circ (\dot{\Delta} \otimes \dot{\Delta})\circ (\check{\imath} \otimes \imath)\\
&
=m\otimes m \circ  (2\,3)\circ (\check{\imath}\otimes \check{\imath} \otimes \imath\otimes \imath) \circ (\chk{\Delta} \otimes \Delta)\\
&
=\Upsilon \otimes \Upsilon \circ (2\,3) \circ \chk{\Delta} \otimes \Delta.
\end{align*}
Similarly, as the counit $\dot{\veps}$ of $\DYhg$ is an algebra homomorphism, one has 
\begin{align*}
\dot{\veps} \circ \Upsilon&=\dot{\veps} \circ m \circ (\check{\imath} \otimes \imath)
=m \circ (\dot{\veps} \otimes \dot{\veps})\circ (\check{\imath} \otimes \imath)
=m \circ \chk{\veps}\otimes \veps.
\end{align*}
This proves that $\Upsilon$ is a $\Z$-graded homomorphism of topological coalgebras.
Now let us turn to establishing the bijectivity of $\Upsilon$. As $\chk{\hYhg}\otimes \hYhg$ is a topologically free $\C[\![\hbar]\!]$-module with semiclassical limit $U(\tminus)\otimes_\C U(\tplus)$ and $\DYhg$ is a flat deformation of $U(\mft)$, Lemma \ref{L:scl-map} asserts that it is sufficient to establish that the semiclassical limit 
\begin{equation*}
\bar{\Upsilon}=\bar{m}\circ (\check{\underline{\imath}}\otimes \underline{\imath}):U(\tminus)\otimes_\C U(\tplus)\to U(\mft)
\end{equation*}
of $\Upsilon$ is an isomorphism, where $\bar{m}, \check{\underline{\imath}}$ and $\underline{\imath}$ are the semiclassical limits of $m$, $\check{\imath}$ and $\imath$, respectively. As $\check{\underline{\imath}}$ and $\underline{\imath}$ quantize the natural inclusions of $\tminus$ and $\tplus$ into $\mft$ (as was shown in the proof of Lemma \ref{L:Yhg*->DYhg} for $\check{\underline{\imath}}$) and $\bar{m}$ is the multiplication map on $U(\mft)$,  this follows from the Poincar\'{e}--Birkhoff--Witt Theorem for enveloping algebras and the decomposition $\mft=\tminus\oplus \tplus$. 
\end{proof}

\begin{proof}[Proof that $\Upsilon$ is an algebra homomorphism]

By \eqref{Dgr:m}, we must show that $\Upsilon\circ \imath_D$ and $\Upsilon\circ \check{\imath}_D$ are $\C[\![\hbar]\!]$-algebra homomorphisms and that, in addition, $\Upsilon$ satisfies the relations
\begin{equation}\label{suff-Up}
\begin{gathered}
\Upsilon\circ m_{D}\circ (\check{\imath}_D\otimes \imath_D) = m\circ  (\Upsilon \otimes \Upsilon) \circ  (\check{\imath}_D\otimes \imath_D),\\
m\circ (\Upsilon \otimes \Upsilon)\circ (\imath_D \otimes \check{\imath}_D)=m\circ (\Upsilon \otimes \Upsilon) \circ (\rhd\otimes \lhd)\circ (2\,3)\circ \Delta \otimes \chk{\Delta}.
\end{gathered}
\end{equation}
By definition of $\Upsilon$, we have $\Upsilon\circ \imath_D=\imath$ and $\Upsilon\circ \check{\imath}_D=\check{\imath}$. Moreover, both sides of the identity 
\begin{equation*}
\Upsilon\circ m_{D}\circ (\check{\imath}_D\otimes \imath_D)=m\circ  (\Upsilon \otimes \Upsilon) \circ  (\check{\imath}_D\otimes \imath_D)
\end{equation*}
coincide with $\Upsilon$, viewed as a map $\chk{\hYhg}\otimes \hYhg\to \DYhg$. Hence, we are left to verify the second relation of \eqref{suff-Up}, which we shall establish by appealing to the injective $\C[\![\hbar]\!]$-algebra homomorphism $\Phi_z$ from Theorem \ref{T:Phiz}. Namely, it is enough to show that 
\begin{equation*}
\Phi_z\circ m\circ (\Upsilon \otimes \Upsilon)\circ (\imath_D \otimes \check{\imath}_D)=\Phi_z\circ m\circ (\Upsilon \otimes \Upsilon) \circ (\rhd\otimes \lhd)\circ (2\,3)\circ \Delta \otimes \chk{\Delta}.
\end{equation*}
 Since $\Phi_z$ is an algebra homomorphism satisfying $\Phi_z\circ \imath=\tau_z$ and $\Phi_z\circ \check{\imath}=\chk{\Phi_z}$ (see Lemma \ref{L:Yhg*->DYhg}), this is equivalent to 
\begin{equation}\label{suff-Psi}
m_z\circ (\tau_z\otimes \chk{\Phi_z})=m_z\circ ( \chk{\Phi_z}\otimes \tau_z) \circ (\rhd\otimes \lhd)\circ (2\,3)\circ \Delta \otimes \chk{\Delta},
\end{equation}
where $m_z$ is the product in $\LzhYhg$. The proof that \eqref{suff-Psi} is satisfied follows an argument parallel to that employed to establish Part (a) of Theorem 2 in \cite{Radford93}, using Lemma \ref{L:Dr-Yet} in place of \cite{Radford93}*{(9)}. For the sake of completeness, we give a complete argument below.

Applying the left-hand side of \eqref{suff-Psi} to an arbitrary simple tensor $x\otimes f$, we obtain
\begin{equation*}
\tau_z(x)\chk{\Phi_z}(f)=(f\otimes \tau_z(x))\mcR(-z)=(f\otimes \id) \circ \pi_z^\ast \circ \pi_z(1\otimes x),
\end{equation*}
which, by Lemma \ref{L:Dr-Yet}, may be rewritten as 
\begin{equation}\label{suff-Psi:2}
\tau_z(x)\chk{\Phi_z}(f)=f\otimes \id \circ m\otimes \pi_z \circ (4\,3\,2)\circ \dot\rcoad \otimes \dot{\lad}\otimes \id \circ  \op{\Delta}(x)\otimes \mcR(-z).
\end{equation}
On the other hand,  by definition of $\chk{\Phi_z}$, $\rhd$  and $\lhd$, the right-hand side of \eqref{suff-Psi} evaluated on $x\otimes f$ is 
\begin{equation*}
\chk{\Phi_z}(x_1\rhd f_1)\cdot \tau_z(x_2\lhd f_2)=(f_1\otimes \id \circ \dot{\lad}\otimes \id(x_1\otimes \mcR(-z)))\cdot (f_2\otimes \tau_z \circ \dot{\rcoad}(x_2)),
\end{equation*}
where we have employed the sumless Sweedler notation $\Delta(x)=x_1\otimes x_2$ and $\chk{\Delta}(f)=f_1\otimes f_2$.
That this coincides with the expression \eqref{suff-Psi:2} for $\tau_z(x)\chk{\Phi_z}(f)$ is a consequence of the following general computation. For each $a\in \QFhYhg$ and $b\in \hYhg$, we have
\begin{align*}
(f_1(x_1 \bullet a)\otimes b)\cdot & (f_2\otimes \tau_z \circ \dot{\rcoad}(x_2))\\
&=f(x_2^{(1)} (x_1\bullet a)) b\tau_z(x_2^{(2)})\\
&=f\otimes \id \circ m\otimes \pi_z(x_2^{(1)}\otimes (x_1 \bullet a) \otimes b \otimes x_2^{(2)})\\
%
&=(f\otimes \id \circ m\otimes \pi_z \circ (4\,3\,2)\circ \dot\rcoad \otimes \dot{\lad}\otimes \id
\circ \op{\Delta}\otimes \id \otimes \id  ) (  x\otimes a\otimes b),
\end{align*}
where we have set $x\bullet a= \dot{\lad}(x\otimes a)$ and in the first and second equalities we have used the (sumless) Sweedler type notation $\dot{\rcoad}(x)=x^{(1)}\otimes x^{(2)}$. This completes the proof of \eqref{suff-Psi}, and thus the proof that $\Upsilon$ is an algebra homomorphism. \qedhere
\end{proof}
\let\qed\relax
\end{proof}
%

\section{The universal \texorpdfstring{$R$}{R}-matrix}\label{sec:R-matrix}

In this final section, we establish the last assertion of Theorem \ref{T:Intro}, which identifies the universal $R$-matrix $\boldR$ of the Yangian double $\DYhg\cong D(\hYhg)$ with Drinfeld's universal $R$-matrix $\mcR(z)$; see Theorem \ref{T:R}. Though this in fact follows without too much effort from the results of Sections \ref{sec:DYhg-quant} and \ref{sec:DYhg-double}, constructing $\boldR$ precisely does require some care. After laying the groundwork in Sections \ref{ssec:R-Theta} and \ref{ssec:Theta-chi}, we define $\boldR$ and prove our last main result in Section \ref{ssec:R-can}. The final two subsections --- Sections \ref{ssec:R-compute} and \ref{ssec:Rminus-beta} --- are devoted to providing additional context pertinent to this theorem. 
%
\subsection{The isomorphism \texorpdfstring{$\Theta$}{Theta}}\label{ssec:R-Theta}

Let $\theta$ denote the canonical $\C[\![\hbar]\!]$-module injection 
\begin{equation*}
\theta: \hYhg\otimes \Yhgstar\into \Hom_{\C[\![\hbar]\!]}(\QFhYhg,\hYhg)
\end{equation*}
determined on simple tensors by $\theta(x\otimes f)(y)=f(y)x$. This injection is a homomorphism of topological $\C[\![\hbar]\!]$-algebras provided we equip $\Hom_{\C[\![\hbar]\!]}(\QFhYhg,\hYhg)$ with the convolution product 
\begin{equation*}
\varphi_1\star \varphi_2:= m \circ (\varphi_1 \otimes \varphi_2) \circ \Delta|_{\QFhYhg} \quad \forall\; \varphi_1,\varphi_2\in \Hom_{\C[\![\hbar]\!]}(\QFhYhg,\hYhg)
\end{equation*}
and identity element  $\iota\circ \veps$, where $m$, $\Delta$, $\iota$ and $\veps$ are the multiplication, coproduct, unit and counit of $\hYhg$, respectively. Our main goal in this subsection is to identity a natural extension of $\theta$ which is an isomorphism.

Let us begin by noting some topological properties of the homomorphism space $\Hom_{\C[\![\hbar]\!]}(\QFhYhg,\hYhg)$. As $\hYhg$ and $\QFhYhg$ are topologically free with semiclassical limits $U(\tplus)$ and $\Symt$, respectively, $\Hom_{\C[\![\hbar]\!]}(\QFhYhg,\hYhg)$ is a topologically free $\C[\![\hbar]\!]$-modules with 
\begin{equation}\label{E:scl}
\Hom_{\C[\![\hbar]\!]}(\QFhYhg,\hYhg)\cong \left(\Hom_\C(\Symt,U(\tplus))\right)\![\![\hbar]\!].
\end{equation}
In addition, $\Hom_{\C[\![\hbar]\!]}(\QFhYhg,\hYhg)$ is a Hausdorff and complete topological space with respect to the topology associated to the descending filtration $\mathbb{E}_\bullet$ of closed $\C[\![\hbar]\!]$-submodules defined by 
\begin{equation}\label{E_bullet}
\mathbb{E}_n:=\{f\in \Hom_{\C[\![\hbar]\!]}(\QFhYhg,\hYhg): f(\mbJ_{n})=0\},
\end{equation}
where we have set $\mbJ_n=\bigoplus_{k< n}\QFYhg_k$ with $\mbJ_{0}=\{0\}$. Said in more algebraic terms, the natural $\C[\![\hbar]\!]$-module homomorphism 
\begin{equation*}
\Hom_{\C[\![\hbar]\!]}(\QFhYhg,\hYhg)\to \varprojlim_{n} \left(\Hom_{\C[\![\hbar]\!]}(\QFhYhg,\hYhg)/\mathbb{E}_n\right)
\end{equation*} 
is an isomorphism, as is readily verified. Moreover, as the coproduct $\Delta$ on $\hYhg$ is graded, $\mathbb{E}_\bullet$ is a descending filtration of ideals and the above is an isomorphism of $\C[\![\hbar]\!]$-algebras. 

We now turn towards enlarging the domain of $\theta$. 
For each $n\in \Z$, let $\Yhgstar_{\scriptscriptstyle (n)}$ denote the closure of the $\C[\![\hbar]\!]$-submodule of $\Yhgstar$ generated by $\bigoplus_{k\geq n} \Yhgstar_{-k}$. Since $\Yhgstar_k\subset \hbar\Yhgstar_{k-1}$ for $k>0$, $\Yhgstar_{\scriptscriptstyle (0)}=\Yhgstar$ and we have a descending $\N$-filtration 
\begin{equation*}
\Yhgstar=\Yhgstar_{\scriptscriptstyle (0)}\supset \Yhgstar_{\scriptscriptstyle (1)} \supset \cdots \supset \Yhgstar_{\scriptscriptstyle (n)} \supset \cdots 
\end{equation*}
We may thus introduce the topological tensor product
\begin{equation*}
\hYhg \,\wt{\otimes}\, \Yhgstar:= \varprojlim_n (\hYhg\otimes \Yhgstar / \hYhg\otimes \Yhgstar_{\scriptscriptstyle(n)}).
%
\end{equation*}
Since $\Yhgstar$ is topologically generated as a $\C[\![\hbar]\!]$-algebra by the space $\bigoplus_{k>0} \Yhgstar_{-k}$ and, for each $n\in \N$, $\bigoplus_{k\geq n} \Yhgstar_{-k}$ is stable under multiplication by any $x$ in this space, we see that each $\Yhgstar_{\scriptscriptstyle (n)}$ is an ideal in $\Yhgstar$. It follows that $\hYhg \,\wt{\otimes}\, \Yhgstar$ is a $\C[\![\hbar]\!]$-algebra.

To see that $\hYhg \,\wt{\otimes}\, \Yhgstar$ is a topologically free $\C[\![\hbar]\!]$-module, let us introduce the classical spaces $\Symt{\vphantom{)}}^\star_{\scriptscriptstyle(n)}$ and $U(\tplus)\,\wt{\otimes}\,\Symt{\vphantom{)}}^\star$ by setting
\begin{gather*}
\Symt{\vphantom{)}}^\star_{\scriptscriptstyle(n)}:= \bigoplus_{k\geq n}\Symt{\vphantom{)}}^\star_{{-k}},\\
U(\tplus)\,\wt{\otimes}\,\Symt{\vphantom{)}}^\star:=\prod_{n\in \N} U(\tplus)\otimes_\C \Symt{\vphantom{)}}^\star_{-n} 
\cong \varprojlim_{n} U(\tplus)\otimes_\C (\Symt{\vphantom{)}}^\star/\Symt{\vphantom{)}}^\star_{\scriptscriptstyle (n)}).
\end{gather*}
\begin{lemma}\label{L:wt-otimes}
The $\C[\![\hbar]\!]$-algebra $\hYhg \,\wt{\otimes}\, \Yhgstar$ is topologically free with 
%
\begin{gather*}
\hYhg \,\wt{\otimes}\, \Yhgstar/\hbar(\hYhg \,\wt{\otimes}\, \Yhgstar)\cong U(\tplus)\,\wt{\otimes}\,\Symt{\vphantom{)}}^\star.
\end{gather*}
%
%
%
\end{lemma}
\begin{proof}
Let $\upnu_\star:\Yhgstar\iso \Symt{\vphantom{)}}^\star[\![\hbar]\!]$ be any fixed $\Z$-graded isomorphism of topological $\C[\![\hbar]\!]$-modules. We claim that
\begin{equation}\label{Yhgstar-n-scl}
\upnu_\star(\Yhgstar_{\scriptscriptstyle (n)})=\Symt{\vphantom{)}}^\star_{\scriptscriptstyle(n)}[\![\hbar]\!].
\end{equation}
To see this, fix $n\in \N$. Since $\upnu_\star$ is graded, we have 
\begin{equation*}
\upnu_\star(\Yhgstar_{-k})=\prod_{a\in \N}\Symt{\vphantom{)}}^\star_{{-k-a}}\hbar^a \subset \Symt{\vphantom{)}}^\star_{\scriptscriptstyle(n)}[\![\hbar]\!] \quad \forall\; k\geq n.
\end{equation*}
As $\Symt{\vphantom{)}}^\star_{\scriptscriptstyle(n)}[\![\hbar]\!]$ is a closed $\C[\![\hbar]\!]$-submodule of $\Symt{\vphantom{)}}^\star[\![\hbar]\!]$, we get $\upnu_\star(\Yhgstar_{\scriptscriptstyle{(n)}})\subset \Symt{\vphantom{)}}^\star_{\scriptscriptstyle(n)}[\![\hbar]\!]$. Conversely, if $\sum_{a\geq 0}x_a\hbar^a\in \Symt{\vphantom{)}}^\star_{\scriptscriptstyle(n)}[\![\hbar]\!]$, then $y_a:=\upnu_\star^{-1}(x_a)\in \bigoplus_{k\geq n} \Yhgstar_{-k}$ for each $a\in \N$. Therefore, we have 
\begin{equation*}
\upnu_\star(\sum_{a\geq 0} x_a \hbar ^a)= \sum_{a\geq 0} y_a \hbar ^a\in \Yhgstar_{\scriptscriptstyle(n)},
\end{equation*}
which completes the proof of \eqref{Yhgstar-n-scl}. The statement of the lemma now follows from the sequence of isomorphisms
\begin{align*}
\hYhg\otimes \Yhgstar / \hYhg  \otimes \Yhgstar_{\scriptscriptstyle(n)}&\cong (U(\tplus)\otimes_\C \Symt{\vphantom{)}}^\star)[\![\hbar]\!]/(U(\tplus)\otimes_\C \Symt{\vphantom{)}}^\star_{\scriptscriptstyle (n)})[\![\hbar]\!]
\\
&\cong \left(U(\tplus)\otimes_\C (\Symt{\vphantom{)}}^\star/\Symt{\vphantom{)}}^\star_{\scriptscriptstyle (n)})\right)[\![\hbar]\!].\qedhere
\end{align*}
\end{proof}
%
%
%
%
The next result outputs the desired extensions of $\theta$. 
\begin{proposition}\label{P:Theta}
The injection $\theta$ extends to an isomorphism of $\C[\![\hbar]\!]$-algebras
\begin{gather*}
\Theta:\hYhg \,\wt{\otimes}\, \Yhgstar\iso \Hom_{\C[\![\hbar]\!]}(\QFhYhg,\hYhg).
\end{gather*}
\end{proposition}
\begin{proof}
Let $\mathbb{E}_\bullet$ be the filtration of $\mathbb{E}:=\Hom_{\C[\![\hbar]\!]}(\QFhYhg,\hYhg)$ defined in \eqref{E_bullet}. Then
\begin{equation*}
\theta(\hYhg\otimes \Yhgstar_{\scriptscriptstyle(n)})\subset \mathbb{E}_n \quad \forall\; n\in \N. 
\end{equation*}
Indeed, this follows from the definition of $\theta$ and the fact that 
if $f\in \bigoplus_{k\geq n}\Yhgstar_{-k}$ and $x\in \mbJ_{n}$, then $f(x)=0$. Hence, we obtain 
\begin{equation*}
\Theta:\hYhg \,\wt{\otimes}\, \Yhgstar =\varprojlim (\hYhg \otimes \Yhgstar / \hYhg \otimes \Yhgstar_{\scriptscriptstyle(n)}) \to \varprojlim \mathbb{E}/\mathbb{E}_n = \mathbb{E}.
\end{equation*}
As $\hYhg$ and $\mathbb{E}$ are topologically free, to show $\Theta$ is an isomorphism it is sufficient to show its semiclassical limit $\bar\Theta$ is; see Lemma \ref{L:scl-map}. Employing the identifications of \eqref{E:scl} and Lemma \ref{L:wt-otimes}, we see that $\bar{\Theta}$ coincides with the isomorphism
\begin{equation*}
U(\tplus)\,\wt{\otimes} \,\Symt{\vphantom{)}}^\star\iso \Hom_\C (\Symt,U(\tplus))
\end{equation*}
uniquely extending the canonical map $U(\tplus)\otimes_\C \Symt{\vphantom{)}}^\star\into \Hom_\C (\Symt,U(\tplus))$, which is the semiclassical limit $\bar\theta$ of $\theta$, by continuity.  \qedhere
\end{proof}
We now give two remarks pertinent to the above discussion. 
%
%
%
%
\begin{remark}\label{R:Theta}
By Proposition \ref{P:Theta}, we may introduce the $\C[\![\hbar]\!]$-subalgebra 
\begin{equation*}
\QFhYhg\,\wt{\otimes}\, \Yhgstar:= \Theta^{-1}(\End_{\C[\![\hbar]\!]}(\QFhYhg)) \subset \hYhg\,\wt{\otimes}\, \Yhgstar. 
\end{equation*}
By definition, it comes equipped with an isomorphism
\begin{equation*}
\Theta|_{\QFhYhg \,\wt{\otimes}\, \Yhgstar}:\QFhYhg \,\wt{\otimes}\, \Yhgstar\iso \End_{\C[\![\hbar]\!]}(\QFhYhg)
\end{equation*}
extending the injection $\theta|_{\QFhYhg\otimes \Yhgstar}$. 
It is worth noting that, by following the above arguments, it is easy to see that $\QFhYhg\,\wt{\otimes}\, \Yhgstar$ admits the equivalent description 
\begin{equation*}
\QFhYhg\,\wt{\otimes}\, \Yhgstar\cong  \varprojlim_n(\QFhYhg\otimes \Yhgstar / \QFhYhg\otimes \Yhgstar_{\scriptscriptstyle(n)}).
\end{equation*}
Suppose now that $\varphi,\gamma\in\End_{\C[\![\hbar]\!]}(\QFhYhg)$ with $\gamma$ homogeneous of degree $a\in \Z$. Then  $\varphi\otimes \gamma^t$ uniquely extends to an element of $\End_{\C[\![\hbar]\!]}(\QFhYhg\,\wt{\otimes}\,\Yhgstar)$ and $\Theta$ has the property that
\begin{equation}\label{Theta-prop}
\varphi \circ \Theta(y)\circ \gamma=(\Theta \circ (\varphi \otimes\gamma^t))(y) \quad \forall\; y\in \QFhYhg \,\wt{\otimes}\,\Yhgstar. 
\end{equation}
\end{remark}
\begin{remark}\label{R:End(hYhg)}
The argument used to establish Proposition \ref{P:hYhg*->QFhYhg*} implies that the natural $\C[\![\hbar]\!]$-algebra homomorphism 
\begin{equation*}
\End_{\C[\![\hbar]\!]}(\hYhg)\to \Hom_{\C[\![\hbar]\!]}(\QFhYhg,\hYhg), \quad \varphi\mapsto \varphi|_{\QFhYhg},
\end{equation*}
is an embedding. Hence, we can (and shall) view $\End_{\C[\![\hbar]\!]}(\hYhg)$ as a subalgebra of $\Hom_{\C[\![\hbar]\!]}(\QFhYhg,\hYhg)$. We may thus introduce $\hYhg\,\wt{\otimes}\,\hYhg{\vphantom{)}}^\star\subset \hYhg\,\wt{\otimes}\, \Yhgstar$ by setting
\begin{equation*}
\hYhg\,\wt{\otimes}\,\hYhg{\vphantom{)}}^\star:=\Theta^{-1}(\End_{\C[\![\hbar]\!]}(\hYhg)). 
\end{equation*}
We then have 
\begin{equation*}
(\hYhg\,\wt{\otimes}\,\hYhg{\vphantom{)}}^\star)\cap (\QFhYhg\,\wt{\otimes}\, \Yhgstar)\cong \End_{\C[\![\hbar]\!]}^{\QFhYhg}(\hYhg),
\end{equation*}
where the right-hand side consists of all $f\in \End_{\C[\![\hbar]\!]}(\hYhg)$ for which $f(\QFhYhg)\subset \QFhYhg$. 
\end{remark}
%
\subsection{The restriction \texorpdfstring{$\Theta_\chi$}{Theta-chi}}\label{ssec:Theta-chi}
We now give a few comments which concern the triangular decompositions of $\QFhYhg$ and $\Yhgstar$. For each choice of the symbol $\chi$, let $\pi^\chi:\QFhYhg\to \dhYh{\chi}$ denote the $\C[\![\hbar]\!]$-linear projection associated to the triangular decomposition of $\QFhYhg$, as defined in Section \ref{ssec:Yhg*-TD}. These projections give rise to $\C[\![\hbar]\!]$-module embeddings
\begin{equation*}
\Hom_{\C[\![\hbar]\!]}(\dhYh{\chi},Y_\hbar^\chi\mfg)\into \Hom_{\C[\![\hbar]\!]}(\QFhYhg,\hYhg), \quad \varphi^\chi\mapsto \varphi^\chi \circ \pi^\chi.
\end{equation*}
We shall henceforth adopt the viewpoint that $\Hom_{\C[\![\hbar]\!]}(\dhYh{\chi},Y_\hbar^\chi\mfg)$ is a submodule of $\Hom_{\C[\![\hbar]\!]}(\QFhYhg,\hYhg)$, with the above identification assumed. The restriction $\theta_\chi$ of $\theta$ to $Y_\hbar^\chi\mfg\otimes \YTDstar{\chi}$ is then a $\C[\![\hbar]\!]$-module injection 
\begin{equation*}
\theta_\chi: Y_\hbar^\chi\mfg\otimes \YTDstar{\chi}\into \Hom_{\C[\![\hbar]\!]}(\dhYh{\chi},Y_\hbar^\chi\mfg).
\end{equation*}
Following the above procedure, we may introduce the $\C[\![\hbar]\!]$-algebra 
\begin{equation*}
Y_\hbar^\chi\mfg \,\wt{\otimes}\, \YTDstar{\chi}:= \varprojlim_n (Y_\hbar^\chi\mfg \otimes \YTDstar{\chi} / Y_\hbar^\chi\mfg \otimes \YTDstar{\chi}_{\scriptscriptstyle(n)}),
\end{equation*}
where $\YTDstar{\chi}_{\scriptscriptstyle(n)}$ is the closure of the $\C[\![\hbar]\!]$-submodule of $\YTDstar{\chi}$ generated by the direct sum $\bigoplus_{k\geq n}\YTDstar{\chi}_{-k}$. The above arguments show that this completed tensor product is a flat $\C[\![\hbar]\!]$-algebra deformation of 
\begin{equation*}
U(\tplus^\chi)\,\wt{\otimes}\, \msS(\hbar\tplus^\chi){\vphantom{)}}^\star=\prod_{n\in \N}U(\tplus^\chi)\otimes_\C \msS(\hbar\tplus^\chi){\vphantom{)}}^\star_{-n}.
\end{equation*}
It follows that the natural algebra homomorphism $Y_\hbar^\chi\mfg \,\wt{\otimes}\, \YTDstar{\chi}\to \hYhg\,\wt{\otimes}\,\Yhgstar$ is injective 
and, by the proof of Proposition \ref{P:Theta}, that $\theta_\chi$ uniquely extends to an isomorphism of $\C[\![\hbar]\!]$-modules 
\begin{equation}\label{Theta^chi}
\Theta_\chi: Y_\hbar^\chi\mfg\,\wt{\otimes}\, \YTDstar{\chi}\iso \Hom_{\C[\![\hbar]\!]}(\dhYh{\chi},Y_\hbar^\chi\mfg)
\end{equation}
which coincides with the restriction of $\Theta$ to $Y_\hbar^\chi\mfg\,\wt{\otimes}\, \YTDstar{\chi}$. In particular, 
since $Y_\hbar^\chi\mfg\,\wt{\otimes}\, \YTDstar{\chi}$ is a subalgebra of  $\hYhg\,\wt{\otimes}\,\Yhgstar$, the space $\Hom_{\C[\![\hbar]\!]}(\dhYh{\chi},Y_\hbar^\chi\mfg)$ is a subalgebra of $\Hom_{\C[\![\hbar]\!]}(\QFhYhg,\hYhg)$ and $\Theta_\chi$ is an isomorphism of $\C[\![\hbar]\!]$-algebras. 

Finally, as in Remarks \ref{R:Theta} and \ref{R:End(hYhg)}, we set
\begin{gather*}
\dhYh{\chi}\,\wt{\otimes}\, \YTDstar{\chi}:=\Theta^{-1}_\chi(\End_{\C[\![\hbar]\!]}(\dhYh{\chi}))\cong \varprojlim_n(\dhYh{\chi}\otimes \YTDstar{\chi} / \dhYh{\chi}\otimes \YTDstar{\chi}_{\scriptscriptstyle(n)}),\\
Y_\hbar^\chi\mfg\,\wt{\otimes}\, Y_\hbar^\chi\mfg{\vphantom{)}}^\star:=\Theta^{-1}_\chi(\End_{\C[\![\hbar]\!]}(Y_\hbar^{\chi}\mfg)).
\end{gather*} 
%

%
%
\subsection{Canonical tensors and universal R-matrices}\label{ssec:R-can}
By Proposition \ref{P:Theta}, we may introduce $R$ and $R^\chi$ in $\hYhg\,\wt{\otimes}\,\Yhgstar $, for each choice of $\chi$, as the elements
\begin{gather*}
R:= \Theta^{-1}(\id_{\QFhYhg})\in \hYhg\,\wt{\otimes}\,\Yhgstar \quad \text{ and}\quad R^\chi:=\Theta^{-1}(\id_{\dhYh{\chi}})\in Y_\hbar^\chi\mfg\,\wt{\otimes}\, \YTDstar{\chi}. 
\end{gather*}
That is, $R$ and $R^\chi$ are the canonical elements associated to the Hopf pairing $\QFhYhg\times \Yhgstar\to \C[\![\hbar]\!]$ and its restriction to $\dhYh{\chi}\times \YTDstar{\chi}$, respectively.  

As $\id_{\QFhYhg}$ and $\id_{\dhYh{\chi}}$ coincide with $\id_{\hYhg}$ and $\id_{Y_\hbar^\chi\mfg}$ under the identifications of Remark \ref{R:End(hYhg)} and Section \ref{ssec:Theta-chi}, respectively,  $R$ (resp. $R^\chi$) is also the canonical element defined by the pairing between $\hYhg$ and its restricted dual $\hYhg{\vphantom{)}}^\star$ (resp. $Y_\hbar^\chi\mfg$ and $Y_\hbar^\chi\mfg{\vphantom{)}}^\star$). In particular, we have
\begin{equation*}
R\in (\hYhg\,\wt{\otimes}\,\hYhg{\vphantom{)}}^\star)\cap (\QFhYhg\,\wt{\otimes}\, \Yhgstar)
\quad \text{ and }\quad R^\chi\in (Y_\hbar^\chi\mfg\,\wt{\otimes}\,Y_\hbar^\chi\mfg{\vphantom{)}}^\star)\cap (\dhYh{\chi}\,\wt{\otimes}\, \YTDstar{\chi}).
\end{equation*}
It is worth pointing out that one can immediately deduce a number of properties that $R$ and $R^\chi$ satisfy using only elementary properties of $\Theta$. For instance:
\begin{enumerate}[label=(R\arabic*)]
\item\label{R1} Applying \eqref{Theta-prop} with $\gamma=\id$ while using that $\Theta(1\otimes f)=\iota\circ f$ recovers the characteristic identities
\begin{equation*}
(\iota\circ f)\otimes \id \cdot R=1\otimes f\quad \text{ and }\quad (\iota\circ f^\chi)\otimes \id \cdot R^\chi=1\otimes f^\chi
\end{equation*}
for all $f\in \Yhgstar$ and $f^\chi\in \YTDstar{\chi}$. 
 
 \item \label{R2} Taking $\varphi=\gamma=\omega$ in \eqref{Theta-prop} with $y=\id_{\QFhYhg}$ or $\id_{\dhYh{\mp}}$ while using that $\id_{\dhYh{\pm}}=\omega \circ \id_{\dhYh{\mp}}\circ \omega$ gives rise to the identities
 \begin{equation*}
 R=(\omega\otimes \omega^t)R \quad\text{ and }\quad R^\pm=(\omega\otimes \omega^t)(R^\mp).
 \end{equation*}
\item Since $\id_{\QFhYhg}$ is a $\mfg$-invariant element of $\End_{\C[\![\hbar]\!]}(\QFhYhg)$ and $\Theta$ is a $\mfg$-module intertwiner, $R$ is a $\mfg$-invariant element $\QFhYhg\,\wt{\otimes}\, \Yhgstar$. Using the multiplication in $D(\hYhg)$, this can be written as
\begin{equation*}
[x\otimes 1 + 1\otimes x,R]=0 \quad \forall\; x\in \mfg. 
\end{equation*}
 Similarly, $R^\pm$ and $R^0$ are both $\mfh$-invariant.
%
\end{enumerate}
In addition, the standard quantum double arguments show that $R$ satisfies the quasitriangularity relations of \eqref{DH:3} in Section \ref{ssec:D(H)} in suitable completions of tensor powers of $D(\hYhg)$. For instance, the cabling identities $\Delta\otimes \id(R)=R_{13}R_{23}$ and $\id\otimes \chk{\Delta}(R)=R_{13}R_{12}$ are equivalent to the simple identities 
\begin{equation*}
\id^{(1)}_{\QFhYhg}\star \id^{(2)}_{\QFhYhg}=\Delta \quad \text{ and }\quad (\veps\otimes \id_{\QFhYhg}) \star (\id_{\QFhYhg} \otimes \veps)= (1\,2)\circ m
\end{equation*}
 in the convolution algebras 
\begin{gather*} 
 (\QFhYhg^{\otimes 2})\,\wt{\otimes}\,\Yhgstar:=\Hom_{\C[\![\hbar]\!]}(\QFhYhg,\QFhYhg^{\otimes 2})\\
 \QFhYhg\,\wt{\otimes}\, (\Yhgstar)^{\otimes 2}:=\Hom_{\C[\![\hbar]\!]}(\QFhYhg^{\otimes 2},\QFhYhg)
\end{gather*}
  respectively, where $\id^{(a)}_{\QFhYhg}:\QFhYhg\to\QFhYhg^{\otimes 2}$ is given by $x\mapsto x^{(a)}$.

Our main goal in this section  is to establish the remaining assertion of Theorem \ref{T:Intro} from Section \ref{ssec:I-results}, which claims that $R$ can be identified with the universal $R$-matrix of the Yangian. This interpretation allows for a precise framework for understanding the topological quasitriangular structure on $D(\hYhg)$ alluded to above. To achieve this, we will need the next lemma, where $\chk{\Phi_z}$ is as in Proposition \ref{P:dual-z}. 
\begin{lemma}\label{L:1xPhiz}
$\tau_w \otimes \chk{\Phi_z}$ extends to an injective $\C[\![\hbar]\!]$-algebra homomorphism 
 \begin{equation*}
 \tau_w\,\wt{\otimes}\, \chk{\Phi_z}: \hYhg \, \wt{\otimes}\, \chk{\hYhg}\into (\hYhg^{\otimes 2}[\![w]\!])[\![z^{-1}]\!].
 \end{equation*}
 \end{lemma}
\begin{proof}
It is sufficient to show $\id\otimes \chk{\Phi_z}$ extends to an injective $\C[\![\hbar]\!]$-algebra homomorphism 
\begin{equation*}
\id \, \wt{\otimes}\, \chk{\Phi_z}: \hYhg \,\wt{\otimes}\, \chk{\hYhg}\to \hYhg^{\otimes 2}[\![z^{-1}]\!].
\end{equation*}
Since $\chk{\Phi_z}(\Yhgstar_{-k})\subset z^{-k}\Yhgz\subset z^{-k}\hYhg[\![z^{-1}]\!]$, $\id \otimes \chk{\Phi_z}$ induces a family of compatible algebra homomorphisms 
\begin{equation*}
(\id \otimes \chk{\Phi_z})_n: \hYhg \otimes \Yhgstar/ \QFhYhg \otimes \Yhgstar_{\scriptscriptstyle (n)} \to \hYhg^{\otimes 2}[\![z^{-1}]\!]/ z^{-n}\hYhg^{\otimes 2}[\![z^{-1}]\!]. 
\end{equation*}
Taking the projective limit gives the desired extension $\id \, \wt{\otimes}\, \chk{\Phi_z}$. We are left to verify that it is indeed injective, which is perhaps a bit subtle (given that $\chk{\Phi_z}$ is injective), but not difficult. As $\hYhg\, \wt{\otimes}\, \Yhgstar$ is separated with respect to the $\hbar$-adic topology and $\hYhg^{\otimes 2}[\![z^{-1}]\!]$ is torsion free, this can be done by verifying that the semiclassical limit  of $\id \, \wt{\otimes}\, \chk{\Phi_z}$ is injective. This is the linear map 
\begin{equation*}
 U(\tplus)\,\wt{\otimes}\, \Symt{\vphantom{)}}^\star\to \prod_{n\in \N}z^{-n}(U(\tplus)\otimes_\C \widehat{U(\tplus)}_z)\subset U(\tplus)^{\otimes 2}[\![z^{-1}]\!]
\end{equation*}
which is  $\id \otimes \chk{\bar{\Phi}_z}:U(\tplus)\,\wt{\otimes}\, \Symt{\vphantom{)}}^\star_{-n}\to z^{-n}(U(\tplus)\otimes_\C \widehat{U(\tplus)}_z)$ on the $n$-th component of the direct product, 
where $\chk{\bar{\Phi}_z}$ is the semiclassical limit of $\chk{\Phi_z}$ and $\widehat{U(\tplus)}_z=\prod_{n\in \N}U(\tplus)_n z^{-n}$. The desired result now follows from the fact that  $\chk{\bar{\Phi}_z}$ is injective, wich can be seen as a consequence of the equality $\mathscr{Ev}\circ \chk{\Phi_z}= \Phi \circ \check{\imath}$ and that, by \cite{WDYhg}*{Thm.~6.2} and the proof of Lemma \ref{L:Yhg*->DYhg}, both $\Phi$ and $\check{\imath}$ have injective semiclassical limits. \qedhere
\end{proof}
 
 To recover the desired result as stated in \eqref{Intro:2} of Theorem \ref{T:Intro}, we now translate some of the above constructions and results from  $D(\hYhg)$ to $\DYhg$ using the isomorphism $\Upsilon$ from Theorem \ref{T:DYhg-DD} or, equivalently, the Hopf algebra embeddings $\imath$ and $\check{\imath}$ from Theorem \ref{T:dual}. 
Let us set 
\begin{equation*}
\hYhg\,\dot\otimes\,\chk{\hYhg}:=\varprojlim_n \left(\imath(\hYhg) \otimes \check{\imath}(\chk{\hYhg}) / \imath(\hYhg) \otimes \check{\imath}(\chk{\hYhg}_{\scriptscriptstyle{(n)}})\right).
\end{equation*}  
This definition is such that isomorphism 
\begin{equation*}
 \imath\otimes \check{\imath}: \hYhg\otimes \chk{\hYhg}\iso \imath(\hYhg)\otimes \check{\imath}(\chk{\hYhg})\subset \DYhg\otimes \DYhg
\end{equation*}
extends to an isomorphism of $\C[\![\hbar]\!]$-algebras 
\begin{equation*}
\imath \,\dot{\otimes}\, \check{\imath}: \hYhg \,\wt{\otimes}\, \chk{\hYhg} \iso \hYhg\,\dot{\otimes}\, \chk{\hYhg}. 
\end{equation*}
Armed with these preliminaries, we may now introduce the \textit{universal $R$-matrix} of the Yangian double $\DYhg$ as the element
\begin{equation*}
\boldR:=(\imath \,\dot{\otimes}\, \check{\imath})(R)\in (\hYhg\,\dot{\otimes}\,\chk{\hYhg})\cap (\QFhYhg\,\dot{\otimes}\, \chk{\hYhg})\subset \hYhg\,\dot{\otimes}\, \chk{\hYhg},
\end{equation*}
where, for each choice of $A$ and $B$, $A\dot{\otimes} B$ denotes the image of $A\wt{\otimes}B$ under $\imath\,\dot{\otimes}\,\check{\imath}$. Similarly, we define the $\DYhg$ analogues of $R^\pm$ and $R^0$ by setting
\begin{equation*}
\boldR^\pm:=(\imath \,\dot{\otimes}\, \check{\imath})(R^\pm) \in  Y_\hbar^{\pm\!}\mfg\,\dot{\otimes}\, \YTDstar{\pm\!} 
\quad \text{ and } \quad 
\boldR^0:=(\imath \,\dot{\otimes}\, \check{\imath})(R^0) \in  Y_\hbar^{0}\mfg\,\dot{\otimes}\, \YTDstar{0}.
\end{equation*}

By Theorem \ref{T:Phiz} and  Lemma \ref{L:Yhg*->DYhg}, $\Phi_z$ satisfies  $\Phi_z\circ \imath=\tau_z$ and $\Phi_z\circ \check{\imath}=\chk{\Phi_z}$. Lemma \ref{L:1xPhiz} therefore implies that the restriction of $\Phi_w\otimes\Phi_z$ to $\imath(\hYhg) \otimes \check{\imath}(\chk{\hYhg}) \subset \DYhg^{\otimes 2}$  extends to an injective $\C[\![\hbar]\!]$-algebra homomorphism 
\begin{equation*}
\Phi_w \, \dot{\otimes}\, \Phi_z: \hYhg\,\dot{\otimes}\, \chk{\hYhg}\into \hYhg^{\otimes 2}[\![w]\!][\![z^{-1}]\!].
\end{equation*}
Indeed, we may set $\Phi_w \, \dot{\otimes}\, \Phi_z := (\tau_w \,\wt{\otimes}\, \chk{\Phi_z})\circ (\imath \,\dot{\otimes}\, \check{\imath})^{-1}$. With this homomorphism at our disposal, we are now in a position to state and prove the last main result outlined in Section \ref{ssec:I-results}. 
\begin{theorem}\label{T:R}
The following identities hold in $\Yhg^{\otimes 2}[w][\![z^{-1}]\!]$:
\begin{gather*}
(\Phi_w\,\dot\otimes\, \Phi_z) (\boldR)=\mcR(w-z),\\
(\Phi_w\,\dot\otimes\, \Phi_z) (\boldR^\pm)=\mcR^\pm(w-z),\quad (\Phi_w\,\dot\otimes\, \Phi_z) (\boldR^0)=\mcR^0(w-z). 
\end{gather*}
Consequently $\boldR$ admits the factorization $\boldR=\boldR^+\!\cdot\! \boldR^0 \!\cdot\! \boldR^-$ in $\hYhg\,\dot{\otimes}\, \chk{\hYhg}$.
\end{theorem}
\begin{proof}
By Theorems 4.1 and 6.7 of \cite{GTLW19} and Theorem \ref{T:Yhg-R} above, $\mcR(z)$ and $\mcR^\chi(z)$ satisfy
$
(\tau_w\otimes \id)\mcR(z)=\mcR(w+z)$ and $(\tau_w\otimes \id )\mcR^\chi(z)=\mcR^\chi(w+z)
$
in $\Yhg^{\otimes 2}[w][\![z^{-1}]\!]$, for each choice of the symbol $\chi$. Therefore, to prove the first assertion of the theorem it is sufficient to establish the equalities 
\begin{equation*}
(\id \, \wt{\otimes}\, \chk{\Phi_z}) (R)=\mcR(-z) \quad \text{ and }\quad (\id \, \wt{\otimes}\, \chk{\Phi_z}) (R^\chi)=\mcR^\chi(-z)
\end{equation*}
in $(\QFhYhg\otimes \hYhg)[\![z^{-1}]\!]$ and $(\dhYh{\chi}\otimes Y_\hbar^{{\scriptscriptstyle{\shortminus}}\chi}\mfg)[\![z^{-1}]\!]$, respectively. These relations will hold provided their images under the evaluations 
\begin{gather*}
f\otimes \id: (\QFhYhg\otimes \hYhg)[\![z^{-1}]\!]\to \hYhg[\![z^{-1}]\!]\\
f^\chi\otimes \id: (\dhYh{\chi}\otimes Y_\hbar^{{\scriptscriptstyle{\shortminus}}\chi}\mfg)[\![z^{-1}]\!]\to Y_\hbar^{{\scriptscriptstyle{\shortminus}}\chi}\mfg[\![z^{-1}]\!]
\end{gather*}
are satisfied in $\hYhg[\![z^{-1}]\!]$ for each $f\in \Yhgstar$ and $f^\chi\in \YTDstar{\chi}$, respectively. This follows from the definition of $\chk{\Phi_z}$, \ref{R1} and Corollary \ref{C:Yhg*-TD-2}, which collectively output the relations
\begin{gather*}
(f\otimes \chk{\Phi_z}) (R)=\chk{\Phi_z}(f)=(f\otimes \id)\mcR(-z),\\
 (f^\chi\otimes \chk{\Phi_z}) (R)=\chk{\Phi_z}(f)=(f^\chi\otimes \id)\mcR^\chi(-z). 
\end{gather*}
This completes the proof of the first statement of the Theorem. The second assertion
now follows immediately from the decomposition $\mcR(z)=\mcR^+(z)\mcR^0(z)\mcR^-(z)$ (see Theorem \ref{T:Yhg-R}) and the injectivity of $\Phi_w\,\dot\otimes\, \chk{\Phi_z}$. \qedhere
\end{proof}
\begin{remark}
In the closely related setting of \cite{KT96}, a  formal version (\textit{i.e.,} suppressing convergence issues) of the factorization $\boldR=\boldR^+\!\cdot\! \boldR^0 \!\cdot\! \boldR^-$ from Theorem \ref{T:R} was established in \cite{KT96}*{Prop. 5.1} by proving directly that the Hopf pairing between the Yangian and its dual splits with respect to the underlying triangular decompositions; see Theorem 3.1 therein.   
\end{remark}

\begin{remark}\label{R:chev}
Since $\Phi_z$ intertwines the Chevalley involutions of $\DYhg$ and $\hYhg$, it follows from Part \eqref{dual-z:2} of Proposition \ref{P:dual-z} and the relation $\Phi_z\circ \check{\imath}=\chk{\Phi_z}$ of Lemma \ref{L:Yhg*->DYhg}  that  $\check{\imath}$ satisfies $\omega\circ \check{\imath}=\check{\imath}\circ \omega^t$. Consequently, the property \ref{R2} translates to 
 \begin{equation*}
 \boldR=(\omega\otimes \omega)(\boldR) \quad\text{ and }\quad \boldR^\pm=(\omega\otimes \omega)(\boldR^\mp).
 \end{equation*}
Hence, by Theorem \ref{T:R}, $\mcR(z)$ and its components $\mcR^\pm(z)$ satisfy the relations
\begin{equation*}
 \mcR(z)=(\omega\otimes \omega)(\mcR(z)) \quad\text{ and }\quad \mcR^\pm(z)=(\omega\otimes \omega)(\mcR^\mp(z)).
\end{equation*}
In particular, this observation recovers the first identity of Corollary \ref{C:Yhg-R-Chev}.
\end{remark}
One  can now interpret the quasitriangularity relations for $\DYhg$ in terms of the relations of Theorem \ref{T:Yhg-R}. Namely, setting $v=w-z$, we see that the cabling identities of $\boldR$ correspond to the relations 
\begin{align*}
\Delta\otimes \id (\mcR(v))&= \mcR_{13}(v)\mcR_{23}(v)\\
\id\otimes \Delta (\mcR(v))&= \mcR_{13}(v)\mcR_{12}(v)
\end{align*}
in $\Yhg^{\otimes 3}[w][\![z^{-1}]\!]$, which are satisfied by Theorem \ref{T:Yhg-R}. In more detail, these are obtained heuristically by applying $\Phi_u\otimes \Phi_w\otimes \Phi_z$ and $\Phi_w\otimes \Phi_y\otimes \Phi_z$ to 
\begin{equation*}
(\dot\Delta\otimes \id)(\boldR)=\boldR_{13}\boldR_{23} \quad \text{ and }\quad (\id\otimes \dot\Delta)(\boldR)=\boldR_{13}\boldR_{12},
\end{equation*}
respectively, where $\dot\Delta$ is the coproduct on $\DYhg$, and then evaluating $u\mapsto w$ and $y\mapsto z$ while using $(\Phi_u\otimes \Phi_z) \circ \dot\Delta |_{u=z} = \dot \Delta\circ \Phi_z $.  Similarly, applying $\Phi_w\otimes \Phi_z$ to the intertwiner equation $\op{\Delta}(x)=\boldR\, \Delta(x) \boldR^{-1}$ for $x\in \Yhg$ leads to the identity 
\begin{equation*}
\tau_v\otimes \id \circ \op{\Delta}(\tau_z(x))= \mcR(v) \cdot \tau_v\otimes \id \circ \Delta(\tau_z(x)) \cdot \mcR(v)^{-1}
\end{equation*}
in $\Yhg^{\otimes 2}[w][z;z^{-1}]\!]$, which is satisfied by Theorem \ref{T:Yhg-R}. We caution, however, that the situation here is more subtle for general $x\in \DYhg$.  
%
%

%
\subsection{Remarks}\label{ssec:R-compute}
The identifications established in Theorem \ref{T:R} provide a rigorous framework for understanding the motivating remarks given in \cite{GTLW19}*{\S1.6}. We now give a few comments related to this point.

First, we note that the diagonal factor $\mcR^0(z)$ of $\mcR(z)$ was explicitly obtained in \cite{GTLW19}*{\S6.6} by computing the common asymptotic expansion of the two $\mathrm{GL}(V_1\otimes V_2)$-valued meromorphic abelian $R$-matrices constructed in \cite{GTL3}*{\S5}, where $V_1,V_2$ are an arbitrary pair of finite-dimensional representations of the Yangian. As explained in \cite{GTL3}*{\S5.2}, their construction was motivated by the heuristic formula for $\boldR^0$ given in \cite{KT96}*{Thm.~5.2}. Theorem \ref{T:R} makes this relation precise, and shows that the explicit formula for $\mcR^0(z)$ obtained in \cite{GTLW19}, and recalled in Section \ref{ssec:Yhg-R}, does indeed compute the canonical element defined by the pairing on $\dhYh{0}\times \YTDstar{0}$.

Next, we emphasize that the factor $\boldR^\pm$ (equivalently, $\mcR^\pm(z)$) now admits two distinct characterizations. On the one hand, it is uniquely determined by the recurrence relations \eqref{R-recur} which were at the heart of \cite{GTLW19}*{Thm.~4.1}. On the other hand, it arises as the canonical element defined by the pairing on $\dhYh{\pm}\times \YTDstar{\pm}$, and can thus be realized explicitly by computing the dual set  to any fixed homogeneous basis of the $\N$-graded torsion free $\C[\hbar]$-module $\dYh{\pm}$. In more detail, if $\mathsf{B}^\pm_k\subset \dYh{\pm}_k$ is a lift of any basis of the finite-dimensional $k$-th component $\msS_k(\hbar\tplus^\pm)$ of $\msS(\hbar\tplus^\pm)\cong \dYh{\pm}/\hbar\dYh{\pm}$ (see Section \ref{ssec:QFYhg-TD}), then $\mathsf{B}^\pm=\bigcup_k\mathsf{B}^\pm_k$ is a basis of $\QFYhg$ and 
\begin{equation*}
\boldR^\pm=\sum_{x\in \mathsf{B}^\pm} \imath(x)\otimes \check{\imath}(f_x) \in \hYhg\,\dot{\otimes}\chk{\hYhg}, 
\end{equation*}
where $\{f_x\}_{x\in \msB^\pm}\subset \YTDstar{\pm}$ is the dual set to $\msB^\pm$, uniquely determined by $f_x(y)=\delta_{xy}$ for all $x,y \in \msB^\pm$. Here we note if $x$ is of degree $k$, then it follows automatically that $f_x\in \Yhgstar_{-k}\cong \Hom_{\C[\hbar]}^{-k}(\QFYhg,\C[\hbar])$. This implies that the right-hand side of the above expression defines a unique element in $\hYhg\,\dot{\otimes}\chk{\hYhg}$, which coincides with $\boldR^\pm$ as its image under $\Theta_\pm \circ (\imath \,\dot{\otimes}\,\check{\imath})^{-1}$ is $\id_{\dhYh{\pm}}$; see \eqref{Theta^chi} and Section \ref{ssec:R-can}.

In the special case where $\mfg=\mfsl_2$, there are two closed-form expressions for $\boldR^\pm$ (equivalently, $\mcR^\pm(z)$) which have arisen from these two separate viewpoints; see \cite{KT96}*{Thm.~5.1} and \cite{GTLW19}*{Thm.~5.5}. For $\mfg$ of arbitrary rank, no such expressions are known, though an infinite-product formula for $\boldR^\pm$ was conjectured in \cite{KT96}*{(5.43)}, motivated by the earlier works \cites{KT91,KT93,KT94,KST95}. 

%
\subsection{On the blocks of \texorpdfstring{$\boldR^\pm$}{R+\textbackslash-}}\label{ssec:Rminus-beta}
To conclude, we wish to highlight that the dual bases approach discussed in the previous subsection provides a natural interpretation of some of the basic properties of $\mcR^\pm(z)$ discovered in \cite{GTLW19}.

Given $\beta\in Q$, let $\pi_\beta:\DYhg\to\DYhg_\beta$ denote the $\C[\![\hbar]\!]$-linear projection associated to the topological $Q$-grading on $\DYhg$ defined in Section \ref{ssec:DYhg-root}, and let 
$\dot\pi_\beta:\QFhYhg\onto \QFhYhg_\beta$ denote its restriction to $\QFhYhg$. Consider the element 
\begin{equation*}
\dot\pi_\beta\circ \id_{\dhYh{\pm}}=\id_{\dhYh{\pm}}\circ \dot\pi_\beta \in \End_{\C[\![\hbar]\!]}(\QFhYhg),
\end{equation*}
where we view $\End_{\C[\![\hbar]\!]}(\dhYh{\pm})\subset \End_{\C[\![\hbar]\!]}(\QFhYhg)$, as in Section \ref{ssec:Theta-chi}. Note that under the identification provided by the natural inclusion
\begin{equation*}
\End_{\C[\![\hbar]\!]}(\dhYh{\pm}_\beta)\into \Hom_{\C[\![\hbar]\!]}(\dhYh{\pm},Y_\hbar^{\pm\!}\mfg), \quad \varphi\mapsto \varphi\circ \dot\pi_\beta, 
\end{equation*}
it coincides with the identity transformation of $\End_{\C[\![\hbar]\!]}(\dhYh{\pm}_\beta)$. For each $\beta\in Q_+$, we may therefore define $R_\beta^\pm:=\Theta^{-1}(\id_\beta^\pm)$, where $\id_\beta^\pm=\id_{\dhYh{\pm}_{\pm \beta}}$.  By \eqref{Theta-prop} and the reasoning used in \ref{R2} of Section \ref{ssec:R-can}, we have 
\begin{equation}\label{R_beta-props}
R_\beta^\pm=(\dot\pi_{\pm \beta}\otimes \id)(R^\pm)=(\id \otimes \dot{\pi}_{\pm \beta}^t)(R^\pm) \quad \text{ and }\quad R_\beta^\pm=(\omega\otimes \omega^t)(R_\beta^\mp),
\end{equation}
where we note that, for each $\alpha\in Q$, $\dot{\pi}_\alpha^t$ is just the projection $\Yhgstar\onto \Yhgstar_{-\alpha}$ associated to the topological $Q$-grading on $\Yhgstar$; see Corollary \ref{C:Yhg*-Qgrad}. In the topological tensor product $\dhYh{\pm}\,\wt{\otimes}\, \YTDstar{\pm}$, we have the equality $R^\pm=\sum_{\beta\in Q_+}R_\beta^\pm$.

For the sake of the below discussion, it is worth pointing out that the convergence of the infinite sum  $\sum_{\beta\in Q_+}R_\beta^\pm$ is also clear from the point of view of dual bases. Indeed, by Corollary \ref{C:QFYhg-chi} we have $\dYh{\pm}_{\pm\beta}\subset \hbar^{\nu(\beta)}Y_\hbar^\pm(\mfg)_{\pm\beta}$, where we recall that $\nu(\beta)\in \N$ is defined by \eqref{nu-beta}. Hence, any homogeneous element in $\dYh{\pm}_{\pm\beta}$ with respect to the underlying $\N$-grading belongs to $\QFYhg_{\nu(\beta)+\ell}$ for some $\ell\in \N$. It follows that the dual set to any homogeneous basis  of the $\N$-graded $\C[\hbar]$-module $\dYh{\pm}_{\pm\beta}$
belongs to  
\begin{equation*}
\bigoplus_{k\geq\nu(\beta)} \YTDstar{\pm}_{-k} \subset \YTDstar{\pm}_{\scriptscriptstyle(\nu(\beta))},
\end{equation*}
where the space on the right-hand side is the closure of the $\C[\![\hbar]\!]$-module generated by the left-hand side; see Section \ref{ssec:Theta-chi}. 
This implies the convergence of $\sum_{\beta\in Q_+}R_\beta^\pm$ in $\dhYh{\pm}\,\wt{\otimes}\, \YTDstar{\pm}$ while establishing that $R_\beta^\pm\in \hbar^{\nu(\beta)} Y_\hbar^{\pm\!}\mfg\, \wt{\otimes}\, \YTDstar{\pm}$. Next, following the procedure from Section \ref{ssec:R-can}, let us set
\begin{equation*}
\boldR_\beta^\pm:= (\imath\,\dot{\otimes}\,\check{\imath})(R_\beta^\pm) \quad \forall\; \beta \in Q_+.
\end{equation*}
The relations of \eqref{R_beta-props} then translate to 
\begin{equation*}
\boldR_\beta^\pm=(\pi_{\pm \beta}\otimes \id)(\boldR^\pm)=(\id \otimes \pi_{\mp \beta})(\boldR^\pm) \quad \text{ and }\quad \boldR_\beta^\pm=(\omega\otimes \omega)(\boldR_\beta^\mp),
\end{equation*}
where $\omega$ is the Chevalley involution on $\DYhg$; see Section \ref{ssec:DYhg-root}. Now recall that $\mcR_{\beta}^-(z)$ is the $\Yhg_{-\beta}\otimes \Yhg_\beta$ component of $\mcR^-(z)$, characterized by the recurrence relation \eqref{R-recur}. The $\Yhg_{\beta}\otimes \Yhg_{-\beta}$ block of $\mcR^+(z)$ is then 
\begin{equation*}
\mcR^+_\beta(z)=(\omega\otimes \omega)\mcR^-_\beta(z) \quad \forall\; \beta\in Q_+
\end{equation*}
and we have the following corollary of Theorem \ref{T:R}. 
\begin{corollary}\label{C:R-beta}
For each $\beta\in Q_+$, the element $\boldR_\beta^\pm$ satisfies
\begin{equation*}
(\Phi_w\,\dot{\otimes}\,\Phi_z)(\boldR_\beta^\pm)\in (\hbar/z)^{\nu(\beta)}\Yhg^{\otimes 2}[w][\![z^{-1}]\!]
\end{equation*}
in addition to the relation $(\Phi_w\,\dot{\otimes}\,\Phi_z)(\boldR_\beta^\pm)=\mcR^\pm_\beta(w-z)$.
\end{corollary}
Of course, the second assertion is immediate from Theorem \ref{T:R} as $\Phi_w$ is the identity on $\mfg$ and thus satisfies $\pi_\alpha\circ \Phi_w = \Phi_w\circ {\pi_\alpha}|_{\hYhg}$ for all $\alpha\in Q$. The first statement then follows from the properties of $\mcR^-(z)$ established in \cite{GTLW19} and recalled in Section \ref{ssec:Yhg-R}. However, we wish to point out that this assertion is a natural consequence of the above discussion on dual bases. Indeed, we have seen that any homogeneous basis of $\dYh{\pm}_\beta$ lies in $\hbar^{\nu(\beta)}\Yhg$, and that the image of its dual set under $\check{\imath}$  is contained in $\bigoplus_{k\geq \nu(\beta)}\DYhg_{-k}$. As $\Phi_z$ is graded, we have
\begin{equation*}
\Phi_z(\DYhg_{-k})\subset z^{-k}\Yhgz\subset z^{-\nu(\beta)}\Yhg[\![z^{-1}]\!] \quad \forall \; k\geq \nu(\beta),
\end{equation*}
which yields the first statement of the corollary. 
%
\appendix 

\section{Homogenization of the \texorpdfstring{$R$}{R}-matrix}\label{A:R-matrix}

In this appendix, we show that the results of \cite{GTLW19}*{\S7.4} imply Theorem \ref{T:Yhg-R}, as promised in Remark \ref{R:strict}. Let $Y(\mfg)$ denote the Yangian defined over $\C$ with $\hbar$ specialized to $1$: 
\begin{equation*}
Y(\mfg):=\Yhg/(\hbar-1)\Yhg.
\end{equation*}
Slightly abusing notation, we shall denote the images of $x_{ir}^\pm$ and $h_{ir}$ again by $x_{ir}^\pm$ and $h_{ir}$, respectively.
The graded Hopf algebra structure on $\Yhg$ induces on $Y(\mfg)$ the structure of an $\N$-filtered Hopf algebra over the complex numbers with filtration $\mathds{F}_\bullet$ defined by letting $\mathds{F}_k$ denote the image of $\bigoplus_{n\leq k}Y_\hbar(\mfg)_k$. The Yangian $Y(\mfg)$ is then a filtered deformation of the graded Hopf algebra $U(\tplus)$: 
\begin{equation*}
\mathrm{gr}_{\mathds{F}}Y(\mfg)\cong U(\tplus). 
\end{equation*}

One can recover $\Yhg$ from $Y(\mfg)$ using the Rees algebra formalism; see \cite{GRWEquiv}*{Prop.~2.2} and \cite{GRWvrep}*{Thm.~6.10}, for example. In more detail, there is an isomorphism of $\N$-graded Hopf algebras
\begin{gather*}
\upvarphi_\hbar:\Yhg\iso \msR_\hbar(Y(\mfg))=\bigoplus_{k\in \N} \hbar^k \mathds{F}_k(Y(\mfg))\subset Y(\mfg)[\hbar]\\
x_{ir}^\pm \mapsto \hbar^r x_{ir}^\pm,\quad h_{ir}\mapsto \hbar^r h_{ir} \quad \forall\; i\in \mbI, \, r\in \N. 
\end{gather*}

Here the Hopf algebra structure on $\mathsf{R}_\hbar(Y(\mfg))$ is obtained by extending that of $Y(\mfg)$ by $\C[\hbar]$-linearity, and we note that $\mathsf{R}_\hbar(Y(\mfg)\otimes_\C Y(\mfg))\cong \mathsf{R}_\hbar(Y(\mfg))\otimes_{\C[\hbar]} \mathsf{R}_\hbar(Y(\mfg))$. 

In \cite{GTLW19}, the universal $R$-matrix $\mathds{R}(z)$ of the Yangian $Y(\mfg)$ is constructed as a product 
\begin{equation*}
\mathds{R}(z)=\mathds{R}^+(z)\,\mathds{R}^0(z)\,\mathds{R}^-(z)\in 1+z^{-1}Y(\mfg)^{\otimes 2}[\![z^{-1}]\!],
\end{equation*}
where $\mathds{R}^+(z)=\mathds{R}^-_{21}(-z)^{-1}$ and the factors $\mathds{R}^-(z)$ and $\mathds{R}^0(z)$ are as in 
 Sections 4.1 and 6.6 of \cite{GTLW19}, respectively. 
In particular, $\mathds{R}^\pm(z)-1$, $\mathds{R}^0(z)-1$ and $\mathds{R}(z)-1$ lay in the subspace
\begin{equation*}
z^{-1} \prod_{n\in \N} \mathds{F}_n( Y(\mfg)^{\otimes 2})z^{-n}\subset Y(\mfg)^{\otimes 2}[\![z^{-1}]\!]
\end{equation*}
and hence $\mathds{R}^{\pm}(z/\hbar)$, $\mathds{R}^0(z/\hbar)$ and $\mathds{R}(z/\hbar)$ are elements of 
$\msR_\hbar(Y(\mfg))^{\otimes 2}[\![z^{-1}]\!]$. The definitions of $\mcR^\pm(z)$, $\mcR^0(z)$ and $\mcR(z)$ given in Section \ref{ssec:Yhg-R} are such that one has the equalities
\begin{equation}\label{R-homog}
\begin{gathered}
\upvarphi_\hbar^{\otimes 2}(\mcR^\pm(z))=\mathds{R}^{\pm}(z/\hbar), \quad \upvarphi_\hbar^{\otimes 2}(\mcR^0(z))=\mathds{R}^0(z/\hbar),\\ \upvarphi_\hbar^{\otimes 2}(\mcR(z))=\mathds{R}(z/\hbar).
\end{gathered}
\end{equation}
Using this fact and the results of \cite{GTLW19}, we can recover the below proposition, which is a restatement of Theorem \ref{T:Yhg-R}. 
\begin{proposition}\label{P:Yhg-R}
$\mcR(z)$ is the unique formal series in
$
1+z^{-1}\Yhg^{\otimes 2}[\![z^{-1}]\!]
$
satisfying the intertwiner equation 
\begin{equation*}
\tau_z\otimes \id \circ \op{\Delta}(x)= \mcR(z) \cdot \tau_z\otimes \id \circ \Delta(x) \cdot \mcR(z)^{-1} \quad \forall\; x\in \Yhg
\end{equation*}
in $\Yhg^{\otimes 2}[z;z^{-1}]\!]$, in addition to the cabling identities 
\begin{align*}
\Delta\otimes \id (\mcR(z))&= \mcR_{13}(z)\mcR_{23}(z)\\
\id\otimes \Delta (\mcR(z))&= \mcR_{13}(z)\mcR_{12}(z)
\end{align*}
in $\Yhg^{\otimes 3}[\![z^{-1}]\!]$. Moreover, $\mcR(z)$ satisfies the properties \eqref{Yhg-R:1}--\eqref{Yhg-R:3} of Theorem \ref{T:Yhg-R}. 
\end{proposition}
\begin{proof} That $\mcR(z)$ satisfies the properties \eqref{Yhg-R:1}--\eqref{Yhg-R:3} of Theorem \ref{T:Yhg-R} follows from \eqref{R-homog} and the corresponding properties of $\mathds{R}(z)$ established in (3)--(5) of \cite{GTLW19}*{Thm.~7.4}.
 Similarly, that $\mcR(z)$ satisfies the cabling identities in $\Yhg^{\otimes 3}[\![z^{-1}]\!]$ follows from the equality $\upvarphi_\hbar^{\otimes 2}(\mcR(z))=\mathds{R}(z/\hbar)$ and Theorem 7.4 (2) of \cite{GTLW19}, which asserts that $\mathds{R}(z)$ satisfies the cabling identities in $Y(\mfg)^{\otimes 3}[\![z^{-1}]\!]$. As for the intertwiner equation, upon applying the isomorphism $\upvarphi_\hbar^{\otimes 2}$ we deduce that it will hold provided $\mathds{R}(z)$ satisfies 
\begin{equation*}
\tau_z^\upvarphi\otimes \id \circ \op{\Delta}_{Y(\mfg)}(x)= \mathds{R}(z/\hbar) \cdot \tau_z^\upvarphi\otimes \id \circ \Delta_{Y(\mfg)}(x) \cdot \mathds{R}(z/\hbar)^{-1} \quad \forall\; x\in \mathsf{R}_\hbar(Y(\mfg))
\end{equation*}
 in $\mathsf{R}_\hbar(Y(\mfg))^{\otimes 2}[z;z^{-1}]\!]\subset (Y(\mfg)^{\otimes 2})[\hbar][z;z^{-1}]\!]$, where $\tau_z^\upvarphi=\upvarphi_\hbar \circ \tau_z \circ \upvarphi_\hbar^{-1}$. Since $\tau_z^\upvarphi$ is determined by 
 \begin{equation*}
 \tau_z^\upvarphi(x_i^\pm(u/\hbar))=x_i^\pm((u-z)/\hbar),\quad  \tau_z^\upvarphi(h_i(u/\hbar))=h_i((u-z)/\hbar) \quad \forall\; i\in \mbI
 \end{equation*}
 the above equality will hold provided that, for each $x\in Y(\mfg)$, and $\zeta\in \C^\times$, one has 
 \begin{equation}\label{inter-zeta}
\mathring{\tau}_{z/\zeta}\otimes \id \circ \op{\Delta}_{Y(\mfg)}(x)= \mathds{R}(z/\zeta) \cdot \mathring{\tau}_{z/\zeta}\otimes \id \circ \Delta_{Y(\mfg)}(x) \cdot \mathds{R}(z/\zeta)^{-1},
 \end{equation}
 where $\mathring{\tau}_z:Y(\mfg)\to Y(\mfg)[z]$ is obtained by specializing the algebra homomorphism $\tau_z$ defined in \eqref{shift-z}. 
 This equality is immediate from Part (1) of \cite{GTLW19}*{Thm.~7.4}. 
 
 As for the uniqueness assertion; the argument given in Appendix B of \cite{GTLW19} translates naturally to the formal setting. Alternatively, one can see this as consequence of the uniqueness of $\mathds{R}(z)$ itself, as proven therein. Indeed, if $\mathscr{R}(z)\in 1+z^{-1}\Yhg^{\otimes 2}[\![z^{-1}]\!]$ is another solution of the intertwiner equation satisfying the cabling identitites, then to see that $\mathscr{R}(z)=\mcR(z)$, it suffices to show that $\mathds{X}(z)|_{\hbar=\zeta}=\mathds{R}(z/\zeta)$ for each $\zeta\in \C^\times$, where 
 \begin{equation*}
 \mathds{X}(z):=\upvarphi^{\otimes 2}_\hbar(\mathscr{R}(z)) \in Y(\mfg)[\hbar][\![z^{-1}]\!]. 
 \end{equation*}
This follows from the fact that $\mathds{X}(z)|_{\hbar=\zeta}$ and $\mathds{R}(z/\zeta)$ both satisfy the cabling identities in $Y(\mfg)^{\otimes 3}[\![z^{-1}]\!]$ and the intertwiner equation \eqref{inter-zeta}, and so coincide by the uniqueness of $\mathds{R}(z)$, as established in Appendix B of \cite{GTLW19}. \qedhere
\end{proof}


\end{document}

%% file: Preamble.tex
\usepackage[lite]{amsrefs} 
\usepackage{amsmath,amssymb} 
\usepackage{mathtools}
\usepackage[usenames,dvipsnames]{xcolor} 
\usepackage{bbm} 
\usepackage[scr=boondoxo]{mathalfa} 
\usepackage{dsfont} 
\usepackage{enumitem}
\usepackage{version} 
\usepackage[all]{xy} 
\usepackage{tikz-cd} 
\usepackage{accents} 

\usepackage[
	colorlinks,
	plainpages,
	citecolor=red!50!black,
    filecolor=Darkgreen,
    linkcolor=blue!40!black,
    urlcolor=cyan!50!black!90]{hyperref} 
\newtheorem{theorem}{Theorem}[section]
\newtheorem{proposition}[theorem]{Proposition}
\newtheorem{corollary}[theorem]{Corollary}
\newtheorem{lemma}[theorem]{Lemma}

\newtheorem{theoremintro}{Theorem}

\theoremstyle{definition}
\newtheorem{definition}[theorem]{Definition}
\newtheorem{remark}[theorem]{Remark}

\newcommand{\wt}{\widetilde}
\newcommand{\wh}{\widehat}

\newcommand{\End}{\mathrm{End}}

\newcommand{\Hom}{\mathrm{Hom}}

\newcommand{\Ker}{\mathrm{Ker}}

\newcommand{\iso}{\xrightarrow{\,\smash{\raisebox{-0.5ex}{\ensuremath{\scriptstyle\sim}}}\,}}

\newcommand{\into}{\hookrightarrow}
\newcommand{\onto}{\twoheadrightarrow}

\makeatletter  
\newcommand{\sbullet}{%
  \hbox{\fontfamily{lmr}\fontsize{.4\dimexpr(\f@size pt)}{0}\selectfont\textbullet}}

\makeatother

\newcommand{\mfb}{\mathfrak{b}}

\newcommand{\mfg}{\mathfrak{g}}
\newcommand{\mfh}{\mathfrak{h}}

\newcommand{\mfm}{\mathfrak{m}}

\newcommand{\mfn}{\mathfrak{n}}

\newcommand{\mft}{\mathfrak{t}}

\newcommand{\mfsl}{\mathfrak{s}\mathfrak{l}}

\newcommand{\mcA}{\mathcal{A}}

\newcommand{\mcH}{\mathcal{H}}

\newcommand{\mcJ}{\mathcal{J}}

\newcommand{\mcL}{\mathcal{L}}

\newcommand{\mcR}{\mathcal{R}}
\newcommand{\mcS}{\mathcal{S}}

\newcommand{\mcX}{\mathcal{X}}

\newcommand{\mbA}{\mathbf{A}}

\newcommand{\mbB}{\mathbf{B}}

\newcommand{\mbC}{\mathbf{C}}

\newcommand{\mbF}{\mathbf{F}}

\newcommand{\mbI}{\mathbf{I}}

\newcommand{\mbJ}{\mathbf{J}}

\newcommand{\mbT}{\mathbf{T}}

\newcommand{\mbU}{\mathbf{U}}

\newcommand{\C}{\mathbb{C}}
\newcommand{\Q}{\mathbb{Q}}

\newcommand{\Z}{\mathbb{Z}}

\newcommand{\msA}{\mathsf{A}}

\newcommand{\msB}{\mathsf{B}}

\newcommand{\msD}{\mathsf{D}}

\newcommand{\msh}{\mathsf{h}}
\newcommand{\msH}{\mathsf{H}}

\newcommand{\msJ}{\mathsf{J}}

\newcommand{\msK}{\mathsf{K}}

\newcommand{\msm}{\mathsf{m}}
\newcommand{\msM}{\mathsf{M}}

\newcommand{\msN}{\mathsf{N}}

\newcommand{\msp}{\mathsf{p}}

\newcommand{\msq}{\mathsf{q}}

\newcommand{\msr}{\mathsf{r}}
\newcommand{\msR}{\mathsf{R}}

\newcommand{\msS}{\mathsf{S}}

\newcommand{\msV}{\mathsf{V}}

\newcommand{\msW}{\mathsf{W}}

\newcommand{\msX}{\mathsf{X}}

\newcommand{\msY}{\mathsf{Y}}

\newcommand{\veps}{\varepsilon}

\newenvironment{NB}{\noindent
\color{blue}{\bf Note}. \footnotesize
}{}

